\documentclass[11pt]{article}

\usepackage[T1]{fontenc}
\usepackage{amsmath,amssymb,amsthm,mathtools}
\usepackage{thmtools}
\usepackage[margin=1in]{geometry}
\usepackage{hyperref}
\usepackage{cleveref}
\usepackage[maxnames=99]{biblatex}
\usepackage{graphicx}
\usepackage{xcolor}
\usepackage{stmaryrd}

\usepackage{thm-restate}
\usepackage[blocks]{authblk}
\usepackage{xpatch}
\makeatletter
\xpatchcmd{\thmt@restatable}% 
{\csname #2\@xa\endcsname\ifx\@nx#1\@nx\else[{#1}]\fi}% 
{\ifthmt@thisistheone\csname #2\@xa\endcsname\ifx\@nx#1\@nx\else[{#1}]\fi\else\csname #2\@xa\endcsname\fi}% 
{}{} %
\makeatother

\newtheorem{theorem}{Theorem}[section]
\newtheorem{proposition}[theorem]{Proposition}

\newtheorem{corollary}[theorem]{Corollary}
\newtheorem{lemma}[theorem]{Lemma}

\theoremstyle{definition}
\newtheorem{definition}[theorem]{Definition}
\newtheorem{question}[theorem]{Question}
\theoremstyle{remark}
\newtheorem{remark}[theorem]{Remark}

\newtheorem{observation}[theorem]{Observation}

\newcommand\qto[1][\operatorname{q}]{\to_{#1}}
\newcommand\incspace{\ensuremath{\mathcal L}}
\newcommand\edgekernel{\ensuremath{\mathcal K}}
\newcommand\diagonalSubspace{\ensuremath{\mathcal D}}
\newcommand{\indfunction}[1]{\llbracket #1\rrbracket}

\newcommand{\C}{\mathbb C}
\newcommand\RE{\mathsf{RE}}
\newcommand{\R}{\mathbb R}
\newcommand{\1}{\mathbf 1}
\newcommand{\End}{\operatorname{End}}
\newcommand{\Proj}{\operatorname{Proj}}
\newcommand{\Span}{\operatorname{span}}
\newcommand{\rank}{\operatorname{rank}}
\newcommand{\Gr}{\operatorname{Gr}}
\newcommand{\Tr}{\operatorname{Tr}}

\newcommand{\qbinom}[2]{\genfrac{[}{]}{0pt}{}{#1}{#2}_{q}}
\DeclareMathOperator\qCSP{CSP_q}

\newcommand{\bzero}{\mathbf{0}}

\newcommand{\bb}{\mathbf{b}}

\newcommand{\bi}{\mathbf{i}}
\newcommand{\bj}{\mathbf{j}}

\newcommand{\bu}{\mathbf{u}}
\newcommand{\bx}{\mathbf{x}}
\newcommand{\bv}{\mathbf{v}}
\newcommand{\by}{\mathbf{y}}

\newcommand{\be}{\mathbf{e}}

\newcommand{\bz}{\mathbf{z}}
\newcommand{\N}{\mathbb{N}}
\newcommand{\F}{\mathbb{F}}

\newcommand\MIP{\ensuremath{\mathsf{MIP}^*}}

\newcommand{\vA}{x}
\newcommand{\vB}{y}
\newcommand{\vR}{\mathbf{R}}
\newcommand{\vY}{z}
\newcommand{\vZ}{w}
\newcommand{\KG}{\operatorname{KG}}

\title{Schrijver--Delsarte rigidity in association schemes and undecidability of quantum graph homomorphism}

\author{Lorenzo Ciardo\textsuperscript{\textdagger}, 
Iris Hebbeker\textsuperscript{\textdagger},
Gideo Joubert\textsuperscript{\ddag}, 
Jana Kreiß\textsuperscript{\ddag},
and Antoine Mottet\textsuperscript{\ddag}\\
\textsuperscript{\textdagger}TU Graz\\
\textsuperscript{\ddag}Hamburg University of Technology
}

\date{}

\begin{document}
\maketitle
%${}$
%\vspace{-5em}

\begin{abstract}
\noindent We prove $\RE$-completeness of the quantum homomorphism problem parameterised by
families of graphs derived from the classic metric association schemes. These include Kneser graphs, $q$-Kneser graphs, and the complements of Johnson, Grassmann, and Hamming graphs.
Our proof develops a spectral method for establishing non-contextuality of quantum polymorphisms. It combines an equality analysis of Roberson's bound on the projective packing number in terms of Schrijver's theta with a structural argument inspired by Erd\H{o}s--Ko--Rado theory.
\end{abstract}
\setcounter{page}{0}\thispagestyle{empty}\clearpage

\newpage
\tableofcontents

%\newpage

\section{Introduction}
\label{sec_intro}

Quantum graph homomorphisms provide a rich mathematical framework in which to study quantum nonlocality and contextuality (in particular, strong violations of Bell's inequality) as well as the computational power afforded by entanglement. Given graphs $F$ and $G$, the homomorphism game involves a classical verifier and two cooperating, computationally unbounded players who seek to convince the verifier that $F$ admits a homomorphism to $G$. The verifier queries the players separately and checks whether their answers are locally consistent with such a homomorphism; the players cannot communicate during the game. A perfect classical strategy exists if, and only if, a homomorphism $F\to G$ exists. Quantumly, however, the players may coordinate their answers by replying to the verifier's queries according to measurements of 
a shared entangled state. In some cases, this  allows them to pass every check with certainty even when no classical homomorphism exists. This phenomenon is known as \textit{quantum pseudo-telepathy}~\cite{galliard2002pseudo,brassard2005quantum}. In this work, we focus on the \textit{finite-dimensional tensor-product oracular} setting, in which perfect quantum strategies for the cooperating provers can be represented by projective measurements over $\C^d$ for some $d\in\N$, indexed by vertices of $F$, whose outcomes are the vertices of $G$, with measurements associated with adjacent vertices commuting~\cite{abramsky2017quantum}.

When both $F$ and $G$ are part of the input, deciding whether a perfect classical strategy $F\to G$ exists is $\mathsf{NP}$-complete. A richer picture emerges in the so-called \textit{non-uniform} setting, where the target graph $G$ is fixed. The Hell--Ne\v{s}et\v{r}il theorem gives a complete dichotomy for this problem: for every finite loopless graph $G$, the problem of deciding whether an input graph admits a homomorphism to $G$ is in $\mathsf P$ when $G$ is bipartite and is $\mathsf{NP}$-complete otherwise \cite{hell1990complexity}.

The quantum side of the complexity landscape of graph homomorphism is far less understood. At the unrestricted level, the $\MIP=\mathsf{RE}$ 
theorem
implies that the entangled problem is $\mathsf{RE}$-complete \cite{ji2021mip}. In the non-uniform setting, however, establishing undecidability requires reductions that remain sound in the presence of noncommuting measurements. At the tractable end, Man\v{c}inska and Roberson proved that quantum homomorphism games with a bipartite target exhibit no pseudo-telepathy~\cite{manvcinska2016quantum}: in this case, quantum feasibility coincides with classical feasibility, and the problem is therefore polynomial-time solvable. At the opposite end, in~\cite{ji2013binary}, Ji introduced the concept of \textit{commutativity gadgets} to lift some of the classical NP-hardness reductions to reductions that remain sound against quantum provers.
Combining this technique with~\cite{slofstra2019set}, Harris proved undecidability of entangled $3$-colouring~\cite{harris2024universality}. 
Culf and Mastel proved that the latter problem is in fact $\RE$-complete by reducing from the protocols in $\MIP=\RE$, and also extended $\RE$-completeness to
a large class of classically hard homomorphism problems  parameterised by Boolean structures as well as so-called non-two-variable-falsifiable structures (which, however, do not capture the setting of loopless undirected graphs)~\cite{culf2025re}.
More recently, Culf, van Dobben de Bruyn, Vernooij, and Zeman constructed oracular commutativity gadgets for $k$-colouring for every $k\geq4$, thus establishing that quantum $k$-colouring is $\RE$-complete for such $k$, too~\cite{culf2025existence}. Their work also singled out odd cycles and odd graphs as natural candidates beyond complete graphs, whose complexity (and, in particular, existence of commutativity gadgets) could not be resolved via the available theory.

A systematic approach to study the question algebraically (in the more general context of homomorphisms between relational structures) was later developed by the first, third, and fifth authors in~\cite{CJM} through the notion of \emph{quantum polymorphisms}. An $m$-ary quantum polymorphism of a graph $G$ is a quantum homomorphism from $G^m$ to $G$,
where $G^m$ denotes the direct power of $G$; i.e., a perfect quantum strategy for the game whose goal is to prove to a verifier the statement that $G^m\to G$. 
This statement is trivially true, so even a classical, honest strategy exists (for instance, any $m$-ary dictator gives one such strategy). Yet, it was shown in~\cite{CJM} that the \textit{space} of these quantum strategies (in particular, the algebraic structure of such space expressed in terms of so-called minor identities) is able to completely determine the complexity of the problem, thus lifting the analogous classical phenomenon underlying the classical CSP dichotomy theory.
Concretely, quantum polymorphisms serve two related purposes. First, they characterise which fixed-template problems admit Ji's commutativity gadgets. Second, they provide an algebraic framework for transferring complexity between nonlocal games, paralleling the role of ordinary polymorphisms in the classical theory of constraint satisfaction problems.

A quantum polymorphism is \emph{non-contextual} if all the projective measurements defining it commute. In this case, the measurements can be simultaneously diagonalised, and the quantum strategy decomposes into a direct sum of classical polymorphisms. The results of~\cite{CJM} imply, in particular, that if $G$ is non-bipartite (i.e., classically NP-complete) and every quantum polymorphism of $G$ is non-contextual, then a categorical power of $G$ serves as a commutativity gadget, and the corresponding quantum homomorphism problem is thus $\RE$-complete. This shifts the main task from constructing an explicit reduction (and a specific commutativity gadget) to proving a commutation statement about quantum operations.
The method used in \cite{CJM} to establish non-contextuality was based on \emph{contextuality bifurcations}. Roughly speaking, noncommutativity among two measurements forces certain non-orthogonality patterns among the projectors of a quantum polymorphism. If these patterns contradict the orthogonality relations required of a perfect strategy, then no contextual quantum polymorphism can exist. This argument, that was sufficient to prove undecidability for odd cycles, is inherently \textit{combinatorial}: the contradiction is obtained from properties of individual neighbourhoods, namely vertex degrees and intersections of neighbourhoods.

In the present work we make progress towards a quantum Hell--Ne\v{s}et\v{r}il classification by developing a complementary, \textit{spectral} approach for achieving non-contextuality of quantum polymorphisms for graphs, based on ideas from spectral graph theory. 
As a consequence of this approach, we prove the following main result.

\begin{theorem}%
\label{thm_RE_completeness_classes}
    Let %
    $G$ be one of the following graphs:
    \begin{itemize}
        \item The Kneser graph $KG(n,k)$ with $n>2k\geq 2$.
        \item The complement $\overline{J(n,k)}$ to the Johnson graph with $n>2k\geq 4$.
        \item The q-Kneser graph $KG_q(n,k)$ with $n>2k\geq 2$ and $q$ a prime power.
        \item The complement $\overline{J_q(n,k)}$ to the Grassmann graph with $n>2k\geq 4$ and $q$ a prime power.
        \item The complement $\overline{H(d,q)}$ to the Hamming graph with $d\geq 2$ and $q\geq 3$.
    \end{itemize}
    Then determining whether an input graph admits a quantum homomorphism to $G$ is $\RE$-complete.
\end{theorem}
In particular, this settles the question in~\cite{culf2025existence} on odd graphs.
The starting point of our method is Schrijver's parameter $\vartheta^-$, which gives an upper bound on the independence number of a graph and strengthens the Lov\'{a}sz theta bound~\cite{SchrijverDelsarteLovasz1979}. Crucially, this bound can be \emph{quantised}. Roberson showed that $\vartheta^-(G)$ also bounds the \textit{projective packing number} $\alpha_p(G)$~\cite{roberson2013variations}, which in turn bounds the quantum independence number $\alpha_q(G)$ introduced by Man\v{c}inska and Roberson~\cite{manvcinska2016quantum} as the natural quantum relaxation of the classical independence number. Thus, for every graph $G$,
$$
\alpha(G)\leq\alpha_q(G)\leq\alpha_p(G)\leq\vartheta^-(G).
$$
We also refer to~\cite{wocjan2019spectral} for further quantum extensions of spectral bounds on the independence number.

Our main technical result states that the quantum graph homomorphism problem parameterised by a graph $G$ is $\RE$-complete whenever $G$ meets a specific spectral condition that strengthens the requirement that Schrijver's bound be tight by constraining the shape of a certificate matrix witnessing the tightness of the bound.
If this condition holds, using that, if $\vartheta^-=\alpha$ then $\vartheta^-=\alpha_p$, an equality analysis of Roberson's bound implies that
all quantum polymorphisms of $G$ are non-contextual, so the theory of quantum polymorphisms allows concluding that $G$ admits a commutativity gadget and the corresponding problem is thus $\RE$-complete.
We apply this spectral non-contextuality criterion to graphs derived from symmetric association schemes. The guiding idea is that the semidefinite-programming constraints defining Schrijver-rigidity become more manageable linear-programming constraints in the Bose--Mesner algebra of a scheme; a similar strategy was adopted in~\cite{ciardo2025semidefinite} to study SDP and affine relaxations of approximate graph homomorphism problems. Schrijver proved that, for graphs derived from association schemes, the $\vartheta^-$ bound coincides with Delsarte's LP bound \cite{SchrijverDelsarteLovasz1979}. We prove a rigidity analogue of this result, by showing that a representation of a graph derived from an association scheme is Schrijver-rigid if and only if it is Delsarte-rigid.

In order to show that the equivalent Delsarte LP conditions are met for concrete graph families, we need to have a good picture of the spectral properties of the scheme (in particular, the matrix of eigenvalues, translating between the primal and dual bases of the scheme).
We apply our criterion to graphs arising from the three classical metric association schemes: the \textit{Johnson}, \textit{Grassmann}, and \textit{Hamming} schemes, whose spectral descriptions are well known. In this way, we show that both the maximum-distance graphs and the complements of the distance-one graphs satisfy the non-contextuality condition. This yields $\RE$-completeness for the corresponding quantum homomorphism problems,
thus establishing~\Cref{thm_RE_completeness_classes}.
Using as a black box a result of~\cite{culf2025existence} on the equivalence of oracular and non-oracular commutativity gadgets for $C_4$-free graphs, we also extend $\RE$-completeness to the non-oracular regime for odd graphs.

The next section contains a more detailed overview of our proof strategy. The formal proofs are given in the remaining part of the paper.

\section{Overview of results and techniques}

For a Hilbert space $H$, $\End(H)$ and $\Proj(H)$ denote the sets of 
linear applications and projectors  over $H$, respectively.  All Hilbert spaces appearing in this work are nonzero and
finite-dimensional. 
For a statement $\mathcal P$, we write $\indfunction{\mathcal P}$ for its Iverson bracket, which equals $1$ if $\mathcal P$ is true and $0$ otherwise.
Given a graph $G=(X,E)$ and two vertices $x,y\in X$, we write $x\sim_G y$ (or simply $x\sim y$, if $G$ is clear) to denote that $x$ and $y$ are adjacent.
A \emph{projective measurement} (PVM) over a Hilbert space $H$ with outcome set $B$ is a family $(Q_{b})_{b\in B}$ of projectors on $H$ with the property that $\sum_{b\in B}Q_b = I$.
A \emph{quantum function} from a set $A$ to a set $B$ is a family of PVMs $(Q_{a,b})_{b\in B}$ on a Hilbert space $H$, one PVM for each $a\in A$.

\begin{definition}[\cite{manvcinska2016quantum,abramsky2017quantum}]
\label{def:oracular-quantum-homomorphism}
Let $G=(X,E)$ and $G'=(X',E')$ be graphs. An \emph{(oracular) quantum homomorphism}
$Q\colon G\qto G'$ is
a quantum function $Q\colon X\qto X'$ satisfying the following two conditions:

\begin{itemize}
    \item[$(i)$] $Q_{x_1,x'_1}Q_{x_2,x'_2}=0$ whenever $x_1\sim_G x_2$ and $x_1'\not\sim_{G'} x_2'$;
    \item[$(ii)$]
    $[Q_{x_1,x'_1},Q_{x_2,x'_2}]=0$ for each $x_1',x_2'\in X'$ whenever $x_1\sim_G x_2$.
\end{itemize}
\end{definition}
We write $\qCSP(G')$ for the following decision problem: given a finite graph
$G$ as input, decide whether there exists a quantum
homomorphism $G\qto G'$.
\begin{remark}
The notion above is sometimes referred to as the \textit{oracular} version of quantum homomorphism, in that it requires commutation of the PVMs associated with adjacent vertices to capture oracularisable strategies in the corresponding non-local game, see e.g.~\cite{abramsky2017quantum}. In a \textit{non-oracular} quantum homomorphism, requirement $(ii)$ is dropped~\cite{manvcinska2016quantum}.
\end{remark}
 We refer to~\Cref{subsec_prelimns_multilinear_algebra} for further details on the notation.

\subsection{Schrijver theta and projective packings}

The starting point of our analysis is  Schrijver's theta number, which we define next in the dual SDP formulation.

\begin{definition}[\cite{SchrijverDelsarteLovasz1979}]
\label{defn_schrijver_theta}
A
\emph{dual $\vartheta^-$
certificate for a graph $G$ at value $\beta\in\R$}  is a real symmetric matrix $Q\in\R^{X\times X}$ satisfying
\[
 Q\succeq0,\qquad
 Q_{xx}=\beta-1\quad(x\in X),\qquad
 Q_{xy}\leq-1\quad(x,y\in X,\ x\ne y,\ x\not\sim_Gy).
\]
We let $\vartheta^-(G)$ be the minimum $\beta$ for which a dual $\vartheta^-$ certificate for $G$ at value $\beta$ exists.
\end{definition}
Schrijver showed in~\cite{SchrijverDelsarteLovasz1979} that $\vartheta^-$ upper bounds the independence number $\alpha(G)$ of any graph $G$. 
In this work, we shall be interested in the regime where Schrijver's bound is tight; i.e., $\vartheta^-(G)=\alpha(G)$. 
This regime is interesting in our quantum setting as a consequence of two results, proved by Roberson~\cite{roberson2013variations} and Man\v{c}inska--Roberson~\cite{manvcinska2016quantum}, which sharpen Schrijver's bound to capture non-classical notions of independent sets.
\begin{definition}[\cite{roberson2013variations}]
\label{def:projective-packing}
A \emph{projective packing} of $G$ on a nonzero finite-dimensional
Hilbert space $H$ is a family
$R=(R_x)_{x\in X}\subseteq\Proj(H)$ such that $R_xR_y=0$
holds for all $x\sim_G y$.
Its value is
\[
 \operatorname{val}(R)
   =\frac{1}{\dim H}\sum_{x\in X}\operatorname{rank}(R_x).
\]
The \emph{projective packing number} of $G$ is
\[
 \alpha_p(G)=\sup\{\operatorname{val}(R):
             R\text{ is a projective packing of }G\},
\]
where the supremum is taken over all nonzero finite-dimensional Hilbert
spaces.  
\end{definition}

\begin{definition}[\cite{manvcinska2016quantum}]
    The \emph{quantum independence
number} $\alpha_q(G)$ of a graph $G$ is the maximum $n$ such that there exists a quantum non-oracular homomorphism from the $n$-clique $K_n$ to the complement of $G$.
\end{definition}

It was proved in~\cite{roberson2013variations} that $\vartheta^-$ upper bounds the projective packing number of any graph. 
Furthermore, in~\cite{manvcinska2016quantum}, it was shown that
$\alpha_p(G)$ is an upper bound on $\alpha_q(G)$. Since $\alpha_q(G)\geq\alpha(G)$ (as is clear by considering a 1-dimensional Hilbert space), we have the following chain of inequalities.

\begin{theorem}
[\cite{SchrijverDelsarteLovasz1979,roberson2013variations,manvcinska2016quantum}]
\label{thm_alpha_p_schrijver}
    For every finite graph $G$, $\alpha(G)\leq\alpha_q(G)\leq\alpha_p(G)\leq\vartheta^-(G)$.
\end{theorem}

In particular, if Schrijver's bound is tight for a graph $G$, all quantities in~\Cref{thm_alpha_p_schrijver} coincide.
It shall be particularly useful, for our purposes, to understand the implication of the fact that $\alpha_p=\vartheta^{-}$. This requires a slight extension of Roberson's argument in~\cite{roberson2013variations}, whose proof we shall give in~\Cref{sec_non_contextuality_schrijver}. 

\begin{restatable}{theorem}{lemoperatorthetarigidity}
\label{lem:operator-theta-rigidity}
Suppose that $\alpha_p(G)=\vartheta^-(G)$, and let this be witnessed by a projective packing $R=(R_x)_{x\in X}\subseteq \Proj(H)$ and a dual $\vartheta^-$ certificate $Q$. Then  the operator-valued function $x\mapsto R_{x}$, seen as an element of $\C^X\otimes\End(H)$, lies
in $\ker Q\otimes\End(H)$; equivalently, each of its scalar
matrix entries belongs to $\ker Q$.
\end{restatable}

The intuition is as follows.
Suppose that $\alpha(G)=\vartheta^-(G)$ (i.e., the classical Schrijver's bound is tight), and let $Q$ be an optimal dual $\vartheta^-$
certificate.
Letting $v_S$ be the indicator vector of a maximum independent set $S$ of $G$, we have
\[
  v_S^{\top} Q\,v_S
  \;=\;\sum_{x,y\in S}Q_{xy}
  \;\leq\;\lvert S\rvert\bigl(\vartheta^-(G)-1\bigr)-\lvert S\rvert\bigl(\lvert S\rvert-1\bigr)
  \;=\;\lvert S\rvert\bigl(\vartheta^-(G)-\lvert S\rvert\bigr)
  \;=\;0 .
\]
Since $Q$ is positive semidefinite, the left-hand side 
vanishes, and therefore $v_S\in\ker Q$.
A special case of this is the classical statement for the Hoffman bound~\cite[Theorem~2.4.2]{godsil2016erdos}, asserting that  the
characteristic vector of any maximum independent set of a regular graph attaining Hoffman bound is a linear combination of the
all-ones vector and an eigenvector relative to the minimum eigenvalue of the graph.
\Cref{lem:operator-theta-rigidity} is an operator-version of this phenomenon, describing
the structure of projective packings attaining equality in Roberson's bound.

At a high level, we aim to use this information as follows. Recall that our final goal (\Cref{thm_RE_completeness_classes}) is to prove $\RE$-completeness of the quantum homomorphism problem parameterised by $G$. To that end, since the homomorphism problem for $G$ is classically hard, it is enough to show that all quantum polymorphisms of $G$ are non-contextual. Our strategy is then to use~\Cref{lem:operator-theta-rigidity}  to force the configurations generated
by the system of projectors in a perfect quantum strategy for the polymorphism game of $G$ to lie in some restricted subspace. This produces nontrivial linear
identities among the projectors, which, in turn, can be used to force commutation across all measurements. In the following subsection, we shall add to $G$ the extra structure that is needed to force this restriction.

\subsection{Schrijver-rigidity and tame disjointness representations}
\label{subsec:rigid-schrijver}
The  rigidity condition following from~\Cref{lem:operator-theta-rigidity} alone is not sufficient to prove commutation. To strengthen it, we  use an idea from the theory of (generalised) Erd\H{o}s--Ko--Rado theorems, see the treatment in
\cite{godsil2016erdos}. Intuitively, we want to have some explicit description of the large independent sets of $G$ that achieve Schrijver's bound. We formalise this idea as follows.
Henceforth, we fix a graph $G=(X,E)$ and let $v=|X|$ be its order.

\begin{definition}
\label{defn_disjointness_representation}
Fix a finite set $P$ (whose elements we call \textit{features}) and an integer $1\leq r\leq |P|$.  A
\emph{$(P,r)$-disjointness representation} of $G$ consists of an injective map
\[
             T\colon X\to\binom Pr,
             \qquad x\mapsto T_x,
\]
such that
\begin{align}
\label{eqn_1051_1408}
             x\sim_Gy\quad\Longrightarrow\quad T_x\cap T_y=\emptyset,
\end{align}
and $P=\bigcup_{x\in X}T_x$.  Equivalently,  the map $T$ is an injective homomorphism
from $G$ to the Kneser graph $\KG(|P|,r)$ with no unused features.
\end{definition}
The above can be viewed
as an injective variant of  a fractional colouring of a graph, see e.g.~\cite[\S~3.1]{scheinerman2011fractional}, or as a weaker version of Kneser representations of graphs in the sense of~\cite{hamburger2009kneser,furedi2018kneser}. The difference from the latter is that a Kneser representation also requires the converse implication of~\cref{eqn_1051_1408} to hold and therefore captures graphs that are induced subgraphs of Kneser graphs.

Fix a $(P,r)$-disjointness representation $T$ of $G$. Given an element
$p\in P$, define the
\textit{feature fibre} of $p$ as the set
$S_p=\{x\in X\mid p\in T_x\}$. Moreover, define the \textit{feature function} of $p$ as 
the indicator function of $S_p$; i.e., 
the function
$\chi_p\colon X\to\{0,1\}$ given by
$\chi_p(x)=\indfunction{p\in T_x}$.%
\begin{observation}
\label{obs_independence_number_disjointness_representation}
    Note that 
    \begin{align}
\label{eqn_1648_0308}
                  \sum_{p\in P}\chi_p=r\1.
\end{align} and, thus, \begin{align*}
    \sum_{p\in P}|S_p|=\sum_{p\in P}\sum_{x\in X}\chi_{p}(x)=\sum_{x\in X}\sum_{p\in P}\chi_p(x)=\sum_{x\in X}r=rv.
\end{align*}
Since it follows from
\Cref{defn_disjointness_representation} that $S_p$ is a nonempty
independent set of $G$ for each $p\in P$, it is clear that, if $G$ admits a $(P,r)$-disjointness representation, then $\alpha(G)\geq\frac{rv}{|P|}$.
\end{observation}

We let 
the \emph{incidence space} of a $(P,r)$-disjointness representation of $G$ be the space
\[
                  \incspace=\Span\{\chi_p \mid p\in P\}\subseteq\C^X.
\]
It follows from~\eqref{eqn_1648_0308} that $\1\in\incspace$. Hence, we write $\incspace=\C\1\oplus U$ with $U\coloneqq\incspace\cap\1^\perp$.

In the following, we regard $\C^{X\times X}$ as the space of square matrices indexed by $X$
and identify $f\otimes g$ with $fg^\top$.

\begin{definition}
The \emph{edge kernel} of a disjointness representation $T$ is the set
\[
 \edgekernel(T)
 =\{F\in\incspace\otimes\incspace \mid
          F(x,y)=0\text{ whenever }x\sim_Gy\}.
\]
We shall often abbreviate it to $\edgekernel$ when $T$ is clear.
\end{definition}
Consider now the  constant-diagonal subspace
\[
 \diagonalSubspace
 =\{F\in\edgekernel \mid F(x,x)\text{ is independent of }x\}.
\]
We denote the neighbourhood of a vertex $x\in X$ by
$N_G(x)=\{a\in X \mid a\sim_Gx\}$; we write $N(x)$ when $G$ is clear.
Consider the map $\Delta_x\colon \diagonalSubspace\to 
       \C^{N_G(x)\times N_G(x)}$ given by
\begin{align}
\label{eqn_18_08_1510}
 \Delta_x(F)_{a,b}=F(a,a)-F(a,b)
\end{align}
for all $a,b\in N_G(x)$.

\begin{definition}
A disjointness representation is
\emph{tame} if the following two conditions hold:
\begin{enumerate}
\item every $F\in\diagonalSubspace$ is symmetric:
      \[
                         F(x,y)=F(y,x);                     \tag{R1}\label{tag:R1}
      \]
 \item for every $x\in X$, the map $\Delta_x$ is injective:
      \[
       F\in\diagonalSubspace,\quad
       F(a,b)=F(a,a) \;\;\mbox{ for each }a,b\in N_G(x)
       \quad\Longrightarrow\quad F=0.
                        \tag{R2}\label{tag:R2}
      \]
\end{enumerate}
\end{definition}
The quantum polymorphisms of $G$ are quantum homomorphisms from powers of $G$ to $G$. Hence, we need to understand the structure of disjointness representations of \textit{powers} of $G$. In particular, for any disjointness representation $T$ of $G$ and any positive integer $m$, there exists a canonical disjointness representation $T^{[m]}$ of $G^m$
defined by $T^{[m]}(x_1,\dots,x_m)\coloneqq\{(i,p)\mid p\in T_{x_i}\}$, see~\Cref{representation-power}.
In the following, we write $\incspace^{[m]}$, $\edgekernel^{[m]}$, and $\diagonalSubspace^{[m]}$ for $\incspace(T^{[m]})$, $\edgekernel(T^{[m]})$, and $\diagonalSubspace(T^{[m]})$ when $T$ is clear from context.
In particular, we have
\[
\incspace^{[m]}
=
\operatorname{span}
\left\{
\mathbf 1_X^{\otimes(i-1)}
\otimes \chi_p
\otimes \mathbf 1_X^{\otimes(m-i)}
\;\middle|\;
i\in[m],\ p\in P
\right\},
\]
where, recall, $\chi_p(x)=\indfunction{p\in T_x}$.

In this work we consider disjointness representations whose corresponding incidence spaces (in any dimension $m$) are precisely the kernel of an optimal Schrijver dual certificate, as formalised in the following definition.

\begin{definition}
\label{def:power-theta-rigid}
A $(P,r)$-disjointness representation $T$ of $G$ is
\emph{Schrijver-rigid} if, for every $m\geq1$, the graph
$G^m$ has a dual $\vartheta^-$ certificate $Q_m$ at value $\beta_m=\frac{rv^m}{|P|}$ such that $\ker Q_m=\incspace^{[m]}$.
\end{definition}

Combining~\Cref{thm_alpha_p_schrijver} and Observation~\ref{obs_independence_number_disjointness_representation}, we note that, for
any graph $G$ admitting a Schrijver-rigid representation, Schrijver's bound must be tight, in the sense that $\alpha(G)=\vartheta^-(G)$. 
We can now state the following main technical result.

\begin{restatable}{theorem}{thmcriterion}
\label{thm:criterion}
Let $G$ be a connected non-bipartite graph. 
Suppose that $G$ has a Schrijver-rigid tame disjointness representation.
Then $G$ has a commutativity gadget and $\qCSP(G)$ is $\RE$-complete.
\end{restatable}

\subsection{Non-contextuality from tame Schrijver-rigidity}
We next outline the proof of~\Cref{thm:criterion}. The key to the proof is the following result.

\begin{restatable}{lemma}{lemdegreeonecommutativity}
\label{lem:degree-one-commutativity}
Let $G$ be a connected non-bipartite graph admitting a  $(P,r)$-disjointness tame representation with incidence space $\incspace$.
Let also $H$ be a finite-dimensional Hilbert space. Take a family of projectors
$P_{\mathbf x,p}\in\Proj(H)$ for $\mathbf x\in X^m,\ p\in P$
satisfying the following conditions:
\begin{enumerate}
\item 
the assignment $\mathbf x\mapsto P_{\mathbf x,p}$ belongs to
      $\incspace^{[m]}\otimes\End(H)$ for every $p\in P$;
\item $[P_{\mathbf x,p},P_{\mathbf x,q}]=0$ for every
      $\mathbf x\in X^m$ and every $p,q\in P$;
\item $[P_{\mathbf x,p},P_{\mathbf y,q}]=0$ for every $p,q\in P$ whenever
      $\mathbf x\sim_{G^{m}}\mathbf y$.
\end{enumerate}
Then all projectors in the family commute with each other.
\end{restatable}

Once the lemma above is established,~\Cref{thm:criterion} can be easily derived from the theory of quantum polymorphisms from~\cite{CJM}, see~\Cref{subsec_proving_main_technical_thm}.

We are left to prove~\Cref{lem:degree-one-commutativity}. Our strategy here is to first prove that, for any $x\in X$, the dual space of $\diagonalSubspace$ is generated by the linear functionals $\delta-\epsilon_{a,b}$ for $a,b$ neighbours of $x$ in $G$, where $\delta$ and $\epsilon_{a,b}$ assign to $F\in\diagonalSubspace$ its common diagonal and the $(a,b)$-th entry, respectively. In particular, the span $L_x$ of the latter functionals is independent of $x$. 
Then, for fixed $p,q\in P$, we consider the commutator map $C_{p,q}\colon X^m\times X^m\to\End(H)$ defined by   
$C_{p,q}(\mathbf x,\mathbf y)
       \coloneqq[P_{\mathbf x,p},P_{\mathbf y,q}]$ and we show that $C_{p,q}$
belongs to
$\edgekernel^{[m]}\otimes\End(H)$. In particular, we can decompose the commutator as a linear combination of elements whose first factor lies in the edge kernel $\edgekernel^{[m]}$ of the canonical power of the disjointness representation.
This means that we need to understand the shape of $\edgekernel^{[m]}$. The main technical part of our proof is to show that $\edgekernel^{[m]}$ in fact factors along its $m$ one-dimensional components. More precisely,  letting 
$\iota_i$ be the embedding $\C^{X}\to\C^{X^m}$ defined by $\iota_i(f)=\1^{\otimes(i-1)}\otimes f\otimes\1^{\otimes(m-i)}$, we show the following.

\begin{restatable}{lemma}{lempowerkernel}
\label{lem:power-kernel}
Fix $m\geq1$.  Assume that $G$ is connected and non-bipartite.  
Then $\edgekernel^{[m]}=\bigoplus_{i\in[m]}\iota_i(\edgekernel)$.
\end{restatable}

It is then immediate to show that each factor of $C_{p,q}$ according to this decomposition belongs to $\diagonalSubspace\otimes \End(H)$. From this, a straightforward argument using Jacobi identity shows that the commutator $C_{p,q}$ trivialises and, thus, all projectors in the family commute with each other. This proves~\Cref{lem:degree-one-commutativity}.

\subsection{Schrijver-rigidity in association schemes}
\label{subsec_overview_association_schemes}
In order to use the $\RE$-completeness condition of~\Cref{thm:criterion} in practice, it is convenient to focus on particularly symmetric graph classes where Schrijver-rigidity can be rephrased in terms of spectral conditions. 
Graphs derived from association schemes give a nice framework for constructing such families. 
We start by recalling some standard notions from the theory of association schemes, see e.g.~\cite{DelsarteAssociationSchemes1973,bose1939partially,bose1952classification,bannai_ito_1984,godsil2016erdos,brouwer1989distance,godsil2010associationSchemes}.

\begin{definition}
A \emph{(symmetric) association scheme} on a finite set $X$ is a family
$\mathfrak X=\{A_0,A_1,\ldots,A_s\}$ of nonzero symmetric Boolean
$X\times X$ matrices such that
\[
 A_0=I,\qquad \sum_{i=0}^sA_i=J,
 \qquad A_iA_j=A_jA_i\in\Span_\C\{A_0,\ldots,A_s\}
 \quad(i,j\in\{0,\ldots,s\}).
\]
\end{definition}
All schemes appearing in this paper are symmetric, and by the expression ``association scheme'' we mean a symmetric association scheme.
Observe that the matrices $A_i$ in $\mathfrak X$ can be seen as the adjacency matrices of a partition
$X\times X=R_0\sqcup\cdots\sqcup R_s$.
The complex span $\C[\mathfrak X]$ of the matrices in $\mathfrak X$ is
the \emph{Bose--Mesner algebra} of $\mathfrak X$.
It has a second distinguished basis of pairwise orthogonal real projection matrices %
$E_0=|X|^{-1}J,E_1,\ldots,E_s$, which
satisfy
$\sum_{j=0}^sE_j=I$.
Note that the set $\{E_0,E_1,\dots, E_s\}$ is a PVM onto the Hilbert space $\C^X$ with outcome set $\{0,\dots,s\}$.

It will be useful to keep in mind that both the primal basis $(A_i)_i$ and the dual basis $(E_j)_j$ are orthogonal with respect to the Frobenius inner product (see e.g.~\cite[\S~3.5]{godsil2016erdos}). Indeed, since the relations $R_i$ are pairwise disjoint,
\begin{align}
\label{eqn_2608_1840}
\langle A_i,A_k\rangle
=\Tr(A_iA_k)
=\sum_{x,y\in X}(A_i)_{xy}(A_k)_{xy}
=\delta_{ik}|R_i|.
\end{align}
Moreover, since the dual basis consists of mutually orthogonal idempotents, we have
\begin{align}
\label{eqn_2608_1840_b}
\langle E_j,E_\ell\rangle
=\Tr(E_jE_\ell)
=\delta_{j\ell}\Tr(E_j)
=\delta_{j\ell}\rank(E_j).
\end{align}
The spaces $E_j\C^X$ are called the \textit{primitive modules} of $\mathfrak X$, and since $E_i,E_j$ are orthogonal  %
for $i\neq j$, the modules $E_i\C^X$ and $E_j\C^X$ are orthogonal for the standard inner product on $\C^X$.
Moreover, $E_0\C^X$ is $\C\1_X$.
We shall consider the change-of-basis matrix $P\coloneqq (P_{ji})_{0\leq j,i\leq s}$
satisfying
\[
    A_i=\sum_{j=0}^sP_{ji}E_j,
\]
so that $P_{ji}$ is the eigenvalue of $A_i$ on $E_j\C^X$.
For $\Gamma\subseteq[s]$, the \emph{derived graph} $G_\Gamma$ has
vertex set $X$ and adjacency matrix
$A_\Gamma=\sum_{i\in\Gamma}A_i$.

The independence number of $G_\Gamma$ admits the following upper bound, due to Delsarte. Define
\begin{align}
\begin{aligned}
\alpha_{\mathrm{Del}}(G_\Gamma)
\coloneqq\min_{\mathbf b\in\mathbb R^{s+1}}\quad
    &\mathbf1^\top\mathbf b\\
\text{subject to}\quad
    &b_0=1,\\
    &\mathbf b\geq\mathbf0,\\
    &(P^\top\mathbf b)_i\leq0
        &&(i\in[s]\setminus\Gamma).
        \end{aligned}
\end{align}

\begin{theorem}[\cite{DelsarteAssociationSchemes1973}]
\label{thm_delsarte_bound}
    Let $\mathfrak X$ be an association scheme, and let $G_\Gamma$ be a derived graph of $\mathfrak X$. Then $\alpha(G_\Gamma)\leq\alpha_{\operatorname{Del}}(G_\Gamma)$.
\end{theorem}

It was shown in~\cite{SchrijverDelsarteLovasz1979} that, in fact, Delsarte's bound $\alpha_{\mathrm{Del}}$ (which is formulated via a linear program) is \textit{equal} to Schrijver's SDP bound $\vartheta^{-}$ for any $G_\Gamma$.
Our goal is to prove a ``rigidity'' version of this fact: we aim to define a Delsarte rigidity condition (fully in terms of LP constraints) which coincides with Schrijver's rigidity when the latter is restricted to graph classes derived from association schemes. 

First of all, we need to define what it means for a disjointness representation to preserve the scheme structure.
Let $T$ be a $(P,r)$-disjointness representation of a derived graph
$G=G_\Gamma$, with incidence space $\incspace=\C\1\oplus U$. 

\begin{definition}
\label{def:bm-incidence-module}
The representation $T$ is \emph{scheme-compatible} if there is a set
$\Lambda\subseteq[s]$ such that
$U=\bigoplus_{j\in\Lambda}E_j\C^X$.
We call $\Lambda$ the \emph{index set} of $T$.
\end{definition}

As in~\Cref{subsec:rigid-schrijver}, we need to take higher powers of a given association scheme, see e.g.~\cite[\S3]{song2002fusion}. 

\begin{definition}[\cite{song2002fusion}]
Let $\mathfrak X$ be an association scheme over $X$.
    For $m\geq1$, the \textit{tensor-power scheme} (or \textit{direct-power scheme}) $\mathfrak X^{[m]}$ is the association scheme over the set $X^m$ whose primal and dual bases are
\[
 A_{\bi}\coloneqq A_{i_1}\otimes\cdots\otimes A_{i_m},
 \qquad
 E_{\bj}\coloneqq E_{j_1}\otimes\cdots\otimes E_{j_m},
\]
where $\bi=(i_1,\ldots,i_m)$ and $\bj=(j_1,\ldots,j_m)$ belong to
$\{0,\ldots,s\}^m$.
\end{definition}

For a scheme-compatible representation, we set $\Lambda^{[m]}=\{\mathbf0\}\cup
 \{j\be_i\mid j\in\Lambda,\ i\in[m]\}$, so
\begin{align}
    \label{eqn_1653_2408}
 \incspace^{[m]}
 =\bigoplus_{\bj\in\Lambda^{[m]}}E_{\bj}\C^{X^m}.        
\end{align}
Note that the eigenvalue of $A_{\bi}$ on
$E_{\bj}\mathbb C^{X^m}$ is $P_{\bj,\bi}\coloneqq \prod_{\ell\in[m]}P_{j_\ell i_\ell}$.
Indeed, 
\begin{align}
\label{eqn_1807_1208}
    A_{\bi}E_{\bj}
    =
    \left(\bigotimes_{\ell\in[m]}A_{i_\ell}\right)\left(\bigotimes_{\ell\in[m]}E_{j_\ell}\right)
    =
    \bigotimes_{\ell\in[m]}A_{i_\ell}E_{j_\ell}
    =
    \bigotimes_{\ell\in[m]}P_{j_\ell i_\ell}E_{j_\ell}
    =
    \prod_{\ell\in[m]}P_{j_\ell i_\ell}\bigotimes_{\ell\in[m]}E_{j_\ell}
    =
    P_{\bj,\bi}E_\bj.
\end{align}
Every matrix $F_m$ in the Bose--Mesner algebra of the $m$-th tensor-power scheme has
unique expansions
\begin{align}
\label{eqn_1758_1208}
 F_m=\sum_{\bi}f_{\bi}^{(m)}A_{\bi}
    =\sum_{\bj}\lambda_{\bj}^{(m)}E_{\bj},
 \qquad
 \lambda_{\bj}^{(m)}=\sum_{\bi}f_{\bi}^{(m)}P_{\bj,\bi}.
\end{align}
Observe that
\begin{align}
\label{eqn_1748_2608}
    \langle A_\bi,F_m\rangle&=
    \sum_{\bj}\lambda_\bj^{(m)}\langle A_\bi,E_\bj\rangle
    =
    \sum_{\bj}\lambda_\bj^{(m)}\Tr(A_\bi E_\bj)
    =
    \sum_{\bj}\lambda_\bj^{(m)}P_{\bj,\bi}\Tr(E_\bj)
    =
    \sum_{\bj}\lambda_\bj^{(m)}P_{\bj,\bi}\rank(E_\bj),
\end{align}
where the first equality comes from~\cref{eqn_1758_1208}, the third comes from~\cref{eqn_1807_1208}, and the fourth is true since $E_\bj$ is a projector. 
We can now define a version of Schrijver-rigidity for association schemes. Let $\mu=r/|P|$.

\begin{definition}
\label{def_delsarte}
A scheme-compatible disjointness representation of $G_\Gamma$ with index set $\Lambda$ is \emph{Delsarte-rigid} if, for every $m\geq1$, there exists a
real symmetric matrix $F_m$ in the Bose--Mesner algebra of the $m$-th tensor-power scheme
such that, expressing $F_m$ as in~\cref{eqn_1758_1208}, 
\begin{align*}
 &\lambda_{\mathbf0}^{(m)}=1,\qquad
   \lambda_{\bj}^{(m)}\geq0\quad(\bj\ne\mathbf0),          \tag{D1}\\
 &f_{\mathbf0}^{(m)}=\mu,\qquad
   f_{\bi}^{(m)}\leq0
   \quad\bigl(\bi\notin\Gamma^m\cup\{\mathbf0\}\bigr),  \tag{D2}\\
 &\lambda_{\bj}^{(m)}=0
   \quad\Longleftrightarrow\quad
   \bj\in\Lambda^{[m]}\setminus\{\mathbf0\}.            \tag{D3}
\end{align*}
\end{definition}

It is easy to see directly from the definition above that, if $G_\Gamma$ admits a Delsarte-rigid representation, then Delsarte's bound in~\Cref{thm_delsarte_bound} must be tight. Indeed, take the certificate $F_1$ for $m=1$, and define $b_j\coloneqq\lambda_j^{(1)}\rank(E_j)$ for $j\in\{0,\dots,s\}$. Since $\rank(E_0)=\rank(J)=1$, (D1) implies $b_0=1$ and $\bb\geq 0$. Moreover, for each $i\in[s]\setminus\Gamma$, 
\begin{align*}
    (P^\top\bb)_i
    =
    \sum_{j}P_{ji}\lambda_j^{(1)}\rank(E_j)=\langle A_i,F_1\rangle=f_i^{(1)}\langle A_i,A_i\rangle,
\end{align*}
where we have used~\cref{eqn_1748_2608} and~\cref{eqn_2608_1840}. This is nonpositive by virtue of (D2). Finally, using (D2) again,
\begin{align*}
    \1^\top\bb=\sum_j\lambda^{(1)}_j\rank(E_j)
    =\operatorname{Tr}(F_1)=vf^{(1)}_0=\mu v.
\end{align*}
It follows that $\alpha_{\mathrm{Del}}(G_\Gamma)\leq\mu v$, while Observation~\ref{obs_independence_number_disjointness_representation} implies that $\alpha(G_\Gamma)\geq\mu v$. Hence,~\Cref{thm_delsarte_bound} implies $\alpha(G_\Gamma)=\alpha_{\mathrm{Del}}(G_\Gamma)$. 

We can now lift the equality between $\vartheta^-$ and
$\alpha_{\mathrm{Del}}$ to the level of rigid disjointness representations:
for any $G_\Gamma$, Delsarte-rigidity and Schrijver-rigidity coincide.

\begin{restatable}{lemma}{lemdelsarteiffschrijver}
\label{lem:delsarte-iff-schrijver}
Let $G=G_\Gamma$ be a derived graph of an association scheme.  Then a disjointness
representation of $G$ is Delsarte-rigid if, and only if, it is Schrijver-rigid.
\end{restatable}

A consequence of this is that checking applicability of our $\RE$-completeness criterion~\Cref{thm:criterion} is particularly simple if the given graph is derived from an association scheme:  the SDP condition of~\Cref{def:power-theta-rigid} can be safely replaced by the LP condition of~\Cref{def_delsarte}; the latter can be easily checked  as long as the matrix $P$ of eigenvalues of the scheme is known
(which will be the case for the families of graphs we will consider in~\Cref{section_examples}). 
We point out that the same mechanism was used in~\cite{ciardo2025semidefinite} to simplify the problem of exhibiting fooling certificates for a relaxation combining SDP and affine integer programming applied to the approximate graph homomorphism problem~\cite{brakensiek2018promise}.

\subsection{The classic metric association schemes}
Finally, 
we apply our framework to graphs derived from the Johnson, Grassmann, and Hamming schemes. In each scheme, we consider two natural derived graphs: the maximum-distance graph and the complement of the distance-one graph.

For the maximum-distance graphs of the Johnson and Grassmann schemes, which are the Kneser and $q$-Kneser graphs, our spectral condition 
has a particularly simple witness, given by a disjointness representation such that the non-constant part $U$ of the incidence space is the eigenspace of the adjacency matrix corresponding to its least eigenvalue $\tau$. We call such a representation \emph{Hoffman-rigid}, as formalised in~\Cref{def:hoffman-feature}. (In particular, similarly to the Schrijver and Delsarte cases, the existence of a Hoffman-rigid representation implies that the Hoffman ratio bound~\cite{haemers2021hoffman} on the independence number is tight.)
The maximum-distance Hamming graph is the categorical power $K_q^d$. Since it is homomorphically equivalent to $K_q$, its quantum homomorphism problem is equivalent to quantum $q$-colouring, whose $\RE$-completeness for $q\geq3$ is already known~\cite{culf2025re,culf2025existence}.

The complements of the distance-one graphs require more flexible certificates.
Indeed, when the scheme has diameter at least three (so that the two classes of derived graphs do not coincide), we show that Hoffman's ratio bound is not tight, so no Hoffman-rigid representation can exist and we need to use the full power of Delsarte.
We construct representations using $(k-1)$-subsets in the Johnson case, $(k-1)$-dimensional subspaces in the Grassmann case, and axis-parallel lines  
in the Hamming case. In general, the resulting incidence spaces contain several primitive modules on which the adjacency matrix has different eigenvalues. We therefore construct a suitable linear combination of the edge-relation matrices that acts as the common scalar $-1$ on all these nonconstant incidence modules, while keeping the remaining eigenvalues in the required interval.

Finally, the Kneser family includes the odd graphs
$O_n=\KG(2n-1,n-1)$, for $n\geq2$.
These graphs are $C_4$-free, so a result of~\cite{culf2025existence} allows us to pass from oracular to non-oracular commutativity gadgets. Our construction therefore settles the question posed therein about commutativity gadgets for odd graphs in the non-oracular setting, too, and establishes $\RE$-completeness in both models.

\subsection{Outlook}
\label{subsec_outlook}

The main technical contribution of this work is a spectral method for
establishing non-contextuality of quantum polymorphisms, combining an
equality analysis of Schrijver's bound on the projective-packing number with structural ideas from
Erd\H{o}s--Ko--Rado theory. 
While this method allows us to make progress towards a quantum Hell--Ne\v{s}et\v{r}il classification by settling the complexity of some natural graph classes,
it provably does not capture the whole landscape of $\RE$-complete quantum graph homomorphism problems for two reasons:
\begin{itemize}
\item
Some graphs admit no commutativity gadget and nevertheless give rise to
an $\RE$-complete quantum homomorphism problem. The diamond graph
$K_4$ minus an edge is a simple example: it admits contextual quantum
endomorphisms~\cite[Example~5.12]{culf2025existence}, but it is
homomorphically equivalent to $K_3$, so its quantum homomorphism problem
is $\RE$-complete~\cite{harris2024universality,culf2025re}. The diamond graph is not a core but, as shown below by the \textit{Moser spindle} example, this phenomenon
persists even for cores.
\item
Other graphs admit commutativity gadgets but fail our spectral
criterion. Odd cycles of length at least five provide examples:
they admit commutativity gadgets~\cite{CJM} and are thus $\RE$-complete, but Schrijver's bound
is not tight on them. Consequently, they admit no Schrijver-rigid representation, and~\Cref{thm:criterion} does not apply.
\end{itemize}

\paragraph{The Moser spindle.}
Let $M$ be the Moser spindle~\cite{moser_problems}, obtained from $C_5$ by replacing two
nonadjacent vertices by copies of $K_2$. More explicitly, its vertices
are partitioned into cyclically ordered fibres
\[
 F_0=\{a_0,a_1\},\qquad F_1=\{b\},\qquad
 F_2=\{c_0,c_1\},\qquad F_3=\{d\},\qquad F_4=\{e\},
\]
each fibre inducing a clique, consecutive fibres being completely joined, and with no other edges. An independent set meets at most two fibres and contains at most one vertex from each, so
$\alpha(M)=2$ and $\chi(M)\geq4$. Since $M-v$ is $3$-colourable for
every vertex $v$, the graph is $4$-vertex-critical and hence a core.

Within each fibre of size 2, the two vertices can be swapped
and this can be done in a way that the two swappings on $F_0$ and $F_2$ are not simultaneously diagonalisable. %
More precisely, choose noncommuting projectors $P,R$ on $\C^2$,
and define a quantum function $U$ by
\[
\begin{aligned}
 U_{a_i,a_i}&=P, & U_{a_i,a_{1-i}}&=I-P,\\
 U_{c_i,c_i}&=R, & U_{c_i,c_{1-i}}&=I-R
 \qquad (i\in\{0,1\}),
\end{aligned}
\]
with $U_{x,x}=I$ for $x\in\{b,d,e\}$ and all unlisted entries zero.
It is easy to see that this is a contextual quantum endomorphism, so $M$ admits no commutativity gadget.
Nevertheless, $\qCSP(M)$ is $\RE$-complete: the graph $M'$ that coincides with $M$ except that the edge between $d$ and $e$ is removed admits a q-definition in $M$ (see~\cite{CJM} for the introduction to this concept) and
therefore $\qCSP(M')$ reduces to $\qCSP(M)$, and $M'$ is homomorphically equivalent to $K_3$, whose $\qCSP$ is $\RE$-complete.

\paragraph{Two classification problems.}
These examples distinguish a quantum Hell--Ne\v{s}et\v{r}il classification from a characterisation of commutativity gadgets. The former asks whether $\qCSP(G)$ is $\RE$-complete for every finite
non-bipartite graph $G$; the latter asks which graphs have only non-contextual quantum polymorphisms. 

For the existence of commutativity gadgets, we believe a spectral characterisation remains a possibility (although, as argued above, it must necessarily go beyond our condition in~\Cref{thm:criterion}).
Commutativity gadgets are intrinsically ``global'', in the sense that the same gadget should work for any pair of vertices where commutation is required. This global requirement makes spectral rigidity properties natural
candidates for a characterisation. 
Concretely, we ask the following question:
\begin{question}
\label{q_hierarchy}
Is there a hierarchy of semidefinite relaxations for projective packings such that the existence of commutativity gadgets for graphs is characterised by a suitable rigidity condition at some finite level?
\end{question}

In contrast, quantum-sound reductions  can in principle make use of more general, ``non-uniform'' versions of Ji's commutativity gadgets, whose existence  appears unlikely to be captured via spectral tools alone.
Hence, we expect progress towards a complexity classification to come from combining spectral arguments with combinatorial methods, such as the contextuality bifurcations of~\cite{CJM}. We leave this direction to future work.

\section*{Acknowledgements}
The third, fourth, and fifth authors acknowledge funding from the project “Hamburg
Quantum Computing”, co-financed by ERDF of the European Union and by the Fonds of the Hamburg
Ministry of Science, Research, Equalities and Districts (BWFGB).

\section*{AI statement}
During the preparation of this work, the authors used ChatGPT to assist with literature search, exploration and checking of mathematical arguments, and drafting and revision of the manuscript.
Early partial results concerning Kneser graphs were obtained with the support of ChatGPT and subsequently improved by the authors.
The connection with association schemes, Erd\H{o}s-Ko-Rado theory, Schrijver's theta number, and Roberson's projective packings is due to the authors.
The authors take full responsibility for the results, proofs, references, and final text.

\section{Preliminaries on (multi)linear algebra}
\label{subsec_prelimns_multilinear_algebra}

Throughout the paper, graphs are finite, undirected, and loopless.  
For a positive integer $m$, we write $[m]=\{1,\ldots,m\}$.  
We let $\KG(n,k)$ denote the Kneser graph on the vertex set
$\binom{[n]}k$, in which two vertices are adjacent exactly when the corresponding sets
are disjoint. $\alpha(G)$ denotes the independence number of $G$, i.e., the size of the largest independent set of $G$.

For $m\geq 1$,
the categorical power $G^{m}$ of $G$ has vertex set $X^m$, with
\[
 (x_1,\ldots,x_m)\sim_{G^{m}}(y_1,\ldots,y_m)
 \quad\Longleftrightarrow\quad x_i\sim_Gy_i\quad(i\in[m]).
\]

We let $\1$ denote the all-one vector of suitable size. $J=\1\1^\top$ denotes the square all-one matrix. We denote the zero vector by $\mathbf0$ and the zero matrix of suitable size by $0$. We denote the identity matrix of suitable size by $I$.
Given two square matrices $A$ and $B$, we write $A \preceq B$ (or $B\succeq A$) to indicate that the matrix $B-A$ is positive semidefinite (PSD).
We let $\Tr(A)$ denote the trace of $A$.
For two complex matrices of the same size, their Frobenius inner
product is $\langle A,B\rangle=\operatorname{Tr}(A^*B)$.
We shall use the following standard result.
\begin{lemma}[\protect{\cite[\S2.4 Problem 12(c)]{horn2012matrix}}]
\label{lemma_commutator_horn_johnson}
Let $A,B$ be square complex matrices and suppose that $A$ is diagonalisable. Then $[A,B]=0$ if, and only if, $[A,[A,B]]=0$.
\end{lemma}
\begin{proof}
The forward implication is clear, so we prove the opposite one.
Choose a basis in which
$A=\operatorname{diag}(\lambda_1,\ldots,\lambda_n)$, and write
$B=(b_{ij})$. Then $[A,B]_{ij}=(\lambda_i-\lambda_j)b_{ij}$
and, consequently,
\[
 [A,[A,B]]_{ij}=(\lambda_i-\lambda_j)^2b_{ij}.
\]
The hypothesis therefore gives
$(\lambda_i-\lambda_j)^2b_{ij}=0$ for every $i,j$. If
$\lambda_i=\lambda_j$, then $[A,B]_{ij}=0$; if
$\lambda_i\neq\lambda_j$, then $b_{ij}=0$, and again
$[A,B]_{ij}=0$. Hence, $[A,B]=0$.
\end{proof}

For linear applications $A,B,C$ on the same vector space, the commutator
$[A,B]=AB-BA$ satisfies the \textit{Jacobi identity}
$[A,[B,C]]+[B,[C,A]]+[C,[A,B]]=0$.

Given a complex vector space $V$, its \emph{dual space} is $V^*=\operatorname{Hom}_{\C}(V,\C)$,
the vector space of complex-linear functionals on $V$. Given a subspace $S\subseteq V$,
the \emph{annihilator} of $S$ in $V^*$, denoted $S^0$, is the space of functionals $f\in V^*$
such that $f|_{S}=0$.
If $V$ is finite-dimensional and $W$ is any complex vector space, there is a canonical linear
isomorphism
\[
 \Theta_{V,W}\colon V\otimes W\to\operatorname{Hom}(V^*,W),
 \qquad
 \Theta_{V,W}(v\otimes w)(f)\coloneqq f(v)\,w.
\]
Taking $W=\C$ and identifying $V\otimes\C$ with $V$, this specialises to the canonical
double-dual isomorphism $V\cong V^{**}$, which sends $v\in V$ to the functional
$f\mapsto f(v)$ on $V^*$.

For a finite set $X$, we write $\C^X$ for the space of complex-valued functions on $X$, with its standard Euclidean
inner product denoted by $\langle-,-\rangle$. Given a space $ Z\subseteq\C^X$ and a
finite-dimensional complex vector space $ W$, we shall also identify $ Z\otimes W$
with a space of $ W$-valued functions on $X$ via
\[
  Z\otimes W\to W^X,
 \qquad
 f\otimes w\mapsto\bigl(x\mapsto f(x)\,w\bigr).
\]
Concretely, this map sends $u$ to $x\mapsto\Theta_{ Z, W}(u)(\mathrm{ev}_x)$,
where $\mathrm{ev}_x\in Z^*$ is the evaluation at $x$. Equivalently, if $w_1,\dots,w_d$ is any basis of
$W$, then a function $F\colon X\to W$, written uniquely as
$F(x)=\sum_{i=1}^d f_i(x)w_i$, belongs to $Z\otimes W$ if, and only if,
$f_i\in Z$ for every $i\in[d]$; note that this condition is independent of the chosen
basis of $ W$. Taking $ Z= W=\C^X$, we also obtain the
identification
\[
 \C^X\otimes\C^X\cong\C^{X\times X},
 \qquad (f\otimes g)(x,y)=f(x)g(y).
\]

\section{Tame representations and Schrijver-rigidity}
\label{sec:rigid-schrijver}

In this section, we give a proof of~\Cref{thm:criterion}. First, we shall derive some structural properties of the edge kernels of disjointness representations. Then, we will use them to obtain non-contextuality of quantum polymorphisms for any graph admitting a Schrijver-rigid tame representation.

\subsection{Decomposing the edge kernel}
\label{sect:representations-powers}

First of all, as anticipated in~\Cref{subsec:rigid-schrijver}, we show that any disjointness representation of a graph can be naturally lifted to a representation of its powers.

\begin{lemma}\label{representation-power}
Let $G=(X,E)$ be an undirected graph with a $(P,r)$-disjointness representation $T$.
Then the map $T^{[m]}\colon X^m\to\binom{[m]\times P}{r}$ defined by $T^{[m]}(x_1,\dots,x_m)\coloneqq\{(i,p)\mid p\in T_{x_i}\}$ is a disjointness representation of $G^m$.
\end{lemma}
\begin{proof}
    If $S\coloneqq T^{[m]}(x_1,\dots,x_m)=T^{[m]}(y_1,\dots,y_m)$, then $\{p\mid (i,p)\in S\}=T_{x_i}=T_{y_i}$ so by injectivity of $T$ we get $x_i=y_i$ for all $i\in[m]$.
    It follows that $T^{[m]}$ is injective.
    Furthermore, if $(x_1,\dots,x_m)$ and $(y_1,\dots,y_m)$ are adjacent in $G^m$, then $T_{x_i}\cap T_{y_i}=\emptyset$ for all $i\in[m]$ and, thus, $T^{[m]}(x_1,\dots,x_m)\cap T^{[m]}(y_1,\dots,y_m)=\emptyset$.
    Finally, given an arbitrary feature $(i,p)\in [m]\times P$, there exists $x\in X$ such that $p\in T_x$ and therefore $(i,p)\in T^{[m]}(x,\dots,x)$ so that $[m]\times P=\bigcup_{\bf x} T^{[m]}(\bf x)$.
\end{proof}

Let $\incspace^{[m]}=\incspace(T^{[m]})=\Span\{\chi_{(i,p)} \mid i\in[m],p\in P\}\subseteq\C^{X^m}$ be the incidence space of the disjointness representation $T^{[m]}$.
Recall that $\C^{X^m}\cong (\C^X)^{\otimes m}$ by identifying $f_1\otimes\dots\otimes f_m\in (\C^X)^{\otimes m}$ with the vector whose component indexed by $(x_1,\dots,x_m)$ is $\prod_{i=1}^m f_i(x_i)$.
Let $\iota_i$ be the embedding $\C^{X}\to\C^{X^m}$ defined by $\iota_i(f)=\1^{\otimes(i-1)}\otimes f\otimes\1^{\otimes(m-i)}$.
As a vector, the coefficient of $(x_1,\dots,x_m)$ in $\iota_i(f)$ is therefore simply $f(x_i)$.
We also use the notation $\iota_i$ for the induced embedding $\C^{X}\otimes\C^{X}\to \C^{X^m}\otimes\C^{X^m}$.
Note that we have $\chi_{(i,p)} = \iota_i(\chi_p)$ and that $\iota_i$ preserves orthogonality, from which the next lemma follows.
\begin{lemma}\label{L-power}
    $\incspace^{[m]} = \C\1\oplus\bigoplus_{i=1}^m U_i$, where $U_i=\iota_i(U)$.
\end{lemma}
\begin{proof}
We have $\incspace^{[m]}=\Span\{\chi_{i,p}\mid i\in[m],p\in P\}=\sum \iota_i(\incspace)=\C\1_{X^m}+\sum\iota_i(U)$
and it remains to prove that the sum is direct.
Suppose that $\lambda\1_{X^m}+\sum_{i=1}^m\iota_i(u_i)=0$ for some $u_1,\dots,u_m\in U$.
Fix $j\in[m]$.
    Note that for $i\neq j$ we have $\langle \iota_i(u_i),\iota_j(u_j)\rangle = \sum_{\bx\in X^m}u_i(x_i)u_j(x_j) = |X|^{m-2}(\sum_x u_i(x))(\sum_x u_j(x))=0$
    since $u_i$ is orthogonal to $\C\1_X$.
Similarly, $\langle\lambda\1_{X^m},\iota_j(u_j)\rangle=0$.
It follows that $\langle \iota_j(u_j),\iota_j(u_j)\rangle =0$ and thus $\iota_j(u_j)=0$.
\end{proof}

We next show that for all $m\geq 1$, the edge kernel $\edgekernel^{[m]}=\edgekernel(T^{[m]})$
can be decomposed into $m$ summands from the edge kernel $\edgekernel$.

\lempowerkernel*

\begin{proof}
Write $\vec E=\{(x,y) \mid x\sim_Gy\}$ for the set of oriented edges of $G$ and set $R=\C^{\vec E}$. Consider the
linear maps $\alpha,\beta\colon U\to R$ defined by
$\alpha(u)(x,y)=u(x)$ and $\beta(u)(x,y)=u(y)$ for $u\in U$ and
$(x,y)\in\vec E$. Since $G$ has no isolated vertices, both maps are
injective.

We first show that $\C\1_{\vec E}+\alpha(U)+\beta(U)$ is a direct sum.
Suppose that $c\1_{\vec E}+\alpha(u)+\beta(v)=\bzero$ for some
$c\in\C$ and $u,v\in U$. 
Suppose that $x,x'$ are joined by a walk of length $2$, say $x\sim _G y\sim_G x'$.
Then evaluating the previous identity at $(x,y)\in\vec E$ and $(x',y)\in\vec E$, we obtain $c\1_{\vec E}+u(x)+v(y)=c\1_{\vec E}+u(x')+v(y)$ and therefore $u(x)=u(x')$.
Iterating this argument shows that $u(x)=u(x')$ whenever $x,x'$ are connected by a walk of even length.
Since $G$ is connected and non-bipartite, any two vertices are connected by such a walk and therefore $u$ is constant on $X$.
Since $U\perp\1$, it follows that $u=0$.
Similarly, we get $v=0$ and then $c=0$, hence the sum is direct.

Choose a linear complement $Z$ so that
\begin{align}\label{eqn_1138_1808}
 R=\C\1_{\vec E}\oplus\alpha(U)\oplus\beta(U)\oplus Z.
\end{align}
Let $\mathcal S=\{\C\1_{\vec E},\alpha(U),\beta(U),Z\}$. Taking the
$m$-fold tensor product of this decomposition %
gives the direct-sum decomposition
\begin{align}
\label{eqn_1621_1408}
 \C^{\vec E^m}
 \cong R^{\otimes m}
 =
 \bigoplus_{(S_1,\ldots,S_m)\in\mathcal S^m}
 \left(\bigotimes_{k=1}^m S_k\right).
\end{align}

Recall that
$\incspace^{[m]}=\C\1\oplus\bigoplus_{i\in[m]}U_i$ by~\Cref{L-power}. Distributing the
tensor product
$\incspace^{[m]}\otimes\incspace^{[m]}$ over these
direct sums shows that any $F\in\edgekernel^{[m]}\subseteq\incspace^{[m]}\otimes\incspace^{[m]}$ admits a unique expression 
\begin{align}
\label{eqn_1722_1408}
 F
 =
 \lambda J+\sum_{i\in[m]}A_i\otimes\1_{X^m}+\sum_{j\in[m]}\1_{X^m}\otimes B_j
   +\sum_{i,j\in[m]}H_{ij},
\end{align}%
where $\lambda\in\C$, $A_i\in U_i,B_j\in U_j$, and
$H_{ij}\in U_i\otimes U_j$.

For $K\in\C^{X^m}\otimes C^{X^m}$, define $\rho(K)$ to be $\rho(K)(e_1,\dots,e_m)\coloneqq K(\bf x,\bf y)$ where $e_i=(x_i,y_i)$ for all $i\in[m]$.
Write $\hat F=\rho(F)$ and use analogous notation for the restrictions of the terms in \cref{eqn_1722_1408}.
Since $\bx\sim_{G^m}\by$ exactly when $(x_k,y_k)\in\vec E$ for every
$k\in[m]$, the hypothesis of the lemma says that $\hat F=0$.

Fix distinct $i,j\in[m]$. Let $P_{ij}$ be the direct summand in
\cref{eqn_1621_1408} having $\alpha(U)$ in coordinate $i$,
$\beta(U)$ in coordinate $j$, and $\C\1_{\vec E}$ in every other
coordinate. Let $\pi_{ij}$ denote the coordinate projection onto
$P_{ij}$ associated with the direct-sum decomposition
in~\cref{eqn_1621_1408}.

We claim that $\pi_{ij}(\hat F)=\hat H_{ij}$.
Indeed, fix $a,b\in[m]$ and choose a rank-1 tensor decomposition
$H_{ab}=\sum_{\ell=1}^{N}
\iota_a(u_{\ell})\otimes \iota_b(v_{\ell})$, where
$u_{\ell},v_{\ell}\in U$.
For $\be=(e_1,\ldots,e_m)\in\vec E^m$, with
$e_i=(x_i,y_i)$ for $i\in[m]$, we have
\begin{align}
 \hat H_{ab}(\be)
 &= \sum_{\ell=1}^{N}\iota_a(u_\ell)(\bx)\iota_b(v_\ell)(\by)\notag\\
 &= \sum_{\ell=1}^{N} u_\ell(x_a)v_\ell(y_b)\notag\\
 \label{eqn_1703_1408a}&=\sum_{\ell=1}^{N}
 \alpha(u_{\ell})(e_a)
 \beta(v_{\ell})(e_b).
\end{align}
If $a\neq b$ (say, $a<b$), this shows that $\hat H_{ab}=\sum_{\ell=1}^N \1^{\otimes(a-1)}\otimes\alpha(u_\ell)\otimes\1^{\otimes (b-a-1)}\otimes\beta(v_\ell)\otimes\1^{\otimes(m-b)}$ and thus $\hat H_{ab}$ belongs to $P_{ab}$.
Its component in $P_{ij}$ is therefore zero unless $(a,b)=(i,j)$.
Suppose now $a=b$. 
By~\cref{eqn_1703_1408a}, $\hat H_{aa}(\be)$ depends only on $e_a$, so
\begin{align*}
    \hat H_{aa}\in\C\1_{\vec E}^{\otimes a-1}\otimes R\otimes\C\1_{\vec E}^{\otimes m-a}.
\end{align*}
This shows that 
every summand occurring in $\hat H_{aa}$ has
$\C\1_{\vec E}$ in every tensor-coordinate other than $a$. 
Hence
$\pi_{ij}(\hat H_{aa})=0$, because every nonzero element of $P_{ij}$ has nonconstant factors in
the two distinct coordinates $i$ and $j$.

Similarly, $\rho(\lambda\1_{X^m\times X^m})\in\C\1_{\vec E^m}$, $\rho(A_a\otimes\1_{X^m})\in \C\1_{\vec E}^{\otimes a-1}\otimes R\otimes\C\1_{\vec E}^{\otimes m-a}$, and $\rho(\1_{X^m}\otimes B_b)\in \C\1_{\vec E}^{\otimes b-1}\otimes R\otimes\C\1_{\vec E}^{\otimes m-b}$.
None of these terms has a component in $P_{ij}$.
It follows that $\pi_{ij}(\hat F)=\hat H_{ij}$.

Let
$\iota_{ij}\colon\alpha(U)\otimes\beta(U)\to P_{ij}$
be the canonical isomorphism obtained by inserting
$\1_{\vec E}$ in every coordinate other than $i$ and $j$. By
\cref{eqn_1703_1408a}, we have $0
 =\pi_{ij}(\hat F)
 =\hat H_{ij}
 =\iota_{ij}\!\left(
   (\alpha\otimes\beta)(H_{ij})
   \right).$
The maps $\alpha$ and $\beta$ are injective, and tensor products of
injective linear maps between complex vector spaces are injective.
Since $\iota_{ij}$ is also injective, we conclude that $H_{ij}=0$.
This holds for every $i\neq j$.

For each $i\in[m]$, define
$h_i=A'_i\otimes\1_X+\1_X\otimes B'_i+H'_{ii}$,
where $A'_i, B'_i\in U$ and $H'_{ii}\in U\otimes U$ are such that $A_i=\iota_i(A'_i)$, $B_i=\iota_i(B'_i)$, and $ H_{ii}=\iota_i(H'_{ii})$.
Then $h_i\in \incspace\otimes\incspace$, and
\cref{eqn_1722_1408} reduces to
\begin{align}
\label{eqn_1728_14_08}
 F=\lambda J+\sum_{i\in[m]}\iota_i(h_i).
\end{align}
For an element $k$ of $\C^X\otimes\C^X$, define similarly as above $\hat k\in R$ to be $\hat k(e) = k(x,y)$ for every $e=(x,y)\in\vec E$.
Note that $(\rho\circ\iota_i)(k)(e_1,\dots,e_m)=\hat k(e)$ for all $e_1,\dots,e_m\in\vec E$.
We claim that each $\hat h_i$ is constant on $\vec E$.
Fix $i\in[m]$ and two ordered edges $e,e'\in\vec E$.
For each $a\neq i$, choose an ordered edge $e_a\in\vec E$. Evaluate
\cref{eqn_1728_14_08} on the tuples $(e_1,\dots,e_{i-1},e,e_{i+1},\dots,e_m)$ and $(e_1,\dots,e_{i-1},e',e_{i+1},\dots,e_m)$.
This gives
$\lambda+\hat h_i(e)+\sum_{a\neq i}\hat h_a(e_a)=0$ and
$\lambda+\hat h_i(e')+\sum_{a\neq i}\hat h_a(e_a)=0$, since $\rho(F)=0$ by assumption.
Subtracting gives $h_i(e)=h_i(e')$.
Thus there is some $\lambda_i\in\C$ such that $\hat h_i(e)=\lambda_i$ for every $e\in\vec E$.
Evaluating \cref{eqn_1728_14_08} on any $(e_1,\dots,e_m)\in\vec E$ gives $\lambda+\sum_{i\in[m]}\lambda_i=0$.

Define
$D_i=h_i-\lambda_i(\1_X\otimes\1_X)$.
Then $D_i\in \incspace\otimes\incspace$ and $D_i(x,y)=0$ whenever $x\sim_G y$, thus $D_i\in\edgekernel$.
Moreover,
\[
 F
 =\lambda\1_{X^m\times X^m}+\sum_{i\in[m]}\iota_i\bigl(D_i+\lambda_i(\1_X\otimes\1_X)\bigr)
 =\sum_{i\in[m]}\iota_i(D_i),
\]
which proves that $\edgekernel^{[m]} = \Span \bigcup_i\iota_i(\edgekernel)$.

We now argue that the sum $\bigoplus \iota_i(\edgekernel)$ is indeed direct.
Suppose that
\begin{align}
    \label{sum_is_direct}
    \sum_i\iota_i(D_i)=0
\end{align} for some $D_1,\dots,D_m\in\edgekernel$.
Fix edges $e_a=(x_a,y_a)$ for all $a\in[m]$.
Let $i\in[m]$, and $x,x',y,y'\in X$ be arbitrary.
Let $\bx=(x_1,\dots,x_{i-1},x,x_{i+1},\dots,x_m)$ and $\bx'=(x_1,\dots,x_{i-1},x',x_{i+1},\dots,x_m)$ and similarly for $\by$ and $\by'$.
Considering the coefficients of the tensor in~\cref{sum_is_direct} at $(\bx,\by)$ and $(\bx',\by')$, we obtain $D_i(x,y)=D_i(x',y')$ and thus $D_i$ is a constant tensor.
Moreover, since $D_i\in\edgekernel$, it must be the zero tensor, which concludes the proof.
\end{proof}

A way to understand~\Cref{lem:power-kernel} in elementary terms is the following.
Suppose that $F\colon X^m\times X^m\to\C$ is a function such that $F(\bx,\by)=0$ whenever $\bx$ and $\by$ are adjacent in $G^m$, and such that $F$ is a linear combination of functions of the form $(\bx,\by)\mapsto \indfunction{p\in T_{x_i}}\indfunction{q\in T_{y_j}}$ for $i,j\in[m]$ and $p,q\in P$.
Then there exist $F_1,\dots,F_m\colon X\times X\to\C$ such that $F_i(x,y)=0$ for every $i\in[m]$ whenever $x$ and $y$ are adjacent in $G$; each is in the span of the functions $(x,y)\mapsto \indfunction{p\in T_x}\indfunction{q\in T_y}$ for $p,q\in P$; and  $F(\bx,\by)=\sum_{i=1}^mF_i(x_i,y_i)$ holds.
We also note the following simple corollary of~\Cref{lem:power-kernel}.

\begin{corollary}\label{cor:power-kernel}
Let $\mathcal W$ be a vector space, and suppose that $G$ is connected and non-bipartite.
Then $\edgekernel^{[m]}\otimes\mathcal W = \bigoplus_{i=1}^m \iota_i(\edgekernel)\otimes\mathcal W$.
\end{corollary}

\subsection{Two more edge-kernel lemmas}
We next prove two simple facts on the edge kernel of a disjointness representation, which shall be useful later in~\Cref{section_examples}.
They are not needed for the proof of~\Cref{thm:criterion}, so readers
who wish to proceed directly to that proof may skip ahead
to~\Cref{sec_non_contextuality_schrijver}.

\begin{lemma}
\label{lem:local-incidence-edge-kernel}
Let $T$ be a $(P,r)$-disjointness
representation of a graph $G$. 
Suppose that $(i)$ for every $y\in X$ the functions
in $\incspace$ supported on $X\setminus N_G(y)$ are precisely
$\Span\{\chi_p \mid p\in T_y\}$; $(ii)$ every
$S_p$ contains at least three vertices; and $(iii)$
$T_x\cap T_y=\{p\}$ for any distinct $x,y\in S_p$.  Then the edge kernel of $T$ is
\[
 \edgekernel
   =\Span\{\chi_p\otimes\chi_p \mid p\in P\}.
\]
In particular, every matrix in $\edgekernel$ is symmetric.
\end{lemma}

\begin{proof}
Observe first that, for each $p\in P$ and each pair $x,y$ of adjacent vertices of $G$, we have $(\chi_p\otimes\chi_p)(x,y)=\chi_p(x)\chi_p(y)=0$ since every feature fibre of $T$ is independent.
Hence, $\Span\{\chi_p\otimes\chi_p \mid p\in P\}\subseteq\edgekernel$.
Conversely, take $F\in\edgekernel$. For each $x,y\in X$, consider the functions $\vartheta_y,\varphi_x\in\C^X$ defined by $\vartheta_y(z)=F(z,y)$ and $\varphi_x(z)\coloneqq F(x,z)$. Since $F\in\edgekernel$, both $\vartheta_y$ and $\varphi_x$ belong to $\incspace$. Moreover, $\vartheta_y$ is supported on $X\setminus N_G(y)$ and
$\varphi_x$ is supported on $X\setminus N_G(x)$.
Hence, by $(i)$, there exist coefficients $c_p(y)$ and $e_p(x)$ for $p\in P$ such that
\begin{align*}
\vartheta_y=\sum_{p\in T_y}c_p(y)\chi_p
\qquad\mbox{and}\qquad
    \varphi_x=\sum_{p\in T_x}e_p(x)\chi_p.
\end{align*}
It follows that
\begin{align*}
 F(x,y)&=
 \vartheta_y(x)
 =
 \sum_{p\in T_y}c_p(y)\chi_p(x)
 =
 \sum_{p\in T_x\cap T_y}c_p(y)\\
 &=
 \varphi_x(y)
       =\sum_{p\in T_x}e_p(x)\chi_p(y)
       =
      \sum_{p\in T_x\cap T_y}e_p(x). 
\end{align*}
Now fix some $p\in P$. Pick any $z\neq z'\in S_p$, and choose $z''\in S_p\setminus\{z,z'\}$, which exists by $(ii)$. Using $(iii)$, we find $c_p(z)=F(z'',z)=e_p(z'')=F(z'',z')=c_p(z')$ and, thus, $c_p$ is constant over $S_p$. By the same argument, $e_p$ is constant over $S_p$. Moreover, the two values are equal; call them $\gamma_p$.
As a consequence, 
\[
 F(x,y)=\sum_{p\in T_x\cap T_y}\gamma_p
       =\sum_{p\in P}\gamma_p\chi_p(x)\chi_p(y),
\]
whence we deduce that $F\in\Span\{\chi_p\otimes\chi_p \mid p\in P\}$.
\end{proof}

\begin{lemma}
\label{lem:diagonal-feature-rigidity}
Let $T$ be a $(P,r)$-disjointness
representation of a graph $G$.   Suppose that its
feature functions are linearly independent and $\edgekernel
   =\Span\{\chi_p\otimes\chi_p \mid p\in P\}$.
Then $\diagonalSubspace=\C\omega_T$, where  $\omega_T\in\C^{X\times X}$ is the matrix defined by $\omega_T(x,y)\coloneqq 
|T_x\cap T_y|$.  Moreover, if every vertex of $G$ has degree at least two, then $T$ is tame.
\end{lemma}

\begin{proof}
If
$F=\sum_pc_p\chi_p\otimes\chi_p$ has constant diagonal $\lambda$, then
the identity~\eqref{eqn_1648_0308} gives
$\sum_p(c_p-\lambda/r)\chi_p=\bzero$.  Linear independence therefore gives
$c_p=\lambda/r$ for every $p$, and hence $F=\frac{\lambda}{r}\sum_p\chi_p\otimes\chi_p=\frac{\lambda}{r}\omega_T$.
This shows that $\diagonalSubspace\subseteq\C\omega_T$, and the converse inclusion is trivial.

Suppose now that the minimum degree of $G$ is at least two.
Note first that $\omega_T$ is symmetric.
Fix a vertex $u$ and choose distinct $a,b\in N_G(u)$. 
Suppose, for the sake of contradiction, that $\Delta_u(\omega_T)$ is $0$, so that in particular $\omega_T(a,b)=\omega_T(a,a)$ and therefore $|T_a\cap T_b|=|T_a|$. Since each $T_x$ has size $r$, this means that $T_a=T_b$, which contradicts injectivity of $T$. Hence, $\edgekernel$ satisfies both (\ref{tag:R1}) and (\ref{tag:R2}), so $T$ is tame.
\end{proof}

\subsection{Equality in Roberson's bound and global commutativity for tame representations}
\label{sec_non_contextuality_schrijver}

A key ingredient of our proof of~\Cref{thm:criterion} is a tightness analysis of Roberson's
bound in~\Cref{thm_alpha_p_schrijver}, which describes the case where the inequality involving the projective packing number and Schrijver theta is crisp.
Recall the identification 
\begin{align}
\label{eqn_1604_1408}
 \mathbb C^{X}\otimes W
 \cong\{ W\text{-valued functions on }X\}
\end{align}
described in~\Cref{subsec_prelimns_multilinear_algebra}.

\lemoperatorthetarigidity*

\begin{proof}
Let $d=\dim H$,
let $t=\operatorname{val}(R)=\frac{1}{d}\sum_{x\in X}\Tr(R_x)$, and consider the $X\times X$
matrix $M$ whose $xy$-entry is $M_{xy}=\frac{1}{d}\Tr(R_xR_y)$.
Set also
$P=M/t$. It was proved in~\cite[Lemma~6.8.2]{roberson2013variations} that $P$ is feasible for the primal SDP formulation of
$\vartheta^-(G)$, in that it satisfies $\operatorname{Tr}(P)=1$,
$P_{xy}=0$ whenever $x\sim_Gy$, $P\geq0$ entrywise, and
$P\succeq0$. Moreover, the same proof shows that
$\langle J,P\rangle\geq t$.
On the other hand, since $Q$ is a dual $\vartheta^-$ certificate at value $t$, weak
duality gives $\langle J,P\rangle\leq t$. Consequently,
$\langle J,P\rangle=t$, so $P$ and $Q$ are both optimal.
Complementary slackness then gives $\langle Q,P\rangle=0$, and hence
$\langle Q,M\rangle=0$.

Consider now the tensor $\Xi\in\C^X\otimes\End(H)$ associated with $R$; explicitly, $\Xi=\sum_{x\in X}\be_x\otimes R_x$. 
Taking the inner product on $\C^X\otimes\End(H)$ given by the tensor product of the standard Euclidean inner product $\langle -, -\rangle$ on $\C^X$ and the usual Frobenius inner product
$\langle A,B\rangle_H=\Tr(A^*B)$ on $\End(H)$, we find
\begin{align}
\label{eqn_1923_1408}
\notag
 \left\langle
 \Xi,
 \bigl(Q\otimes\operatorname{id}_{\End(H)}\bigr)\Xi
 \right\rangle
 &=
 \sum_{x,y\in X}\left\langle
 \be_x\otimes R_x,
 \bigl(Q\otimes\operatorname{id}_{\End(H)}\bigr)\be_y\otimes R_y
 \right\rangle\\
 \notag
&=
 \sum_{x,y\in X}\langle
 \be_x\otimes R_x,
 Q\be_y\otimes R_y\rangle\\
 &=
 \sum_{x,y\in X}Q_{xy}\Tr(R_xR_y)
 =
 d\langle Q,M\rangle
 =
 0.
\end{align}
Since $Q\succeq0$, the operator
$Q\otimes\operatorname{id}_{\End(H)}$ is positive semidefinite.
The vanishing of its quadratic form in~\cref{eqn_1923_1408} therefore implies
\[
 \Xi\in
 \ker\!\left(Q\otimes\operatorname{id}_{\End(H)}\right)
 =
 (\ker Q)\otimes\End(H),
\]
as required.
\end{proof}

The next lemma shall be used to prove that the projectors in a perfect quantum strategy associated with an $m$-ary quantum polymorphism (see~\Cref{defn_quantum_polymorphism} in~\Cref{subsec_proving_main_technical_thm}) must be fully commuting whenever the first factor of each ``section'' of the strategy (i.e., each function $\mathbf x\mapsto P_{\mathbf x,p}$) belongs to the $m$-dimensional incidence space $\incspace^{[m]}$ of a tame disjointness representation.

\lemdegreeonecommutativity*

\begin{proof}
Recall that $\diagonalSubspace$ is the subspace of the edge kernel $\edgekernel$ consisting of all $F\in\edgekernel$ whose diagonal is constant.
Let $\diagonalSubspace^*$ be the dual space of $\diagonalSubspace$, and consider the linear functional $\delta\in\diagonalSubspace^*$ that associates to $F\in\diagonalSubspace$ the common diagonal value of $F$. Also, for $a,b\in X$, consider the linear functional $\varepsilon_{a,b}\in\diagonalSubspace^*$ defined by $\varepsilon_{a,b}(F)\coloneqq F(a,b)$.
Thus, $\delta=\varepsilon_{a,a}$ and the choice does not depend on $a\in X$.
For fixed $x\in X$, consider the space
\[
 L_x=\Span\{\delta-\varepsilon_{a,b} \mid a,b\in N_G(x)\}
       \subseteq\diagonalSubspace^*.
\]
Recall also the definition of the operator $\Delta_x\colon \diagonalSubspace\to\C^{N_G(x)\times N_G(x)}$ given in~\cref{eqn_18_08_1510}, and note that $\Delta_x(-)_{a,b}=\delta-\varepsilon_{a,b}$. 
The annihilator of $L_x$ in $\diagonalSubspace^{**}\simeq\diagonalSubspace$ is the space
\begin{align*}
    L_x^0
    &=
    \{\phi\in\diagonalSubspace^{**}\mid\phi|_{L_x}=0\}\\
    &=
    \{\phi\colon\diagonalSubspace^*\to\C \mid \phi(\delta-\varepsilon_{a,b})=0\;\forall a,b\in N_G(x)\}\\
    &
    \simeq\{F\in \diagonalSubspace\mid (\delta-\varepsilon_{a,b})(F)=0\;\forall a,b\in N_G(x)\}\\
    &=
    \{F\in \diagonalSubspace\mid (\Delta_x F)_{a,b}=0\;\forall a,b\in N_G(x)\}\\
    &=
    \{F\in \diagonalSubspace\mid \Delta_x F=0\}\\
    &=\ker\Delta_x.
\end{align*}

Since $\diagonalSubspace$ is
finite-dimensional,  $L_x^0$ is isomorphic to the orthogonal complement of $L_x$ in $\diagonalSubspace^*$. The fact that the representation is tame requires that $\Delta_x$ is injective for each $x$; i.e., $\ker\Delta_x=\{0\}$. It follows that $L_x^0=\{0\}$, so $L_x=\diagonalSubspace^*$.

Next, fix $p,q\in P$, and consider the commutator map $C_{p,q}\colon X^m\times X^m\to\End(H)$ defined by   
$C_{p,q}(\mathbf x,\mathbf y)
       \coloneqq[P_{\mathbf x,p},P_{\mathbf y,q}]$. We claim that $C_{p,q}$
belongs to
$\incspace^{[m]}\otimes\incspace^{[m]}\otimes\End(H)$.
To see why, observe that by condition $1.$ there exist finite collections
$(f_i)_{i\in I},(g_j)_{j\in J}\subseteq\incspace^{[m]}$ and $(A_i)_{i\in I},(B_j)_{j\in J}\subseteq\End(H)$ such that
$P_{\mathbf x,p}=\sum_{i\in I} f_i(\mathbf x)A_i$ and
$P_{\mathbf y,q}=\sum_{j\in J} g_j(\mathbf y)B_j$ for each $\bx,\by\in X^m$.
By bilinearity of the commutator,
\[
 C_{p,q}(\mathbf x,\mathbf y)
 =[P_{\mathbf x,p},P_{\mathbf y,q}]
 =\sum_{i,j}f_i(\mathbf x)g_j(\mathbf y)[A_i,B_j].
\]
This means that $C_{p,q}$ admits the decomposition 
\begin{align*}
    C_{p,q}=\sum_{i,j}f_i\otimes g_j\otimes[A_i,B_j]\in\incspace^{[m]}\otimes\incspace^{[m]}\otimes \End(H),
\end{align*}
which proves the claim. 
Furthermore, condition $3.$ states that $C_{p,q}(\mathbf x,\mathbf
y)=0$ whenever $\mathbf x\sim_{G^m}\mathbf y$.
It follows that $C_{p,q}\in\edgekernel^{[m]}\otimes\End(H)$.
We can then apply~\Cref{cor:power-kernel}
with $\mathcal W=\End(H)$, %
from which we get that there exist $D^{p,q}_i\in\edgekernel\otimes\End(H)$ for $i\in[m]$ such that $C_{p,q}=\sum_i(\iota_i\otimes\mathrm{id}_{\mathcal W})(D^{p,q}_i)$.
Thus,
\begin{align}
\label{eqn_1622_1708}
 [P_{\mathbf x,p},P_{\mathbf y,q}]
 =\sum_{i=1}^mD_i^{p,q}(x_i,y_i).
\end{align}
Furthermore, condition $2.$ implies that
 $\sum_iD_i^{p,q}(x_i,x_i)=0$ for each $\bx\in X^m$.
We can now vary one coordinate at a time to see that each $D_i^{p,q}$ has
constant diagonal.
Indeed, for any $i\in[m]$, choose $\bx,\by\in X^m$ that only differ at position $i$.
Evaluating the previous sum at $\bx$ and $\by$ and subtracting, we get $D_i^{p,q}(x_i,x_i)=D_i^{p,q}(y_i,y_i)$.
Thus, $D_i^{p,q}\in\diagonalSubspace\otimes\End(H)$.

If $p=q$, condition~\textup{(\ref{tag:R1})} makes the right side of~\cref{eqn_1622_1708}
symmetric in $\mathbf x,\mathbf y$, whereas the left side is
antisymmetric.
Hence, 
\begin{align}
\label{eqn_1647_1708}
[P_{\mathbf x,p},P_{\mathbf y,p}]=0
\end{align}
for each $\mathbf x,\mathbf y\in X^m$.

Suppose now that $p\neq q$.
Let $\Theta(D_i^{p,q})\colon\diagonalSubspace^*\to\End(H)$ correspond to $D_i^{p,q}$ under the  canonical identification
$\Theta\colon\diagonalSubspace\otimes\End(H)\cong\operatorname{Hom}(\diagonalSubspace^*,\End(H))$ described in~\Cref{subsec_prelimns_multilinear_algebra}, so that in particular
\begin{align}
\label{eqn_1709_1708}
    \Theta(D_i^{p,q})(\varepsilon_{a,b})
    =
    D_i^{p,q}(a,b).
\end{align}
Fix a tuple $\mathbf z\in X^m$ and a coordinate $j\in[m]$. 
We claim that
\begin{align}
    \label{eqn_claim_1756_1708}
    [P_{\bz,p},D_j^{p,q}(a,b)]=0
\end{align}
for each $a,b\in X$.
For any $\bx,\by\in X^m$,~\cref{eqn_1647_1708} and the Jacobi
identity give
\begin{align}
    \label{eqn_1726_1708}
 [P_{\mathbf z,p},[P_{\mathbf x,p},P_{\mathbf y,q}]]
 = [P_{\mathbf x,p},[P_{\mathbf z,p},P_{\mathbf y,q}]].    
\end{align}
Choose now $a,b\in N_G(z_j)$.
By definition of the edge kernel, $\varepsilon_{z_j,a}-\varepsilon_{z_j,b}$ is the zero functional on $\diagonalSubspace$, and thus
\begin{align}
    \label{eqn_1724_1708}
\Theta(D_j^{p,q})(\varepsilon_{z_j,a}-\varepsilon_{z_j,b})=0.
\end{align}
Take two tuples $\by,\by'$ such that $y_j=a$, $y_j'=b$, and $y_i=y_i'$ for each $i\neq j$.
We find
\begin{align}
\label{eqn_xxxx_xxxx}
    [P_{\bz,p},P_{\by,q}]-[P_{\bz,p},P_{\by',q}]
    &=
    \sum_{i\in[m]}(D_i^{p,q}(z_i,y_i)-D_i^{p,q}(z_i,y_i'))
    =
    \sum_{i\in[m]}\Theta(D_i^{p,q})(\varepsilon_{z_i,y_i}-\varepsilon_{z_i,y_i'})=0
\end{align}
by~\cref{eqn_1724_1708}, so that $[P_{\bz,p},P_{\by,q}]=[P_{\bz,p},P_{\by',q}]$.
Choosing $\bx$ so that $x_j=a$ and using that $\epsilon_{a,a}=\delta$, we deduce from the above and~\eqref{eqn_1726_1708} that 
\begin{align*}
    [P_{\mathbf z,p},[P_{\mathbf x,p},P_{\mathbf y,q}]]
    =
    [P_{\mathbf z,p},[P_{\mathbf x,p},P_{\mathbf y',q}]]
\end{align*}
and, thus, by~\cref{eqn_xxxx_xxxx},
\[
 [P_{\mathbf z,p},
       \Theta(D_j^{p,q})(\delta-\varepsilon_{a,b})]=0.     
\]
Recalling that $L_{z_j}=\diagonalSubspace^*$, this means that
$[P_{\mathbf z,p},\Theta(D_j^{p,q})(\varepsilon)]=0$ for each $\varepsilon\in\diagonalSubspace^*$. Applying~\cref{eqn_1709_1708} with $\varepsilon=\varepsilon_{a,b}$ for arbitrary $a,b\in X$, this proves the claimed~\cref{eqn_claim_1756_1708}.

Combining~\cref{eqn_1622_1708} and~\cref{eqn_claim_1756_1708}, we deduce that 
$[P_{\mathbf x,p},P_{\mathbf y,q}]$ commutes with $P_{\mathbf x,p}$ for each $\bx,\by\in X^m$. Since projectors are diagonalisable, we can conclude via~\Cref{lemma_commutator_horn_johnson} that $[P_{\mathbf x,p},P_{\mathbf y,q}]=0$, and the proof is finished.
\end{proof}

\subsection{Proving Theorem~\ref{thm:criterion}}
\label{subsec_proving_main_technical_thm}

We can finally complete the proof of the main technical result of this paper, which we restate below for convenience.

\thmcriterion*

First, we describe the needed polymorphic tools.

\begin{definition}[\cite{CJM}]
\label{defn_quantum_polymorphism}
    Let $G=(X,E)$ be a graph. An $m$-ary \textit{quantum polymorphism} of $G$, for $m\geq 1$, is a quantum homomorphism $Q:G^m\qto G$. Explicitly, $Q$ is a quantum function $X^m\qto X$ satisfying the following two conditions:
\begin{itemize}
    \item $Q_{(x_1,\dots,x_m),x}Q_{(y_1,\dots,y_m),y}=0$ whenever $\{x_i,y_i\}\in E$ for each $i\in [m]$ and $\{x,y\}\not\in E$;
    \item
    $[Q_{(x_1,\dots,x_m),x},Q_{(y_1,\dots,y_m),y}]=0$ for each $x,y\in X$ whenever $\{x_i,y_i\}\in E$ for each $i\in [m]$.
\end{itemize} 
A quantum polymorphism is \emph{non-contextual} if all its
projectors commute.
\end{definition}

\begin{definition}[\protect{\cite{ji2013binary}; see also~\cite{culf2025re,culf2025existence,CJM}}]\label{comm-gadget}
    Fix a graph $H=(X_H,E_H)$.
    A \emph{commutativity gadget} for $H$ is a graph $G=(X_G,E_G)$ with two distinguished vertices $u,v\in X_G$ satisfying the following properties:
    \begin{itemize}
        \item For any quantum function $Q\colon\{u,v\}\qto X_H$ such that $Q_{u,a}$ and $Q_{v,b}$ commute for all $a,b\in X_H$, % 
        there exists a quantum homomorphism $Q'\colon G\qto H$ extending $Q$.
        \item For every quantum homomorphism $Q\colon G\qto H$ and all $a,b\in X_H$, it holds that $[Q_{u,a},Q_{v,b}]=0$.
    \end{itemize}
\end{definition}

We shall use the following result, which holds in the more general setting of arbitrary relational structures.
\begin{theorem}[\cite{CJM}]
    A graph $G$ admits a commutativity gadget if, and only if, all of its quantum polymorphisms are non-contextual.
\end{theorem}

Non-bipartite graphs give rise to NP-complete CSPs~\cite{hell1990complexity}.
In particular, unconditionally on $\operatorname{P}\neq\operatorname{NP}$, it was shown in~\cite{bulatov2005h} that any non-bipartite graph pp-constructs $\operatorname{3-SAT}$; see, e.g.,~\cite{barto2017polymorphisms} for the terminology. Combining this with~\cite[Corollary~46]{CJM} and~\cite[Theorem~50]{CJM} we have that, if $G$ additionally admits a commutativity gadget, then $\qCSP(\operatorname{3-SAT})$ (which is known to be $\RE$-complete by~\cite{ji2021mip,mastel2024two}) reduces in logspace to $\qCSP(G)$. As a consequence, we obtain the following.

\begin{corollary}
\label{thm_nonbipartite_noncontextual_undecidable}
If all quantum polymorphisms of a non-bipartite graph $G$ are non-contextual, then $\qCSP(G)$ is $\RE$-complete. 
\end{corollary}

\begin{proof}[Proof of~\Cref{thm:criterion}]
Let $T$ be a tame Schrijver-rigid $(P,r)$-disjointness representation of $G$.
Recall that $v=|X|$, $\mu=\frac{r}{|P|}$, and $S_p$ denotes the feature fibre of $p\in P$.                   
Fix $m\geq1$, and let $Q\colon G^{m}\qto G$ be an $m$-ary quantum
polymorphism of $G$ over a nontrivial Hilbert space $H$. For $\bx\in X^m$ and $p\in P$, define the operators
\[
       P_{\mathbf x,p}
       \coloneqq\sum_{a\in S_p}Q_{\mathbf x,a}.                    
\]
Each $P_{\mathbf x,p}$ is a projector, since it is a sum of projectors lying in a single PVM.
Note also that $[P_{\bx,p},P_{\bx,q}]=0$ for each $p,q\in P$.
Moreover, since $|T_a|=r$ for
every $a\in X$,
\[
 \sum_{p\in P}P_{\mathbf x,p}
 =\sum_{a\in X}|T_a|Q_{\mathbf x,a}=rI.                  
\]
The
oracular condition also gives
$[P_{\mathbf x,p},P_{\mathbf y,q}]=0$ whenever $\mathbf x\sim_{G^m}\mathbf y$.
We now aim to prove that condition (1) of~\Cref{lem:degree-one-commutativity} holds.
To that end, we use the projective-packing bound.
Observe that, for any fixed $p\in P$, we have $P_{\mathbf x,p}P_{\mathbf y,p}=0$
whenever $\mathbf x\sim_{G^m}\mathbf y$,
since any two vertices in the independent set $S_p$ are non-adjacent. Hence, the family
$(P_{\mathbf x,p})_{\mathbf x\in X^m}$ is a projective packing of $G^m$, whose value is
\[
       t_p\coloneqq\frac{1}{\dim H}\sum_{\mathbf x\in X^m}\Tr(P_{\mathbf x,p}).
\]
Let now $Q_m$ be a dual $\vartheta^-$ certificate of $G^m$ at value $\mu v^m$ with $\ker Q_m=\incspace^{[m]}$ witnessing that $T$ is Schrijver-rigid. We know from~\Cref{thm_alpha_p_schrijver} that
\begin{align}
\label{eqn_1030_2108}
    t_p\leq \alpha_p(G^m)
    \leq\vartheta^-(G^m)\leq \mu v^m
\end{align}
for each $p\in P$. 
On the other hand, summing over $p$ gives
\begin{align*}
    \sum_{p\in P}t_p
    =
    \frac{1}{\dim H}\sum_{\bx\in X^m}\Tr(rI)=rv^m=|P|\mu v^m,
\end{align*}
so~\cref{eqn_1030_2108} holds with equality for every $p$. We can then apply~\Cref{lem:operator-theta-rigidity} to conclude that  the function
$\mathbf x\mapsto P_{\mathbf x,p}$
lies in $\ker Q_m\otimes\End(H)=\incspace^{[m]}\otimes\End(H)$ for each $p$. This means that the projectors $P_{\bx,p}$ verify all conditions  of~\Cref{lem:degree-one-commutativity}, and we deduce that they all commute with each other.

Fix $\bx\in X^m$ and $a\in X$,
and label the elements of $T_a$ by $p_1,\dots,p_r$.
Orthogonality and idempotency of the members of the PVM $(Q_{\bx,b})_{b\in X}$ gives
\begin{align*}
    \prod_{p\in T_a}P_{\bx,p}
    &=
    \prod_{p\in T_a}\sum_{b\in S_p} Q_{\bx,b}
    =
    \left(\sum_{b_1\in S_{p_1}}Q_{\bx,b_1}\right)\cdot\left(\sum_{b_2\in S_{p_2}}Q_{\bx,b_2}\right)\cdot\dots\cdot\left(\sum_{b_r\in S_{p_r}}Q_{\bx,b_r}\right)\\
    &=
    \sum_{b\in \bigcap_{j\in[r]} S_{p_j}}Q_{\bx,b}
    =
    \sum_{\substack{b\in X\\T_a\subseteq T_b}}Q_{\bx,b}.
\end{align*}
All sets $T_b$ have size $r$, so $T_a\subseteq T_b$ implies
$T_a=T_b$.  Injectivity of the disjointness representation then gives $b=a$.  Thus,
\[
Q_{\mathbf x,a}
=\prod_{p\in T_a}P_{\mathbf x,p}. 
\]
Every projector $Q_{\bx,a}$ is then a product of globally commuting projectors.  Therefore, all projectors in the family $(Q_{\mathbf x,a})_{\bx\in X^m,a\in X}$ commute with each other, and $Q$
is non-contextual.  Applying~\Cref{thm_nonbipartite_noncontextual_undecidable}, we deduce that $\qCSP(G)$ is $\RE$-complete.
\end{proof}

\section{Schrijver-rigidity in association schemes}
\label{sec:association-schemes}

In this section, we shall fix a (symmetric) association scheme $\mathfrak X=\{A_0,A_1,\dots,A_s\}$ and follow the notation of~\Cref{subsec_overview_association_schemes}.
Recall that, for a set $\Gamma\subseteq[s]$, the derived graph $G_\Gamma$ has
vertex set $X$ and adjacency matrix
$A_\Gamma=\sum_{i\in\Gamma}A_i$.
It is regular, and its eigenvalue on $E_j\C^X$ is
$\eta_j(\Gamma)=\sum_{i\in\Gamma}P_{ji}$.  Its complement is also
derived, with relation set $[s]\setminus\Gamma$.
Note that, if $G_\Gamma$ is derived from the original scheme and $G_\Gamma^m$ is its $m$-th categorical power, then the adjacency matrix of $G_\Gamma^m$ is
\[
    A(G_\Gamma^{m})=A(G_\Gamma)^{\otimes m}
    =
        \left(\sum_{i\in\Gamma}A_i\right)^{\otimes m}
        =
    \sum_{\bi\in\Gamma^m}A_{\bi},          
\]
so $G_\Gamma^{m}$ is derived from the tensor-power scheme $\mathfrak X^{[m]}$ with relation set $\Gamma^m\subseteq [s]^m$.
We write
\[
    E_0^{(m)}\coloneqq E_{\mathbf0}=E_0^{\otimes m}=|X|^{-m}J
\]
when the underlying set $X$ and the value of $m$ are clear.

Let $T$ be a $(P,r)$-disjointness representation of a derived graph
$G=G_\Gamma$, with incidence space $\incspace=\C\1\oplus U$. As always, write
$v=|X|$ and $\mu=r/|P|$. Recall the definition of a scheme-compatible representation given in~\Cref{subsec_overview_association_schemes}.
Observe that, for such a representation,
\begin{align*}
 U_i
 &= (E_0\C^X)^{\otimes(i-1)}\otimes U\otimes
    (E_0\C^X)^{\otimes(m-i)}
 =\bigoplus_{j\in\Lambda}E_{j\be_i}\C^{X^m},
\end{align*}
where $U_i$ is given in~\Cref{L-power} and $\Lambda$ is the index set of the representation (see~\Cref{def:bm-incidence-module}).

\lemdelsarteiffschrijver*

\begin{proof}
Suppose first that $T$ is a
Delsarte-rigid $(P,r)$-disjointness representation.  Fix $m\geq1$, put $N=v^m$, and let $F_m$ be a
a witness of Delsarte-rigidity (as in~\Cref{def_delsarte}).  Set $Q_m\coloneqq NF_m-J$.
Since $J=NE_{\mathbf0}$, condition~\textup{(D1)} gives
$Q_m\succeq0$, while~\textup{(D2)} gives diagonal entries
$\mu N-1$ and entries at most $-1$ on distinct nonadjacent pairs.
Hence, $Q_m$ is a dual $\vartheta^-$ certificate for $G^m$ at value $\mu N=\frac{rv^m}{|P|}$.
Furthermore, by~\textup{(D3)} and~\eqref{eqn_1653_2408},
\[
 \ker Q_m
 =
 \bigoplus_{\mathbf j\in\Lambda^{[m]}}
 E_{\mathbf j}\mathbb C^{X^m}
 =
 \incspace^{[m]}.
\]
Thus $Q_m$ is the certificate required by Schrijver-rigidity in~\Cref{def:power-theta-rigid}.

Conversely, suppose that $T$ is Schrijver-rigid. Since every feature fibre
$S_p$ is independent and the certificate $Q_1$ gives
$\vartheta^-(G)\leq\mu v$, 
\Cref{thm_alpha_p_schrijver} implies that
$|S_p|\leq\alpha(G)\leq\vartheta^-(G)\leq\mu v$.
Reasoning as in Observation~\ref{obs_independence_number_disjointness_representation}, we find
$\sum_{p\in P}|S_p|=rv=|P|\mu v$, and it follows that
$|S_p|=\mu v$ for every $p\in P$.

For each $\bi\in\{0,\ldots,s\}^m$, let $R_{\bi}$ denote the binary relation over $X^m$ corresponding to $A_\bi$, and let $\Pi_m$ be the Frobenius-orthogonal
projection onto the Bose--Mesner algebra $\C[\mathfrak X^{[m]}]$ of $\mathfrak X^{[m]}$.
Recall from~\Cref{subsec_overview_association_schemes} that both families $(A_{\mathbf i})_{\bi}$ and $(E_{\mathbf j})_\bj$ are Frobenius-orthogonal bases of $\C[\mathfrak X^{[m]}]$.
Hence, for each matrix $M\in\C^{X^m\times X^m}$, we have
\begin{align}
\label{eqn_1719_2408}
 \Pi_m M
 =
 \sum_{\mathbf i}
 \frac{\langle A_{\mathbf i},M\rangle}
      {\langle A_{\mathbf i},A_{\mathbf i}\rangle}
 A_{\mathbf i}
 =
 \sum_\bi\frac{1}{|R_\bi|}\left(\sum_{(\bu,\bv)\in R_\bi}M_{\bu,\bv}\right)A_\bi
 =
  \sum_{\mathbf j}
 \frac{\langle E_{\mathbf j},M\rangle}
      {\langle E_{\mathbf j},E_{\mathbf j}\rangle}
 E_{\mathbf j}
 =
 \sum_{\mathbf j}
 \frac{\operatorname{Tr}(E_{\mathbf j}M)}
      {\operatorname{rank}(E_{\mathbf j})}
 E_{\mathbf j},
\end{align}
where the last equality holds since the $E_\bj$'s are Hermitian and idempotents.
Define $\overline Q_m\coloneqq\Pi_m Q_m$.  The primal expansion in~\cref{eqn_1719_2408} shows that
$\overline Q_m$ is obtained by averaging the entries of $Q_m$ on
each relation of the tensor scheme.  
Hence,
$\overline Q_m$ has diagonal $\mu N-1$ and entries at most $-1$ on
every nondiagonal nonedge relation.
The dual expansion shows that $\overline Q_m\succeq0$, because
\begin{align}
    \label{eqn_1807_2408}
 \operatorname{Tr}(E_{\mathbf j}Q_m)
 =
 \bigl\|Q_m^{1/2}E_{\mathbf j}\bigr\|_{\mathrm F}^{\,2}
 \geq0.
\end{align}
Thus $\overline Q_m$ is again a dual $\vartheta^-$ certificate for
$G^m$ at value $\mu N$.  Moreover, the quantity in~\cref{eqn_1807_2408} vanishes
if and only if
$E_{\mathbf j}\mathbb C^{X^m}\subseteq\ker Q_m$.  It follows that $\ker\overline Q_m
 \subseteq
 \ker Q_m
 =
 \incspace^{[m]}$.
We now prove the converse inclusion.
For $i\in[m]$ and $p\in P$, define the set $S_{i,p}^{[m]}
 \coloneqq
 X^{i-1}\times S_p\times X^{m-i}$.
This is an independent set of $G^m$ of size
$|S_p|v^{m-1}=\mu N$, and its indicator is
$\iota_i(\chi_p)=\chi_{i,p}$.
The span of these vectors is~$\incspace^{[m]}$ by the discussion in~\Cref{sect:representations-powers}.
Since $\overline Q_m$ is a dual $\vartheta^-$ certificate, we find
\[
 0
 \leq
 \mathbf1_{S_{i,p}^{[m]}}^\top
 \overline Q_m
 \mathbf1_{S_{i,p}^{[m]}}
 \leq
 |S_{i,p}^{[m]}|
 \bigl(\mu N-|S_{i,p}^{[m]}|\bigr)
 =
 0.
\]
Since $\overline Q_m\succeq0$, this implies
$\mathbf1_{S_{i,p}^{[m]}}\in\ker\overline Q_m$.
Since the indicators of the sets $S_{i,p}^{[m]}$ span
$\incspace^{[m]}$, we deduce that 
$\incspace^{[m]}\subseteq\ker\overline Q_m$ and, consequently, 
\begin{align}
\label{eqn_1808_1509}
\ker\overline Q_m=\incspace^{[m]}.
\end{align}

Taking $m=1$, we see that $\incspace$ is the kernel of an element of
the Bose--Mesner algebra and is therefore a direct sum of primitive
modules.  Hence, there is a set
$\Lambda\subseteq[s]$ such that
$U=\bigoplus_{j\in\Lambda}E_j\mathbb C^X$.
Thus $T$ is scheme-compatible, with index set $\Lambda$.
Define now
\[
 F_m\coloneqq\frac{1}{N}\bigl(\overline Q_m+J\bigr).
\]
Note that $F_m$ is real and symmetric, and it lies in the Bose--Mesner algebra $\C[\mathfrak X^{[m]}]$.
Take primal and dual decompositions of $F_m$ as in~\eqref{eqn_1758_1208}, and 
let $q_{\mathbf j}$ be the eigenvalue of $\overline Q_m$ on
$E_{\mathbf j}\mathbb C^{X^m}$.  Since $J=NE_{\mathbf0}$, we have
$\lambda_{\mathbf0}^{(m)}=1$ and $\lambda_{\mathbf j}^{(m)}
 =\frac{q_{\mathbf j}}{N}\geq0$
whenever  $\mathbf j\neq\mathbf0$, thus proving~\textup{(D1)}.
For $\mathbf j\neq\mathbf0$, the kernel identity~\cref{eqn_1808_1509} gives
$\lambda_{\mathbf j}^{(m)}=0$ if, and only if,
$\mathbf j\in\Lambda^{[m]}\setminus\{\mathbf0\}$.  Since
$\lambda_{\mathbf0}^{(m)}=1$, this proves~\textup{(D3)}.
In the  primal decomposition, the coefficient of
$A_{\mathbf0}$ in $F_m$ is
\[
 \frac{1}{N}\bigl((\mu N-1)+1\bigr)=\mu,
\]
while every coefficient corresponding to a nondiagonal nonedge  relation is
nonpositive.  Thus~\textup{(D2)} also holds, and $T$ is Delsarte-rigid.
\end{proof}

The description of Delsarte rigidity in
Definition~\ref{def_delsarte} is quantified over every tensor power $m$. 
In the examples of~\Cref{section_examples}, we shall use a slightly stronger
 condition, formulated entirely in the primal and dual bases of the
Bose--Mesner algebra of the first power (i.e., of the original scheme).

\begin{proposition}
\label{prop:bose-mesner-spectral-test}
Let $T$ be a scheme-compatible disjointness representation of $G=G_\Gamma$ with
index set  $\Lambda$.  Suppose that there exists some matrix %
$B\in\Span\{A_i\mid i\in\Gamma\}$
acting on $E_j\mathbb C^X$ as the scalar $\theta_j\in\R$, where
\begin{align}
\label{eqn_1611_2508}
 \theta_0=\kappa,
 \qquad \kappa>1,
 \qquad
 \theta_j=-1\quad(j\in\Lambda),
 \qquad
 -1<\theta_j<\kappa\quad(j\notin\Lambda\cup\{0\}).     
\end{align}
Then
$\mu=\frac1{\kappa+1}$,
and $T$ is a Delsarte-rigid representation.  Moreover, one may take
\[
 F_m=\frac\mu{\kappa^{m-1}}(B^{\otimes m}+\kappa^{m-1}A_{\mathbf0})%
\]
as the $m$-th certificate in~\Cref{def_delsarte}.
\end{proposition}

\begin{proof}
Fix $p\in P$ (where $P$ is the feature set of $T$) and define $\mu_p\coloneqq\frac{|S_p|}{v}$.  Since
$\chi_p-\mu_p\mathbf1\in U=\bigoplus_{j\in\Lambda}E_j\mathbb C^X$, we find $B(\chi_p-\mu_p\1)=-\chi_p+\mu_p\1$ and, thus,
\[
 B\chi_p=
 \mu_p B\1-\chi_p+\mu_p\1
 =
 (\kappa+1)\mu_p\mathbf1-\chi_p.
\]
Recall that $S_p\neq\emptyset$, and take
$x\in S_p$; by independence of $S_p$, 
no member of $S_p$ is adjacent to $x$.  The support
condition on $B$ therefore gives $(B\chi_p)(x)=0$, and hence
$\mu_p=1/(\kappa+1)$.  Averaging over the feature fibres using
\eqref{eqn_1648_0308} then shows that $\mu=\frac{r}{|P|}=\frac{1}{\kappa+1}$, as required.

The matrix $B^{\otimes m}$ is supported on the edge relation of
$G^{m}$.  Thus $F_m$ has diagonal relation coefficient
$\mu$ and zero coefficients on every nondiagonal nonedge
relation, proving~\textup{(D2)}.  Its eigenvalue on the primitive module
indexed by $\bj=(j_1,\ldots,j_m)$ is
\[
 \lambda_{\bj}^{(m)}
   =\frac1{(\kappa+1)\kappa^{m-1}}\left(\prod_{\ell=1}^m\theta_{j_\ell}+\kappa^{m-1}\right).
\]
At $\bj=\mathbf0$, this is $1$.  If exactly one
coordinate of $\bj$ lies in $\Lambda$ and all others are zero (i.e., if $\bj\in\Lambda^{[m]}\setminus\{\bzero\}$), we get $\lambda_{\bj}^{(m)}=0$.
Furthermore, suppose $\bj$ is not in such form and the product $\prod_{\ell=1}^m\theta_{j_\ell}$ is negative.
Then the absolute
value of the product is strictly smaller than $\kappa^{m-1}$, as the negative factors have
absolute value at most $1$, with equality only on $j\in\Lambda$,
whereas positive factors have value at most $\kappa$, with equality only
on $j=0$. Also, if the product is nonnegative, then $\lambda_{\bj}^{(m)}>0$.  Altogether, this proves~\textup{(D1)} and~\textup{(D3)}.
\end{proof}

A particularly simple case occurs when the nonconstant part $U$ of the incidence space of a scheme-compatible representation is precisely the eigenspace of the adjacency matrix $A(G_\Gamma)$ corresponding
to its least eigenvalue.
\begin{definition}
\label{def:hoffman-feature}
Let
$\tau=\min_{0\leq j\leq s}\eta_j(\Gamma)$ and $\Lambda_\tau=\{j\in[s] \mid \eta_j(\Gamma)=\tau\}$.
A scheme-compatible disjointness representation of $G_\Gamma$ is \emph{Hoffman-rigid} if its index set  is $\Lambda_\tau$; equivalently, if
\[
 \incspace\cap\1^\perp
 =\bigoplus_{j\in\Lambda_\tau}E_j\mathbb C^X.
\]
\end{definition}

\begin{proposition}
\label{prop:hoffman-implies-delsarte}
Suppose that $G=G_\Gamma$ is connected and non-bipartite.  Every Hoffman-rigid
representation of $G$ is Delsarte-rigid.
\end{proposition}

\begin{proof}
Define $d=\eta_0(\Gamma)$, and note that
$G_\Gamma$ is $d$-regular.  Since $G$ is connected, $d$ is a simple
adjacency eigenvalue.  Since $G$ is non-bipartite, $-d$ is not an
adjacency eigenvalue.  Moreover, the least eigenvalue $\tau$ is
negative.  Set $B\coloneqq -\frac{A(G)}{\tau}$ and $\kappa\coloneqq-\frac d{\tau}>1$.
Note that $B$ is an element of $\Span\{A_i\mid i\in\Gamma\}$ and, on $E_j\mathbb C^X$, it has eigenvalue
$-\eta_j(\Gamma)/\tau$.  This equals $-1$ precisely for
$j\in\Lambda_\tau$, while every remaining nonconstant eigenvalue lies
strictly between $-1$ and $\kappa$.  Hence,
\Cref{prop:bose-mesner-spectral-test} applies.
\end{proof}

\begin{remark}
\label{rem:hoffman-tightness}
Just like for Schrijver and Delsarte rigidity, a Hoffman-rigid representation forces Hoffman's ratio bound (see~\cite{haemers2021hoffman}) on the independence number to be tight.
More precisely, every feature fibre of a Hoffman-rigid representation
is a maximum independent set of size $-\frac{\tau v}{d-\tau}$, where $d$ is the common degree of the vertices of the (regular) graph $G$.

To see this, let $A$ be the adjacency matrix of
$G$, and let $\tau<0$ be its least eigenvalue. For each $p\in P$, write
$\mu_p=\frac{|S_p|}{v}$.
Since $\chi_p\in\mathcal L$,
$\chi_p-\mu_p\mathbf{1}$ belongs to
$\mathcal L\cap\mathbf{1}^{\perp}$, which is the $\tau$-eigenspace
of $A$ by Hoffman rigidity. Hence, using $A\mathbf{1}=d\mathbf{1}$,
\[
  A\chi_p
  =d\mu_p\mathbf{1}
   +\tau\bigl(\chi_p-\mu_p\mathbf{1}\bigr)
  =(d-\tau)\mu_p\mathbf{1}+\tau\chi_p.
\]
Evaluating this identity
at any $x\in S_p$ and using that the set $S_p$ is independent gives
\[
  0=(A\chi_p)_x
   =(d-\tau)\mu_p+\tau
\]
and, thus, $\mu_p=\frac{-\tau}{d-\tau}$. Hence, $S_p$ meets Hoffman ratio bound with equality.
\end{remark}

\section{The classic metric association schemes}
\label{section_examples}
In this section, we consider the Johnson, Grassmann, and Hamming schemes, and apply the results of the previous sections to obtain commutativity gadgets for certain graphs derived from these association schemes, see~\Cref{tab:AS-graphs}. In particular, we prove~\Cref{thm_RE_completeness_classes}.

The Johnson scheme for parameters $n$ and $k$ is given by vertex set $X=\binom{n}{k}$ and relations $R_s=\{(x,y)\mid |x\cap y|=k-s\}$.
The Grassmann scheme for parameters $q,n,k$ (where $q$ is a prime power) is obtained by taking vertex set $X=\{U\leq\mathbb F_q^n \mid \dim U=k\}$ and letting $R_s=\{(U,V)\mid \dim(U\cap V)=k-s\}$.
The Hamming scheme for parameters $d$ and $q$ is given by $X=[q]^d$ and $R_s=\{(x,y)\mid h(x,y)=s\}$, where $h(x,y)=|\{i\in[d]\mid x_i\neq y_i\}|$ is the Hamming distance.

For each of these association schemes, we consider the graphs given by their maximum-distance relations and the complement of the graph given by their distance-one relations, and we show that commutativity gadgets exist within natural parameter ranges.
\begin{table}[!h]
    \centering
    \small
\setlength{\tabcolsep}{3pt}
\begin{tabular}{|p{0.14\textwidth}|p{0.23\textwidth}|p{0.23\textwidth}|p{0.27\textwidth}|}
\hline
Scheme & Relation $R_s$ & Maximum-distance graph & Complement of the distance-one graph \\
\hline
Johnson & $|\vA\cap \vB|=k-s$ & $R_k=\KG(n,k)$ & $R_2\cup\cdots\cup R_k=\overline{J(n,k)}$ \\
\hline
Grassmann & $\dim(U\cap W)=k-s$ & $R_k=\operatorname{KG}_q(n,k)$ & $R_2\cup\cdots\cup R_k=\overline{J_q(n,k)}$ \\
\hline
Hamming & $h(x,y)=s$ & $R_d=X_d(q)$ & $R_2\cup\cdots\cup R_d=\overline{H(d,q)}$ \\
\hline
\end{tabular}
    \caption{The Johnson, Grassmann, and Hamming schemes, and their derived graphs.}
    \label{tab:AS-graphs}
\end{table}

When the scheme has two classes (i.e. $k=2$ or $d=2$), the two selected graphs coincide, since the complement of $R_1$ is $R_2$; for $k\geq3$ (or $d\geq 3$) they are distinct.

\subsection{The Johnson scheme}
We consider the two families of derived graphs arising from the Johnson scheme. The common tool is the expression of incidence spaces through inclusion modules $\mathcal{W}_j$ , whose decomposition into the dual basis is given in~\Cref{lem:johnson-modules}.
For the Kneser graph, the identity feature map has incidence space $\mathcal W_1$, leading directly to Hoffman-rigidity. For the complement of the Johnson graph, the natural features are the $(k-1)$-subsets, and the incidence space is $\mathcal W_{k-1}$; in this case we construct a matrix in the Bose--Mesner algebra to satisfy the finite Delsarte-rigidity criterion.

\label{sec:johnson-scheme}
For $0\le j\le k$, take $n>2k\geq2$, and let $X=\binom{[n]}{k}$. We define
$\mathcal W_j\coloneqq \operatorname{span}\left\{\chi_\vR \mid \vR\in\binom{[n]}{j}\right\}$, where
$\chi_\vR(\vA)=\indfunction{\vR\subseteq \vA}$.
The next result follows from the known properties of the Johnson scheme~\cite[\S~6]{godsil2016erdos}.

\begin{lemma}\label{lem:johnson-modules}
For $0\le j\le k$,
\[
\mathcal W_j=\bigoplus_{i=0}^{j}E_i\mathbb C^X,
\qquad
\dim \mathcal W_j=\binom nj.
\]
Consequently, the functions $\chi_\vR$, $\vR\in\binom{[n]}{j}$, form a basis
of $\mathcal W_j$.
\end{lemma}
\begin{proof}
By~\cite[Theorem~6.3.3]{godsil2016erdos}\footnote{In the notation of Godsil and Meagher \cite{godsil2016erdos}, $\mathcal W_j=\operatorname{col}(W_{j,k}^{\top})$.}, the $i$-th primitive module of the Johnson scheme is $E_i\mathbb C^X=\mathcal W_i\cap \mathcal W_{i-1}^{\perp}$, where $\mathcal{W}_{-1}=0$ and $\dim (E_i\mathbb C^X)=\binom ni-\binom n{i-1}$. Also, by ~\cite[Corollary~6.3.4]{godsil2016erdos}, $E_0+\cdots+E_j$ is the orthogonal projection onto $\mathcal W_j$.
Thus, $\mathcal W_j=\bigoplus_{i=0}^{j}E_i\mathbb C^X$ and $\dim \mathcal W_j
=\sum_{i=0}^{j}\left(\binom ni-\binom n{i-1}\right)
=\binom nj.$
\end{proof}

\subsubsection{Kneser graphs}
\label{subsec:kneser-verification}

\begin{figure}[t]
  \centering
  \includegraphics[width=.75\linewidth]{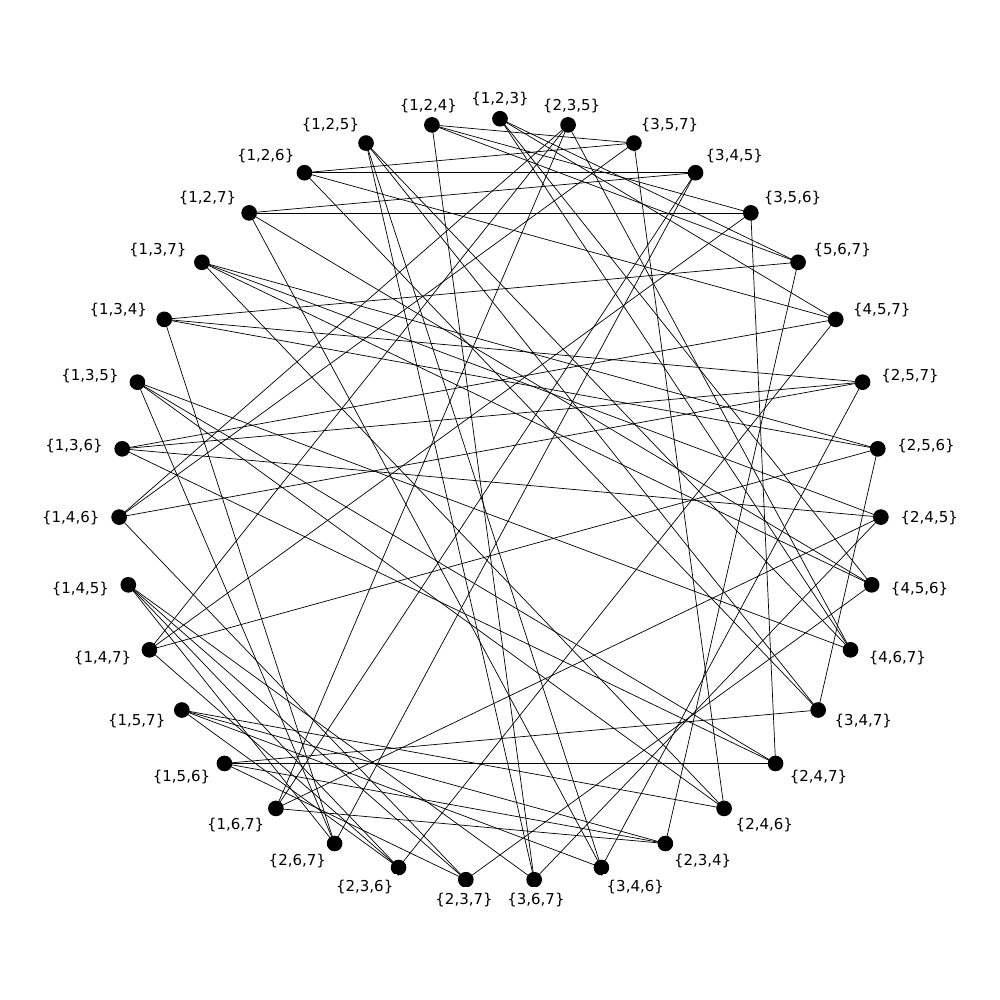}
  \caption{The Kneser graph $\operatorname{KG}(7,3)$.  Its vertex labels are the $3$-subsets of $[7]=\{1,\ldots,7\}$, and two vertices are adjacent precisely when the corresponding subsets are disjoint.  The graph has $35$ vertices and $70$ edges.}
  \label{fig:kneser-kg7-3}
\end{figure}
Recall that the Kneser graph $\KG(n,k)$ is the graph with vertices $X$ and where $\vA\sim \vB$ if, and only if, $\vA\cap \vB=\emptyset$. The edges of $KG(n,k)$ form the maximum-distance relation of the
Johnson scheme, and thus the adjacency matrix is $A_k$.

By the $i=k$ case of \cite[Theorem~6.5.2]{godsil2016erdos}, the primitive modules are eigenspaces of $A_k$ (the adjacency matrix of $KG(n,k)$), and the eigenvalue
corresponding to $E_j\mathbb C^X$ is
\[
\lambda_j=(-1)^j\binom{n-k-j}{k-j},
\qquad 0\le j\le k.
\]
Since the lambdas are distinct, the eigenspace of $KG(n,k)$ with eigenvalue $\lambda_j$
is precisely $E_j\mathbb C^X$.

In particular, the least eigenvalue is $\tau=\lambda_1=-\binom{n-k-1}{k-1}$, with eigenspace $E_1\mathbb C^X$. 
Moreover, $\lambda_0=\binom{n-k}{k}$ is the degree and has multiplicity
one, while $-\lambda_0$ is not an eigenvalue. Thus $KG(n,k)$ is
connected, regular, and non-bipartite.

Our goal is to show that the Kneser graph $KG(n,k)$ has a tame Hoffman-rigid disjointness representation. The next elementary matrix lemma will determine said edge kernel.

\begin{lemma}
\label{lem:kneser-cross-sums}
Let $n>2k$ and consider a matrix $M=(m_{ij})\in\mathbb C^{n\times n}$.  If $\sum_{i\in \vA}\sum_{j\in \vB}m_{ij}=0$ for every two disjoint $\vA,\vB\in\binom{[n]}k$, then $M$ is diagonal.
\end{lemma}

\begin{proof}
We first prove the following claim: for any $Y\subseteq [n]$ with $|Y|>k$ and any family of numbers $(u_i)_{i\in Y}$ such that $\sum_{i\in b}u_i=0$ for every $b\in \binom Yk$, we have $u_i=0$ for every $i\in Y$. 
Indeed, for any distinct $i,j\in Y$, choose $k$-sets $b_i,b_j\subseteq Y$ such that $b_i\setminus b_j=\{i\}$, and  $b_j\setminus b_i=\{j\}$.
Subtracting the two vanishing sums gives that $u_i=u_j$ and hence, every $u_i$ has a common value $c$. Considering a sum over any $k$-set yields that $kc=0$, so $c=0$ proving the claim. 
Now fix any $\vA\in X$, we let $j\notin \vA$ and apply the claim to $[n]\setminus \vA$ with $u_j=\sum_{i\in \vA}m_{i,j}$. The hypothesis of the claim is the assumption of the lemma, so this yields $ \sum_{i\in \vA}m_{i,j}=0$ for every $j\notin \vA$ since $|[n]\setminus \vA|=n-k>k$. 
Finally, fix any $j\in [n]$ and apply the claim to $[n]\setminus{j}$ with $u_i=m_{i,j}$, the hypothesis of the claim now holds by the previous paragraph. We conclude that $m_{i,j}=0$ for every $i\neq j$, completing the proof.
\end{proof}

\begin{proposition}
\label{prop:kneser-criterion}
For every $n>2k\geq2$, the identity map $\vA\mapsto T_\vA=\vA$ is a Hoffman-rigid tame representation of $\KG(n,k)$. 
Consequently, the graph admits a
commutativity gadget and the corresponding $\qCSP$ is $\RE$-complete.
\end{proposition}

\begin{proof}
For the proposed representation we have that $P=[n]$, $r=k$ and one can easily see for any $i\in [n]$ that $S_i =\{\vA\in X \mid i\in \vA \}, \chi_i(\vA)=\indfunction{i\in \vA}$, and  $\sum_{i}\chi_i=k\1$.
Every $S_i$ is a nonempty independent set, since two $k$-sets containing $i$
cannot be disjoint and for this representation we have $\incspace=\mathcal W_1$. Hence, by Lemma
\ref{lem:johnson-modules}, $\incspace=E_0\mathbb C^X\oplus E_1\mathbb C^X$ with $\incspace\cap\mathbf1^\perp=E_1\mathbb C^X$, and the functions $\chi_i$, $i\in[n]$, form a basis of $ \incspace$.
Therefore, this is a Hoffman-rigid representation with index set $\{1\}$.
Consider any  $F\in\incspace\otimes\incspace$, and note that $\{\chi_i\mid i\in [n]\}$ forms a basis for $\incspace$ by~\Cref{lem:johnson-modules}. We write
\begin{align*}
    F &= \sum_\ell f_\ell \otimes g_\ell
      = \sum_\ell\sum_{i,j\in[n]} c_{f,\ell,i}c_{g,\ell,j}(\chi_i\otimes\chi_j)
      = \sum_{i,j\in[n]} m_{i,j} (\chi_i\otimes\chi_j)
\end{align*}
so, for any $\vA,\vB\in X$, $F(\vA,\vB)$ has the (unique) expression
\begin{align*}
    F(\vA,\vB) &= \sum_{i,j\in[n]} m_{i,j} (\chi_i\otimes\chi_j)(\vA,\vB)
           = \sum_{\substack{i\in \vA\\ j\in \vB}} m_{i,j}
           = \sum_{i\in \vA}\sum_{j\in \vB} m_{i,j}.
\end{align*}
The map $F$ belongs to the edge kernel $\edgekernel$ precisely when the requirements for~\Cref{lem:kneser-cross-sums} hold. This is equivalent to $M_F$ being
diagonal and, in particular, symmetric. Also, since $F=\sum_{i,j\in[n]}m_{i,j}(\chi_i\otimes \chi_j)$, the statement ``$F\in \edgekernel$ if, and only if, $M_F$ is diagonal'' is equivalent to

\[\edgekernel=\Span\{\chi_i\otimes \chi_i\mid i\in [n]\}.\]

For every $\vA\in X$, the set $[n]\setminus \vA$ has
$n-k\geq k+1$ elements and therefore contains two distinct $k$-sets;
both are neighbours of $\vA$.  Hence
\Cref{lem:diagonal-feature-rigidity} applies and proves
tameness.   By~\Cref{prop:hoffman-implies-delsarte}, the Hoffman-rigid representation is Delsarte-rigid. By~\Cref{lem:delsarte-iff-schrijver}, it is Schrijver-rigid. \Cref{thm:criterion} therefore applies.
\end{proof}

\begin{corollary}
    For every $n\geq2$, the odd graph $O_n=\KG(2n-1,n-1)$ admits both an oracular and a non-oracular
commutativity gadgets, and the corresponding $\qCSP$ is $\RE$-complete in both the oracular and non-oracular settings.
\end{corollary}
\begin{proof}
    The oracular statement is clear from~\Cref{prop:kneser-criterion}. Since $O_n$ is $C_4$-free, it follows from~\cite[Corollary 5.28]{culf2025existence} that the commutativity gadget from~\Cref{prop:kneser-criterion} also works as a non-oracular commutativity gadget. The $\RE$-completeness of the non-oracular $\qCSP$ is then proved in a similar way as in~\cite[Proposition~103]{CJM}, by observing that $O_n$ pp-defines itself via a tree gadget and, thus, it $q$-no-defines itself (see~\cite[Proposition~100]{CJM} and~\cite[Definition~93]{CJM} for the terminology). By~\cite[Proposition~98]{CJM}, $\qCSP(O_n)$ reduces to its non-oracular version, and the proof is concluded.
\end{proof}

    This answers the question on the existence of commutativity gadgets and complexity for odd graphs that was posed in~\cite[Question 5]{culf2025existence}, in both the oracular and non-oracular settings.

\subsubsection{Complements of Johnson graphs}
\label{subsec:johnson-complement}

\begin{figure}[t]
  \centering
  \includegraphics[width=.75\linewidth]{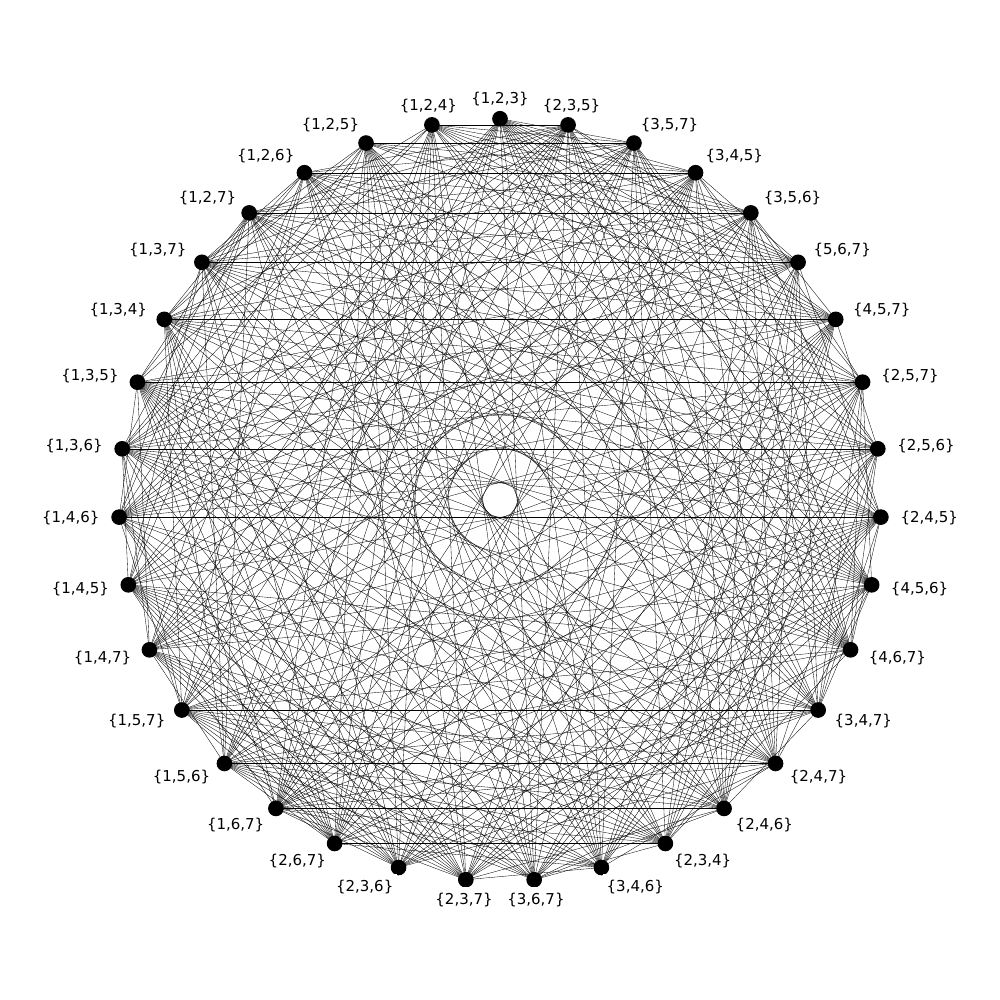}
  \caption{The complement $\overline{J(7,3)}$ of the Johnson graph.  Its vertex labels are the $3$-subsets of $[7]=\{1,\ldots,7\}$, and two distinct vertices $\vA,\vB$ are adjacent precisely when $|\vA\cap \vB|\leq1$.  The graph has $35$ vertices and $385$ edges.}
  \label{fig:johnson-complement-j7-3}
\end{figure}

Let $J(n,k)$ be the Johnson graph on $\binom{[n]}k$, whose edges join
pairs with intersection of size $k-1$.  Its complement has adjacency relation $\vA\sim \vB\quad\Longleftrightarrow\quad |\vA\cap \vB|\leq k-2$.

As shown in~\Cref{tab:AS-graphs}, the adjacency matrix of $\overline{J(n,k)}$ equals $\sum_{i=2}^kA_i$ and it therefore arises from the Johnson scheme with $\Gamma=\{2,\dots,k\}$.
We show that this graph has a Delsarte-rigid representation.
Note that the full power of the Delsarte machinery is needed in this case: we  show in~\Cref{Prop:Delsarte necessary}, that $\overline{J(n,k)}$ does not admit a Hoffman-rigid representation.

We will need to express elements from the dual basis in terms of the primal basis. To this end, we write the normalised change of basis matrix $Q=(Q_{ij})$ defined by

\[E_j=\frac{1}{v}\sum_{i=0}^sQ_{ij}A_i,\]
so that $PQ=vI$, where, as usual, $v=|X|$. We determine $Q$; using~\cref{eqn_2608_1840} and expanding $E_j$ gives

\[\langle E_j, A_i \rangle =\left\langle\frac{1}{v}\sum_{\ell=0}^sQ_{\ell j}A_{\ell},A_i\right\rangle=\frac{1}{v} Q_{ij} |R_i|.\]
$P_{ji}$ is the eigenvalue of $A_i$ on $E_j\C^X$. Hence, $A_iE_j=P_{ji}E_j$ and taking traces gives $\langle E_j, A_i \rangle=P_{ji}\rank(E_j)$, thus 

\begin{equation}\label{eq: Q_{i,j}}
    Q_{ij}=\frac{v\rank(E_j)}{|R_i|}P_{ji}.
\end{equation}

The $j=k$ case of~\cite[Theorem~6.5.2]{godsil2016erdos} gives $P_{ki}=(-1)^i\binom ki$. So, since $|R_i|=\binom nk\binom ki\binom Ni$, we get
\begin{equation}\label{eqn: Q_{ik}}
Q_{ik}=\frac{\binom nk m_k}{\binom nk\binom ki\binom Ni}(-1)^i\binom ki=(-1)^i\frac{m_k}{\binom Ni},
\end{equation}
where $m_k=\rank(E_k)=\binom{n}{k}-\binom{n}{k-1}$.
\begin{proposition}
\label{prop:johnson-complement}
If $n>2k\geq4$, then $T\colon \vA\mapsto T_\vA=\binom{\vA}{k-1}$ is a Delsarte-rigid tame 
representation of $\overline{J(n,k)}$.  Consequently, the graph admits a
commutativity gadget and the corresponding $\qCSP$ is $\RE$-complete.
\end{proposition}

\begin{proof}
We claim that the map $T$ given in the proposition statement is a $(P,r)$-disjointness representation. The feature set  of this representation is $P=\binom{[n]}{k-1}$, and the characteristic vectors span $\incspace=\mathcal W_{k-1}$, with $r=k$.
Put $N=n-k$, $v=\binom nk$, and $X=\binom{[n]}k$.

Two vertices $\vA,\vB\in \overline{J(n,k)}$ are adjacent if, and only if, they do not share a $(k-1)$-subset or, equivalently, $T_\vA\cap T_\vB=\emptyset$. Thus, $T$ is a homomorphism.
Also, if $T_\vA=T_\vB$, then $\vA=\bigcup T_\vA=\bigcup T_\vB=\vB$ so $T$ is a disjointness representation with $\mu=\frac{k}{\binom n{k-1}}=\frac{N+1}{v}$.

Since $\incspace= \mathcal W_{k-1}$ we get by~\Cref{lem:johnson-modules} that $\incspace= \C\1\oplus \bigoplus_{i=1}^{k-1} E_i\C^X$. Thus, $T$ is scheme-compatible with index set $\{1,\dots,k-1\}$.

In order to use~\Cref{prop:bose-mesner-spectral-test} to prove the Delsarte-rigidity of $T$,
define
\[B= \kappa E_0 - \sum_{j=1}^{k-1}E_j + \theta E_k,  \]
for some $\kappa>1$ and $-1<\theta<\kappa$ to be determined and such that $B\in\Span\{A_i\mid i\in\Gamma\}$.
Note that the conditions in~\cref{eqn_1611_2508} are satisfied by definition. Since $\sum_{i=0}^kE_i=I$, we get $B= (\kappa+1) E_0 + (\theta+1) E_k  - I.$

We perform a change of basis and use the coefficients $Q_{ik}$ defined in~\cref{eqn: Q_{ik}} to impose that $B$ lies in $\Span\{A_i\mid i\in\Gamma\}$.
Using $Q_{i0}=1$ for all $0\le i\le k$ we get

\[B= \frac{(\kappa+1)}{v}\sum_{i=0}^kA_i + \frac{(\theta+1)}{v}\sum_{i=0}^kQ_{ik}A_i  - A_0.\]
Substituting $i=0$ into~\cref{eqn: Q_{ik}} and reordering gives us 
\[Q_{0k}=\frac{m_k}{1}, \qquad c_0=\frac{\kappa+1+(\theta+1)m_k}{v}-1,\]
where $c_0$ is the coefficient of $A_0$ in $B$; repeating for $i=1$ gives

\[Q_{1k}=-\frac{m_k}{N}, \qquad c_1=\frac{\kappa+1-(\theta+1)m_k/N}{v}.\]

We write $B$ as $\sum_{i=0}^k c_iA_i$; now the condition that $B\in\Span\{A_i\mid i\in\Gamma\}$ is that $c_0=c_1=0$.
Thus, we get the conditions

\[\kappa+1+(\theta+1)m_k=v, \qquad \kappa+1=\frac{(\theta+1)m_k}{N}\]
and solving for $(\kappa+1)$ and $(\theta+1)$ results in
\[ B=\frac v{N+1}E_0+\frac N{N+1-k}E_k-I.\]
By construction, $B\in\Span\{A_i\mid i\in\Gamma\}$, while its eigenvalues on $E_0$, on $E_1\oplus\cdots\oplus E_{k-1}$, and on $E_k$, respectively, are $\kappa=\frac v{N+1}-1$, $-1$, and $\theta=\frac{k-1}{N+1-k}$.
Note that
\[\kappa+1=\frac{v}{N+1}=\frac{1}{N+1}\binom{n}{k}=\frac{1}{N+1}\frac{(N+k)\dots(N+1)}{k!}=\frac{(N+k)\dots(N+2)}{(k)\dots(2)},\]
so
\[\kappa+1=\prod_{i=2}^k\frac{N+i}{i}\geq\frac{N+2}{2}>\max\left\{2,\frac{N}{N+1-k}\right\}.\]
Here the last inequality uses $N> k\ge 2$.  Hence
$\kappa>\max\{1,(k-1)/(N+1-k)\}$.  This is~\cref{eqn_1611_2508} for $\Lambda=\{1,\ldots,k-1\}$, so \Cref{prop:bose-mesner-spectral-test} applies and $T$ is a Delsarte-rigid representation.

In order to complete the proof we still need to show that the representation is tame.
For this, we apply~\Cref{lem:local-incidence-edge-kernel}. Note that for the given representation, we know that $ S_\vR=\{\vA\in X \mid \vR\subseteq \vA\}$ and $\chi_\vR(x)=\indfunction{x\in S_\vR}$ for any $\vR\in\binom{[n]}{k-1}$.

The second and third assumptions of this lemma can easily be seen to be true: every $S_\vR$ has $n-k+1\geq 3$ elements and if $\vA,\vB\in S_\vR$ are distinct then they have $\vR$ as their unique common $(k-1)$-subset; so, $T_\vA\cap T_\vB=\{\vR\}$.

Fix $\vY\in X$; we investigate those functions that are supported on the non-neighbours of $\vY$.
Let $C_\vY=X\setminus N_{\overline{J(n,k)}}(\vY)=\{\vA\in X\colon |\vA\cap \vY|\ge k-1\}$; we show that 

\[\{f\in\incspace\mid\operatorname{supp}(f)\subseteq C_\vY\} =\Span\{\chi_\vR \mid \vR\in\tbinom {\vY}{k-1}\}.\]

For this we will use that $\incspace=\ker E_k$ and work with the principal submatrix $E_{k,\vY}=[(E_k)_{\vA,\vB}]_{\vA,\vB\in C_\vY}$. Any function $f$ supported on $C_\vY$ satisfies $\langle f,E_kf\rangle =\langle f|_{C_\vY},E_{k,\vY}f|_{C_\vY}\rangle$ and, since $E_k$ is an orthogonal projection, this scalar is zero exactly when $E_kf=0$. Hence, if we identify every function $f$ supported on $C_\vY$ with $f|_{C_\vY}$, the space $\{f\in\incspace \mid \operatorname{supp}(f)\subseteq C_\vY\}$ is exactly $\ker E_{k,\vY}$. Set $M=\frac{v}{m_k}E_{k,\vY}$.

For two elements $\vA,\vB$ of $C_\vY$, we have that $|\vA\cap \vB|\ge k-2$. Thus, we use~\cref{eqn: Q_{ik}} to write

\begin{equation}\label{eq: M-decompositon}
    M= \frac{Q_{0k}}{m_k}A_{0,\vY}+ \frac{Q_{1k}}{m_k}A_{1,\vY}+ \frac{Q_{2k}}{m_k}A_{2,\vY}= I-\frac{1}{N}A_{1,\vY}+\frac{1}{\binom{N}{2}}A_{2,\vY}
\end{equation}
where $A_{s,\vY}$ denotes the restriction of $A_s$ to $C_\vY$. We will compute $\ker M$.
Every element of $C_\vY$ is either $\vY$ or a neighbour of $\vY$ in $J(n,k)$ that can be written as
$\vY_{i,j}=(\vY\setminus \{i\})\cup \{j\}$ for $i\in \vY$ and $j\in[n]\setminus \vY$.
We get $C_\vY= \{\vY\}\cup \{\vY_{i,j}\mid i\in \vY, j\notin \vY\}$.
Let $f\in \ker M$, we write $a=f(\vY)$ and $u_{i,j}=f(\vY_{i,j})$ and we define the row sums, column sums, and total sum

\[r_i=\sum_{j\notin \vY}u_{i,j}, \qquad c_j=\sum_{i\in \vY}u_{i,j}, \qquad s=\sum_{\substack{i\in \vY\\ j\notin \vY}}u_{i,j}. \]
Thus, since $|\vY\cap \vY_{i,j}|= k-1$ for every $i$ and $j$,~\cref{eq: M-decompositon} gives

\[0=(Mf)(\vY)=a -\sum_{\substack{i\in \vY\\ j\notin \vY}}\frac{u_{i,j}}{N}+0=a-\frac{s}{N}\]
and thus, $a=\frac{s}{N}$.

Now evaluating at $\vY_{i,j}$; we want those vertices $\vZ\in C_\vY\setminus \{\vY\}$ such that $|\vY_{i,j}\cap \vZ|=k-1$. These are exactly $\vY$ , those $\vY_{i,\ell}$ such that $j\neq \ell$ (we are replacing $i$ with a different $\ell$) and $\vY_{h,j}$ with $i\neq h$ (we are replacing a different $h$ with $j$). The sum over these neighbours is $a+r_i+c_j-2u_{i,j}$.

We also know that $|\vY_{i,j}\cap \vY_{h,\ell}|=k-2$ precisely when $i\neq h$ and $j\neq \ell$, so the sum over these neighbours is $s-r_i-c_j+u_{i,j}$ and thus we write
\[0=(Mf)(\vY_{i,j})=u_{i,j}-\frac{a+r_i+c_j-2u_{i,j}}{N}+\frac{2(s-r_i-c_j+u_{i,j})}{N(N-1)}.\]
Simplifying and substituting $a=s/N$ gives
\begin{equation}\label{eq: Nuix}
    Nu_{i,j}=r_i+c_j-\frac{s}{N}.
\end{equation}
Summing~\cref{eq: Nuix} over $i\in \vY$ gives $Nc_j=s+kc_j-ks/{N}$ and thus $c_j= s/N$ since $N\neq k$ for every $j\notin \vY$. Substituting back into~\cref{eq: Nuix} yields
\[u_{i,j}=\frac{r_i}{N}=\varphi_i,\]
where $\varphi_i$ depends only on $i$. In particular,
\[a=\frac{s}{N}=\sum_{i\in \vY}\frac{r_i}{N}=\sum_{i\in \vY}\varphi_i.\]
Thus, any $f\in \ker M$ is determined by the $k$ parameters $(\varphi_i)_{i\in \vY}$, and $\dim(\ker M)$ is at most  $k$.

Now consider any $f\in \Span\{\chi_\vR:\vR\in\binom {\vY}{k-1}\}$ and fix $\vA\in N_{\overline{J(n,k)}}(\vY)$. Since $\vA$ is a neighbour of $\vY$, $T_\vA\cap T_\vY=\emptyset$. Hence, $\chi_\vR(\vA)=0$ for every component $\chi_\vR$ of $f$. Thus $f(\vA)=0$ for every $\vA\in N_{\overline{J(n,k)}}(\vY)$. Therefore, $f$ is supported on the non-neighbours of $\vY$ and thus $\ker M$ contains, as a subspace, $\Span\{\chi_\vR \mid \vR\in\tbinom {\vY}{k-1}\}$.
By~\Cref{lem:johnson-modules}, the $k$ functions $\chi_\vR$ are linearly independent and thus $\ker M$ has dimension at least $k$.
Thus,
\begin{equation}\label{eq:A=M=B}
    \{f\in\incspace\mid\operatorname{supp}(f)\subseteq C_\vY\} =\ker M=\Span\{\chi_\vR \mid \vR\in\tbinom {\vY}{k-1}\}
\end{equation}
and the representation $T$ fulfils the hypotheses of~\Cref{lem:local-incidence-edge-kernel}. By said lemma, $ \edgekernel
 =\Span\{\chi_\vR\otimes\chi_\vR \mid \vR\in\tbinom{[n]}{k-1}\}$.
 
The functions $\chi_\vR$ are linearly independent by
\Cref{lem:johnson-modules}.  Moreover, for every $\vA\in X$, the
set $[n]\setminus \vA$ has $N>k$ elements and contains two distinct
$k$-sets, both of which are neighbours of $\vA$.  Therefore
\Cref{lem:diagonal-feature-rigidity} applies and proves tameness. The connected non-bipartite graph $\KG(n,k)$ is a spanning subgraph of
$\overline{J(n,k)}$.  Hence the latter graph is connected and
non-bipartite.  By~\Cref{lem:delsarte-iff-schrijver}, this Delsarte-rigid representation is Schrijver-rigid, and~\Cref{thm:criterion} applies.
\end{proof}

\begin{proposition}\label{Prop:Delsarte necessary}
If $k\geq3$, then  $\overline{J(n,k)}$ has no Hoffman-rigid
representation. For $k=2$, the disjointness representation in \Cref{prop:johnson-complement} is Hoffman-rigid. 
\end{proposition}
\begin{proof}  
An independent set in $\overline{J(n,k)}$ is a family of
$k$-sets with pairwise intersections of size at least $k-1$. Consider distinct $\vA,\vB$ in any such family $\mathcal F$. There are two cases. The first case is that every $\vY\in \mathcal F$ contains $\vA\cap\vB$ and so it has maximum size $N+1$. The second case is that $\vY\subseteq \vA\cup\vB$ for every $\vY\in \mathcal F$, and it then has a maximum size of $\binom{k+1}{k}=k+1$. Thus $\alpha(\overline{J(n,k)})=\max\{N+1,k+1\}=N+1$.

The graph $\overline{J(n,k)}$ has degree $v-1-\binom{k}{k-1}\binom{N}{1}=v-1-kN$ and from~\cite[Theorem 6.3.2]{godsil2016erdos}, the Johnson graph has eigenvalues $P_{j1}=kN-j(n+1-j)$ on the spaces $E_j$, and these decrease strictly with $j$.
Since the adjacency matrix of $\overline{J(n,k)}$ is $J - I - A_1$ and $J = vE_0$, its eigenvalue on $E_j\mathbb{C}^X$ is $v\delta_{j,0} - 1 - P_{j,1}$. Hence the least eigenvalue of $\overline{J(n,k)}$ is attained only on $E_1\mathbb{C}^X$ and equals $\tau = -1 - P_{1,1} = -1 - kN + n = -(k - 1)(N - 1)$.

The Hoffman bound of $\overline{J(n,k)}$ is therefore 

\[\frac{(k-1)(N-1)v}{(v-1-kN)+(k-1)(N-1)}=\frac{(k-1)(N-1)v}{v-k-N}=\frac{v(k-1)(N-1)}{v-n}.\]
This equals $N+1$ for $k=2$, and is strictly larger for $k\geq3$.
When $k=2$, the graph is $KG(n,2)$ and the representation in~\Cref{prop:johnson-complement} is the representation from~\Cref{prop:kneser-criterion} which is Hoffman-rigid. When $k \geq 3$, the Hoffman bound is not tight, so $\overline{J(n,k)}$ admits no Hoffman-rigid representation by~\Cref{rem:hoffman-tightness}.
\end{proof}

\subsection{The Grassmann scheme}
\label{sec:grassmann-scheme}

For a prime power $q$ and $V=\mathbb F_q^n$, write
\[
                   \Gr_q(n,k)\coloneqq \{U\leq V \mid \dim U=k\},
\]
which is the vertex set for the Grassmann scheme on parameters $q$,$n$, and $k$; we study here the two derived graphs shown in~\Cref{tab:AS-graphs} in the range $n>2k$.

Define
\[
    \qbinom{a}{b} \coloneqq  % 
    \prod_{i=0}^{b-1}\frac{q^{a-i}-1}{q^{b-i}-1},
    \qquad
    [a]_q\coloneqq \qbinom{a}{1}
\]
as the Gaussian binomial coefficient (number of $b$-dimensional subspaces of $\mathbb F^a_q$) and the number of one-dimensional
subspaces of $\mathbb F_q^a$, respectively. Also,
$\qbinom ab=0$ when $b>a$.

For a $j$-subspace $R\leq V$, let
$\chi_R(U)=\indfunction{R\leq U}$, and put
\[
       \mathcal W_j\coloneqq \operatorname{span}\{\chi_R \mid \dim R=j\}
       \subseteq\mathbb C^X,
       \qquad \mathcal W_{-1}\coloneqq 0.
\]
By \cite[Theorem~9.4.1]{godsil2016erdos}, we get the following lemma.

\begin{lemma}
\label{lem:grassmann-modules}
    For $0\le j\le k$,
    \[
    \mathcal W_j=\bigoplus_{i=0}^{j}E_i\mathbb C^X,
    \qquad
    \dim \mathcal W_j=\qbinom{n}{j}.%
    \]
    Consequently, the functions $\chi_R$ form a basis of $\mathcal W_j$.
\end{lemma}
\noindent This is a direct analogue to~\Cref{lem:johnson-modules} for use in the proofs of~\Cref{thm:qKneser} and~\Cref{prop:grassmann-complement} in~\Cref{subsec:q-kneser-verification} and~\Cref{subsec:grassmann-complement}.
Hence, the methods we use for the proofs of this section are analogous to those in~\Cref{sec:johnson-scheme}.
In addition, the proof proceeds in the same manner in both cases; the key difference between them being that all binomial coefficients of the first version are to be replaced with q-binomial coefficients in the latter version. \Cref{lem:johnson-modules} is obtained from~\cite[\S~6]{godsil2016erdos}, while~\Cref{lem:grassmann-modules} adapts results from~\cite[\S~9.4]{godsil2016erdos}. %

\subsubsection{\texorpdfstring{$q$}{q}-Kneser graphs}
\label{subsec:q-kneser-verification}

\begin{figure}[t]
  \centering
  \includegraphics[width=.75\linewidth]{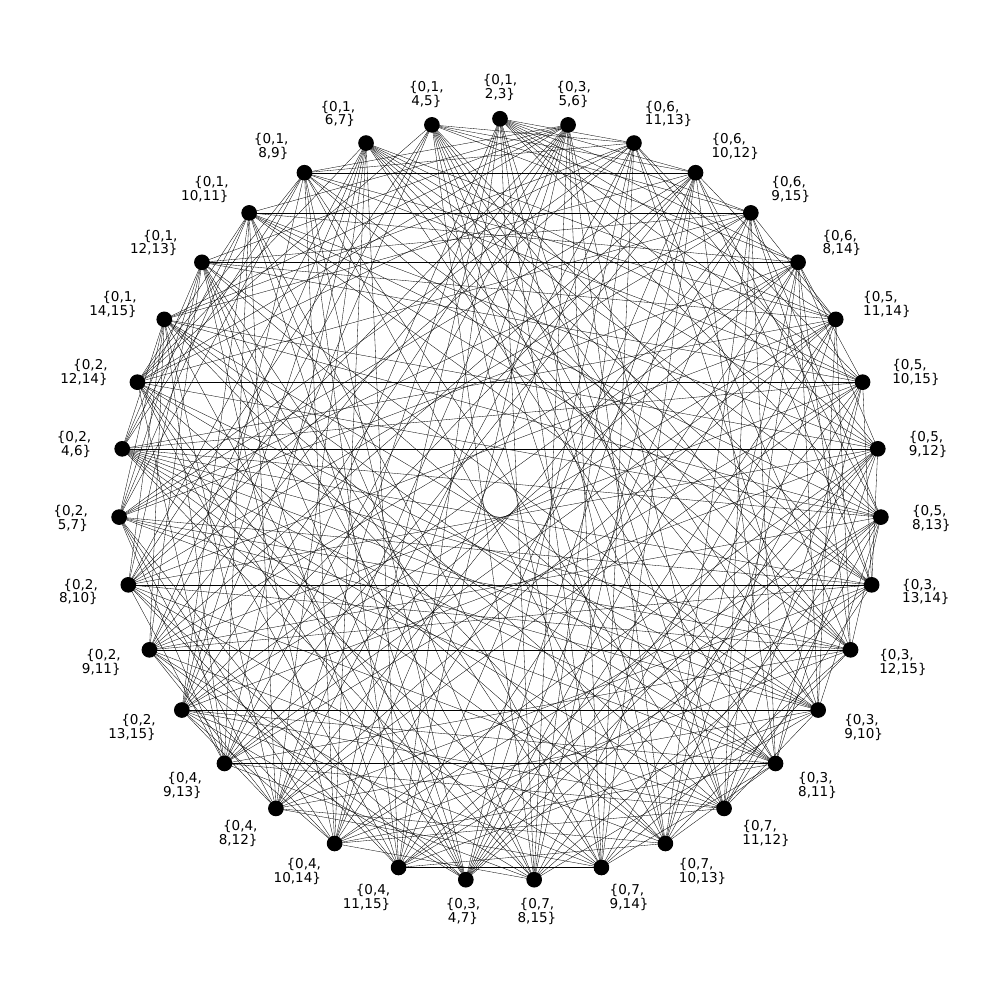}
  \caption{The graph $\operatorname{KG}_2(4,2)=\overline{J_2(4,2)}$, which is both a $q$-Kneser graph and the complement of a Grassmann graph.  Its $35$ vertices are the two-dimensional subspaces of $\mathbb F_2^4$.  We identify each integer $r\in\{0,\ldots,15\}$ with its four-bit binary representation in $\mathbb F_2^4$; each vertex label lists the four elements of the corresponding subspace.  Two vertices are adjacent precisely when the corresponding subspaces intersect only in $0$.  The graph has $280$ edges.}
  \label{fig:q-kneser-grassmann-shared-2-4-2}
\end{figure}

For every subspace $W\leq V$, let
$\mathcal P(W)\coloneqq \{p\leq W \mid \dim p=1\}$.
Recall that the $q$-Kneser graph $\operatorname{KG}_q(n,k)$ has the $k$-dimensional subspaces of $V$ as its vertices, with
$U\sim W$ if and only if $U\cap W=0$.
It is $q^{k^2}\qbinom{n-k}{k}$-regular.

\begin{theorem}
\label{thm:qKneser}
For every $n > 2k \geq 2$, the map $U \mapsto \mathcal{P}(U)$ for $U \in \Gr_q(n,k)$ is a Hoffman-rigid tame representation of $\operatorname{KG}_q(n,k)$.
Consequently, the graph admits a
commutativity gadget and the corresponding $\qCSP$ is $\RE$-complete.
\end{theorem}

\noindent The proof of this theorem uses the following lemma to compute the edge kernel of the representation.

\begin{lemma}
\label{lem:q-disjoint-incidence}
Let $Z\leq V$ have dimension $s$, let $t\geq1$, and suppose that
$n>s+t$.  Form the zero-one matrix $C$ whose columns are indexed by the
points $p\in\mathcal P(V)$ not contained in $Z$, whose rows are indexed by
the $t$-subspaces $W\leq V$ satisfying $W\cap Z=0$, and whose $(W,p)$-entry
is $1$ precisely when $p\leq W$.  Then $C$ has full column rank over
$\mathbb C$.
\end{lemma}
\begin{proof}
Choose a linear complement $Y$ of $Z$, and put $m=\dim Y=n-s>t$.
For each $\ell\in\mathcal P(Y)$, fix $0\ne y_\ell\in\ell$.
Each of the $|\mathcal P(Y)| \cdot |Z|$ points outside $Z$ has a unique expression
\[
    p_{\ell,z} = \Span\{y_\ell + z \}
\]
for each of the $\ell \in \mathcal{P}(Y)$ and $z \in Z$.
If $t=1$, its
rows and columns are both indexed by the points outside $Z$, and thus
$C$ is the identity matrix. Assume henceforth that $t\geq2$.

Every $t$-subspace $W$ of $V$ disjoint from $Z$ is the graph of a linear map
$L\colon A\to Z$, where $A\leq Y$ has dimension $t$ and $W \coloneqq \{a + L(a) \mid a \in A\}$.
Now, with the definition of $C$, construct the Gram matrix $C^*C$, where
\[
    C^*C_{p_{\ell,z}, p_{\ell',z'}} = \sum_{W} \overline{C_{W,p_{\ell,z}}}C_{W,p_{\ell',z'}}.
\]
Thus, the
$(p_{\ell,z},p_{\ell',z'})$-entry of $C^*C$ counts pairs
$(A,L)$ such that
\[
 \ell,\ell'\leq A,\qquad L(y_\ell)=z,\qquad
 L(y_{\ell'})=z'.
\]
Put $h=|Z|=q^s$ and $u=|\mathcal P(Y)|=[m]_q$.  There are three possibilities:

\begin{itemize}
    \item If $(\ell,z)=(\ell',z')$, the entry is on the diagonal of $C^*C$.
    Thus, only the one-dimensional subspace $\ell\le A$ is fixed, which gives entry
    \[
       a=h^{t-1}\qbinom{m-1}{t-1}.
    \]
    \item If $\ell=\ell'$ but $z\ne z'$, computing the entry would require $L(y_\ell) = z$
    and $L(y_\ell) = z'$ for some $y_\ell$. This is impossible, thus the entry is $0$.
    \item If $\ell\ne\ell'$, then $y_\ell$ and $y_{\ell'}$ are linearly independent. Hence, the entry is
    \[
       b=h^{t-2}\qbinom{m-2}{t-2}.
    \]
\end{itemize}
These numbers are obtained by first choosing $A$ and then
counting linear maps $L\colon A\to Z$ with respectively one or two
prescribed values.
Grouping the entries of $C^*C$ into blocks according to $\ell, \ell'\in\mathcal P(Y)$, we address two cases:

\begin{itemize}
    \item If $\ell = \ell'$, then according to the first two cases above, the entry is $a$ if $z = z'$ (the diagonal entries) and $0$ otherwise. Thus, in this case, we obtain the block matrix $aI_h$.
    \item If $\ell\ne\ell'$, then according to case 3 from the above listing, all entries are $b$. Thus, we obtain the all-$b$ matrix $bJ_h$.
\end{itemize}

\noindent Altogether, we have

\[
                  C^*C=aI+b\bigl(J-I_u\otimes J_h\bigr),
\]

\noindent where $I$ and $J$ are the $uh \times uh$ identity and all-ones matrices, respectively.
We now compute the spectrum of $C^*C$. Note first that $J$ and $I_u\otimes J_h$ commute, so the spectrum of $J-I_u\otimes J_h$ is $$\{h(u-1),(-h)^{(u-1)},0^{(u(h-1))}\}$$
(where the exponent denotes the multiplicity). Thus, the spectrum of $C^*C$ is
\[
\{a+bh(u-1),(a-bh)^{(u-1)},a^{(u(h-1))}\}.
\]
The positivity of the first and third types of eigenvalues follows immediately, while
\[
\begin{aligned}
 a-bh
 &=h^{t-1}\qbinom{m-1}{t-1} - h^{t-2}\qbinom{m-2}{t-2}h \\
 &=h^{t-1}\left(
     \qbinom{m-1}{t-1}-\qbinom{m-2}{t-2}\right) \\
 &=h^{t-1}q^{t-1}\qbinom{m-2}{t-1}>0,
\end{aligned}
\]
where the inequality holds because of $m>t$.  Hence, $C^*C$ is positive definite, so $C$ has
full column rank.
\end{proof}

\begin{proof}[Proof of Theorem~\ref{thm:qKneser}]
Put $X=\Gr_q(n,k)$ and $P=\mathcal P(V)$.  For $p\in P$ set
\[
                  S_p=\{U\in X \mid p\leq U\},
                  \qquad \chi_p(U)=\indfunction{U\in S_p}.
\]
Each $S_p$ is a nonempty independent set, and every $U\in X$ belongs to exactly
$[k]_q$ of these independent sets.  Thus
\begin{equation}\label{eq:qk4}
    \sum_{p\in P}\chi_p=[k]_q\mathbf1.
\end{equation}
Moreover, the feature set of $U$ is exactly its set of one-dimensional
subspaces, so the feature map is injective.

We next verify the spectral condition.  
By~\cite[\S~9.7.1] {godsil2016erdos}, the eigenvalues of ${KG}_q(n,k)$ are
\begin{equation}\label{eq:q-kneser-eigenvalues}
    \lambda_j=(-1)^j q^{\binom j2+k(k-j)}
                  \qbinom{n-k-j}{k-j},
\end{equation}
with corresponding eigenspaces
\begin{equation}\label{eq:q-kneser-eigenspaces}
    \mathcal E_j=\mathcal{W}_j\cap \mathcal{W}_{j-1}^\perp.
\end{equation}
As all the eigenvalues $\lambda_j$ are distinct, they are the primitive modules $E_i\mathbb C^X$ of the Grassmann scheme.
Furthermore, \cref{eq:q-kneser-eigenvalues} shows that $\lambda_1$ is the least eigenvalue, with \cref{eq:q-kneser-eigenspaces} giving its corresponding eigenspace $\mathcal E_\tau=\mathcal{W}_1\cap \mathcal{W}_0^\perp$.

Since $\mathcal W_1=\operatorname{span}\{\chi_p \mid p\in P\}$ and~\cref{eq:qk4}
puts $\mathbf1$ in $\mathcal W_1$, \Cref{lem:grassmann-modules} and~\cref{eq:q-kneser-eigenspaces}
show that
\[
       \incspace\coloneqq \operatorname{span}\{\chi_p \mid p\in P\}
             =\mathbb C\mathbf1\oplus\mathcal E_\tau.
\]
Thus the feature map is a Hoffman-rigid representation.  Since the graph is
connected and non-bipartite, \Cref{prop:hoffman-implies-delsarte} makes
it a Delsarte-rigid representation.
Lemma~\ref{lem:grassmann-modules} also shows
that the functions $(\chi_p)_{p\in P}$ are linearly independent, as they form a basis.

We now determine the edge kernel.  Because the $\chi_p$
are linearly independent, every $F\in\incspace\otimes\incspace$ has a
unique expression
\begin{equation}\label{eq:qk12}
    F(U,W)=\sum_{p,r\in P}c_{p,r}\chi_p(U)\chi_r(W).
\end{equation}
Suppose that $F(U,W)=0$ whenever $U\cap W=0$.  Fix $U$ and define, for
$r\not\leq U$,
\[
                         f_U(r)=\sum_{p\leq U}c_{p,r}.
\]
For every $k$-subspace $W$ disjoint from $U$,
\cref{eq:qk12} says
\[
                         \sum_{r\leq W}f_U(r)=0.
\]
Apply Lemma~\ref{lem:q-disjoint-incidence} with $Z=U$ and $t=k$; this is possible as the hypothesis is that $n>2k$.
We obtain
$\sum_{p\leq U}c_{p,r}=0$ whenever $r\not\leq U$.
Now fix $r$.  Applying the same lemma with $Z=r$ and $t=k$
to the above gives
$c_{p,r}=0$ for $p\ne r$.
Consequently,
\begin{equation}\label{eq:qk15}
    \edgekernel
    =\operatorname{span}\{\chi_p\otimes\chi_p \mid p\in P\}.
\end{equation}
One inclusion is immediate: if $U\cap W=0$,
no point
$p$ is contained in both $U$ and $W$.

For every $Z\in X$, a complement of $Z$ has
dimension $n-k>k$ and therefore contains two distinct $k$-subspaces,
both of which are neighbours of $Z$.  Hence
\Cref{lem:diagonal-feature-rigidity} applies, which proves tameness
and gives
\[
 \diagonalSubspace=\mathbb C\omega_T,
 \qquad
 \omega_T(U,W)=\sum_{p\in P}\chi_p(U)\chi_p(W)
               =[\dim(U\cap W)]_q.
\]
By~\Cref{lem:delsarte-iff-schrijver}, the Delsarte-rigid representation is
Schrijver-rigid, so Theorem~\ref{thm:criterion} completes the
proof.
\end{proof}

\begin{remark}\label{rem:q-kneser}
The strict inequality $n > 2k$ is essential to this proof when $k\geq2$.
If $n=2k$, $\dim Z=k$, and $k\geq2$, then $\dim Y=t=k$, so in~\Cref{lem:q-disjoint-incidence}, the last inequality used to prove that the eigenvalue $a - bh$ is positive does not hold.
Thus, the full-rank conclusion, and hence~\Cref{eq:qk15}, does not follow by this method.
\end{remark}

\subsubsection{Complements of Grassmann graphs}
\label{subsec:grassmann-complement}
Let $J_q(n,k)$ be the Grassmann graph on the $k$-subspaces of
$V$, where two
vertices are adjacent when their intersection has dimension $k-1$.

\begin{proposition}
\label{prop:grassmann-complement}
If $n>2k\geq4$, then $\overline{J_q(n,k)}$ has a Delsarte-rigid
tame representation.  Consequently, the graph admits a
commutativity gadget and the corresponding $\qCSP$ is $\RE$-complete.
\end{proposition}

\begin{proof}
Put $N=n-k$, $X=\Gr_q(n,k)$, and $v=\qbinom nk$.
Similarly as in the case of the complement of the Johnson graph, we use the
$(k-1)$-subspaces of $\F_q^n$ as features and let $
 S_R=\{A\in X \mid R\leq A\}$ and $T_A=\{R\leq A \mid \dim R=k-1\}$.
This is an injective disjointness representation, because two
adjacent vertices of $\overline{J_q(n,k)}$ have no common $(k-1)$-subspace, and a
$k$-space is determined by its hyperplanes.  Moreover,
\[
 |T_A|=[k]_q,\qquad
 \mu=\frac{[k]_q}{\qbinom n{k-1}}
     =\frac{[N+1]_q}{v}.
\]
This, together with~\Cref{lem:grassmann-modules}, gives
\begin{equation}\label{eq:cg3}
    \incspace
       =\operatorname{span}\{\chi_R \mid R \leq V, \dim R=k-1\}
       =E_0\mathbb C^X\oplus\cdots\oplus E_{k-1}\mathbb C^X.
\end{equation}
Thus, $T$ is scheme-compatible with index set $\{1,\dots,k-1\}$.
Additionally, the feature functions are linearly independent.

We now build the matrix $B \in \mathbb{C}[\mathfrak{X}]$ in order to prove Delsarte-rigidity using~\Cref{prop:bose-mesner-spectral-test}.
Put
\[
 m_k=\operatorname{rank}(E_k)=v-\qbinom n{k-1}
\]
as well as
\[
B= \kappa E_0 - \sum_{i=1}^{k-1}E_i + \theta E_k,
\]
where we want to determine values for $\kappa>1$ and $-1<\theta<\kappa$
so that  $B\in\Span\{A_i\mid i\in\Gamma\}$.
As we know $\sum_{i=0}^k E_i = I$, this can be rewritten as
\begin{equation}\label{eq:grassmann-b}
    B = (\kappa + 1)E_0 + (\theta + 1)E_k - I.
\end{equation}

We will now compute the change of basis from $(E_j)_j$ to $(A_i)_i$ in order to determine $\kappa$ and $\theta$. Using $j = k$ with \cite[Theorem~9.6.3]{godsil2016erdos} gives $P_{ki} = (-1)^iq^{\binom{i}{2}}\qbinom{k}{i}$, and from~\cite[Corollary~9.3.3]{godsil2016erdos} we see that $|R_i|=vq^{i^2}\qbinom{k}{i}\qbinom{N}{i}$. By substituting these values into~\ref{eq: Q_{i,j}} (which holds for arbitrary schemes), we get

\begin{equation}\label{eq:grassmann-q}
  Q_{ik} = (-1)^i \frac{m_k}{q^{\binom{i+1}{2}}\qbinom{N}{i}}.  
\end{equation}
We can now compute $B$ with a change of basis:
\[
B = \frac{(\kappa + 1)}{v}\sum_{i=0}^kA_i + \frac{(\theta + 1)}{v}\sum_{i=0}^kQ_{ik}A_i - A_0.
\]
Substituting $i = 0$ and $i = 1$ into~\cref{eq:grassmann-q} yields
\[
Q_{0k} = m_k,\qquad Q_{1k} = -\frac{m_k}{q[N]_q}.
\]
Furthermore, we get the following coefficients corresponding to each substitution:
\[
c_0 = \frac{\kappa + 1 + (\theta + 1)m_k}{v} - 1,\qquad c_1 = \frac{\kappa + 1 - (\theta + 1)(m_k/(q[N]_q))}{v}.
\]
Since $B \in \Span\{A_2,\dots,A_k\}$, it follows that $c_0 = c_1 = 0$. This gives
\[
\kappa + 1 + (\theta + 1)m_k = v,\qquad \frac{(\theta + 1)m_k}{q[N]_q} = \kappa + 1.
\]
Solving for $(\kappa + 1)$ and $(\theta + 1)$ and then substituting into~\Cref{eq:grassmann-b}, we get
\[
    B=\frac v{[N+1]_q}E_0+
     q^{1-k}\frac{[N]_q}{[N-k+1]_q}E_k-I.
\]
By construction, $B \in \Span\{A_2,\dots,A_k\}$ and its eigenvalues on $E_0$, $E_1 \oplus \cdots \oplus E_{k-1}$ and $E_k$ respectively are
\[
\kappa = \frac{v}{[N + 1]_q} - 1,\qquad -1,\qquad \theta = q^{1-k}\frac{[k-1]_q}{[N - k + 1]_q}.
\]
For $\kappa + 1$, this yields
\[
\kappa + 1  = \frac{v}{[N + 1]_q}
            = \frac{\qbinom{n}{k}}{[N + 1]_q}
            = \frac{1}{[N + 1]_q}\qbinom{n}{k}
            = \frac{1}{[N + 1]_q}\frac{[N + k]_q\cdots[N + 1]_q}{[k]_q!}
            =\prod_{i=2}^k\frac{q^{N+i}-1}{q^i-1}>2.
\]
It follows that, when we consider $\theta$, we get:
\[
0< q^{1-k}\frac{[k-1]_q}{[N-k+1]_q}
   <q^{1-k}[k-1]_q<1<\kappa.
\]
Thus, $\kappa>1$ and $0<\theta<\kappa$. Together with~\Cref{eq:cg3}, this verifies~\Cref{eqn_1611_2508} for $\Lambda=\{1,\ldots,k-1\}$, so
\Cref{prop:bose-mesner-spectral-test} proves that the representation $T_A$ is
Delsarte-rigid.

In order to complete the proof we still need to show that the representation is tame. For this, we want to apply~\Cref{lem:local-incidence-edge-kernel}, so we should investigate those functions supported on
\[C_Y=X\setminus N_{\overline{J_q(n,k)}}(Y)=\{A\in X\mid \dim(A\cap Y)\ge k-1\}.\]
Since two distinct $k$-subspaces cannot intersect in dimension greater than $k-1$, we have
\[C_Y=\{Y\}\cup N_{J_q(n,k)}(Y).\]
We now specify an arbitrary vertex in $N_{J_q(n,k)}(Y)$. Fix a subspace $C\le V$ such that
$V=Y\oplus C$.
Consider any $R\le Y$ with $\dim(R)=k-1$, any point $p=\Span\{x_p\}$ of $C$, and an element $\alpha\in\mathbb F_q$. If we fix a representative $e_R$ such that
$Y=R\oplus\Span\{e_R\}$,
then the subspaces
\[Y_{R,p,\alpha}=R\oplus\Span\{ x_p+\alpha e_R\}\]
are exactly the Grassmann neighbours of $Y$. In particular, $|C_Y\setminus\{Y\}|=[k]_qq[N]_q$.

Next we consider distinct vertices $A=Y_{R,p,\alpha}$ and  $ A'=Y_{R',p',\alpha'}$ in $N_{J_q(n,k)}(Y)$ and determine their intersection sizes.
If $R=R'$, then $R\le A\cap A'$. Since $A$ and $A'$ are distinct $k$-subspaces, their intersection has dimension $k-1$ and so $A\sim A'$ in $J_q(n,k)$.
Suppose now that $p=p'$ and $R\neq R'$. Then $\dim(R\cap R')=k-2$ and $R+R'=Y$. Write
    \[ A=R\oplus\Span\{ x_p+\alpha e_R\}, \qquad A'=R'\oplus\Span\{ x_p+\alpha'e_{R'}\}.\]    
Since $\alpha'e_{R'}-\alpha e_R\in Y=R+R'$, there exist $r\in R$ and $r'\in R'$ such that $\alpha'e_{R'}-\alpha e_R=r-r'$. We can therefore define the vector $ w=x_p+\alpha e_R+r=x_p+\alpha'e_{R'}+r'$.
Then  $w\in A\cap A'$, since $r\in R\le A$ and $r'\in R'\le A'$. Moreover, $w\notin Y$, because its $C$-component is the nonzero vector $x_p$. Hence,
$(A\cap A')\supseteq(R\cap R')$.
The right-hand side has dimension $k-2$ and, since $A\neq A'$, we also have $\dim(A\cap A')= k-1$, so $A\sim A'$ in $J_q(n,k)$.

Finally, suppose that $R\neq R'$ and $p\neq p'$. If $w\in A\cap A'$, then the $C$-component of $w$ lies in $p\cap p'=0$, so $w\in Y$. Consequently,
\[A\cap A'=(A\cap Y)\cap(A'\cap Y)=R\cap R',\]
and hence $\dim(A\cap A')=k-2$.
We can therefore deduce that $\dim(A\cap A')=k-1$ if, and only if, $R=R'$ or $p=p'$. In particular, the field elements $\alpha$ and $\alpha'$ have no influence on adjacency. For fixed $R$ and $p$, the $q$ vertices $Y_{R,p,\alpha}$ form a clique.

With this analysis in hand, the rest of the argument is that of~\Cref{prop:johnson-complement}. We prove
\begin{equation}\label{eq:cgek1}
    \{f\in\incspace:\operatorname{supp}(f)\subseteq C_Y\}=\Span\{\chi_R \mid R\leq Y,\ \dim R=k-1\}.
\end{equation}
By~\cref{eq:cg3}, $\incspace=\ker(E_k)$, so a function supported on $C_Y$ lies in $\incspace$ exactly when the restriction to $C_Y$ lies in $\ker E_{k,Y}$. We set 
\[M= \frac{v}{m_k}E_{k,Y}.\]
As before, any two elements of $C_Y$ intersect in dimension at least $k-2$, so we use~\cref{eq:grassmann-q} to deduce that
\begin{equation}\label{eq:cgek2}
    M=I-\frac{1}{q[N]_q}A_{1,Y}+\frac{1}{q^3\qbinom{N}{2}}A_{2,Y}
\end{equation}
where $A_{s,Y}$ denotes the restriction of $A_s$ to $C_Y$.

Let $f\in \ker(M)$ and write $f(Y)=a$ and $f(Y_{R,p,\alpha})=u_{R,p,\alpha}$; we use the following sums in our calculation:
\[w_{R,p}=\sum_{\alpha}u_{R,p,\alpha}, \qquad r_R=\sum_pw_{R,p}, \qquad c_p=\sum_Rw_{R,p}, \qquad s=\sum_{R,p}w_{R,p}.\]
Evaluating ~\cref{eq:cgek2} at $Y$ gives $0=(Mf)(Y)=a-\frac{s}{q[N]_q}$ and thus $a=\frac{s}{q[N]_q}$.

Now fix a subspace $R$ and a point $p$; by the previous analysis, the elements of $C_Y$ such that $\dim(Y_{R,p,\alpha}\cap Y_{R',p',\alpha'})=k-1$ are $Y$ and those vertices such that $R=R'$ or $p=p'$. Thus the sum of neighbours is $a+r_R+c_p-w_{R,p}-u_{R,p,\alpha}$ at intersection-one and $s-r_R-c_p+w_{R,p}$ at intersection-two. 
Thus,
\begin{equation}\label{eq:cgek3}
    0=(Mf)(Y_{R,p,\alpha})=u_{R,p,\alpha}-\frac{a+r_R+c_p-w_{R,p}-u_{R,p,\alpha}}{q[N]_q}+\frac{s-r_R-c_p+w_{R,p}}{q^3\qbinom{N}{2}}.
\end{equation}

Now for the only new step in the proof: subtracting the equations at $Y_{R,p,\alpha}$ and $Y_{R,p,\beta}$ gives $(u_{R,p,\alpha}-u_{R,p,\beta})(1+\frac{1}{q[N]_q})=0$, so $u_{R,p,\alpha}=u_{R,p,\beta}$ for all $\alpha,\beta\in \mathbb F_q$.  We conclude that $u_{R,p,\alpha}$ does not depend on $\alpha$ and, so, $u_{R,p,\alpha}=w_{R,p}/q$ .
Substituting this together with $a = s/(q[N]_q)$ into~\cref{eq:cgek3} and collecting the terms (using $\qbinom{N}{2} = [N]_q [N-1]_q / (q+1)$ and $[N]_q = 1 + q[N-1]_q$ ) gives 

\[0=\frac{q[N]_q + 1}{q^2[N]_q([N]_q - 1)}\left([N]_qw_{R,p} - r_R - c_p + \frac{s}{[N]_q}\right),\]
and the prefactor is nonzero, so $[N]_q w_{R,p} = r_R + c_p - s/[N]_q$.

Summing over the $[k]_q$ spaces $R \leq Y$ gives $[N]_q c_p = s + [k]_q c_p - [k]_q s/[N]_q$; since $N>k$, $c_p = s/[N]_q$. Substituting back yields $u_{R,p,\alpha} = r_R/(q[N]_q) = \varphi_R$, where $\varphi_R$ is some parameter that depends only on $R$. 
In particular, $a = s/(q[N]_q) = \sum_R \varphi_R$, so any $f \in \ker M$ is determined by the $[k]_q$ arbitrary parameters $(\varphi_R)_{R \leq Y}$ and thus $\dim(\ker M)$ is at most $[k]_q$.

We finish by showing that $\ker M$ contains, as a subspace, the $[k]_q$-dimensional space $\Span\{\chi_R \mid R\leq Y,\ \dim R=k-1\}$. Consider any $f$ in the given span and fix $Z \in N_{\overline{J_q(n,k)}}(Y)$. Since $Z\sim Y$ in $\overline{J_q(n,k)}$, $T_Z \cap T_Y = \emptyset$, so $\chi_R(Z) = 0$ for every component $\chi_R$ of $f$. Thus $f(Z) = 0$ for every $Z \in N_{\overline{J_q(n,k)}}(Y)$. Therefore, $f$ is supported on the non-neighbours of $Y$ and thus $f\in \ker M$ (since $f\in \ker E_k$). By~\Cref{lem:grassmann-modules}, the  $[k]_q$ functions $\chi_R$ are linearly independent and thus 

\[\{f\in\incspace:\operatorname{supp}(f)\subseteq C_Y\}=\ker M=\Span\{\chi_R \mid R\leq Y,\ \dim R=k-1\}.\]

The second and third assumptions of~\Cref{lem:local-incidence-edge-kernel} can easily be seen to be true: every $S_R$ has $[N + 1]_q \geq 3$ elements and if $A, A' \in S_R$ are distinct, then $A \cap A' = R$ so, $T_A \cap T_{A'} = \{R\}$. Hence the representation $T$ fulfils the hypotheses of~\Cref{lem:local-incidence-edge-kernel}. By said lemma, $K = \Span\{\chi_R \otimes \chi_R \mid R \leq V,\ \dim R = k - 1\}$.
The functions $\chi_R$ are linearly independent by~\Cref{lem:grassmann-modules}. Moreover, for every $A \in X$, a complement of $A$ has dimension $N > k$ and contains two distinct $k$-subspaces, both of which are neighbours of $A$. Therefore~\Cref{lem:diagonal-feature-rigidity} applies and proves tameness. The connected non-bipartite graph $KG_q(n, k)$ is a spanning subgraph of $\overline{J_q(n, k)}$. Hence the latter graph is connected and non-bipartite. By~\Cref{lem:delsarte-iff-schrijver}, this Delsarte-rigid representation is Schrijver-rigid, and~\Cref{thm:criterion} applies.
\end{proof}

\begin{proposition}
    If $k \geq 3$, then $\overline{J_q(n, k)}$ has no Hoffman-rigid representation. For $k = 2$, the disjointness
representation in~\Cref{prop:grassmann-complement} is Hoffman-rigid.
\end{proposition}

\begin{proof}
    An independent family in $\overline{J_q(n, k)}$ is a family of $k$-spaces meeting pairwise in dimension at least $k-1$. Consider distinct $k$-spaces $U, V$ in any such family $\mathcal{F}$. Then, for every $Z \in \mathcal{F}$, there are two cases: if $Z$ contains $U \cap V$ then it has maximum size $[N + 1]_q$. On the other hand, if this is not true, then $Z \cap U$ and $Z \cap V$ span $Z$, so for basis $\{u_i~ |~ i \in [k]\}$ of $U$ and basis $\{v_i~ |~ i \in [k]\}$ of $V$ we get $Z \leq \Span\{\{u_i~ |~ i \in [k]\} \cup \{v_i~ |~ i \in [k]\}\}$ and its size is $[k + 1]_q$.
    It follows that
    \[\alpha(\overline{J_q(n,k)}) = \max\{[N + 1]_q, [k + 1]_q\} = [N + 1]_q.\]
    Furthermore, we write $P_{ij1}$ for the eigenvalues on spaces $E_j$ for graph $J_q(n,i)$.
    From \cite[Chapter 9]{godsil2016erdos}, the Grassmann graph has the following eigenvalues on $E_j$:
    \begin{align}
        P_{kj1}    &= P_{(k-1)j1} + [n - k + 1]_q - [k]_q \nonumber\\
        &= P_{jj1} + \sum_{i = j+1}^k[n - i + 1]_q - [i]_q \nonumber\\
        &= -[j]_q + \frac{q^{N - j} - 1}{q-1}q^{j+1}[k-j]_q \nonumber\\
        &= q^{j+1}[k - j]_q[N - j]_q -[j]_q. \nonumber
    \end{align}
    These strictly decrease with $j$. Thus, the complement has degree $v - 1 - q[k]_q[N]_q$.

    Since the adjacency matrix of $\overline{J_q(n,k)}$ is $J - I - A_1$, where $J = vE_0$, its eigenvalue on $E_j\mathbb{C}^X$ is $v\delta_{j,0} - 1 - P_{j,1}$. Hence, the least eigenvalue of $\overline{J_q(n,k)}$ is attained only on $E_1\mathbb{C}^X$ and equals $-q^2[k-1]_q[N-1]_q$.
    Therefore, the Hoffman bound of $\overline{J_q(n,k)}$ is
    \[
    \frac{vq^2[k- 1]_q[N-1]_q}{v-1-q[k]_q[N]_q + q^2[k-1]_q[N-1]_q}
    = \frac{vq^2[k- 1]_q[N-1]_q}{v - [n]_q}.
    \]
    For $k = 2$, this equals $[N + 1]_q$, and the graph is $\KG_q(n,2)$, meaning that the representation in~\Cref{prop:grassmann-complement} is the representation from~\Cref{thm:qKneser}, which is Hoffman-rigid. For $k \geq 3$, the Hoffman bound is strictly larger than $[N + 1]_q$. Thus, the Hoffman bound is not tight, so $\overline{J_q(n,k)}$ admits no Hoffman-rigid representation by~\Cref{rem:hoffman-tightness}.
\end{proof}

\subsection{The Hamming scheme}
\label{sec:hamming-scheme}
For $d\geq1$ and $q\geq2$, let $X\coloneqq [q]^d$. For $x,y\in X$, define their Hamming distance 
\[ h(x,y) = |\{i\in[d] \mid x_i\ne y_i\}|. \]

The Hamming scheme on $X$, denoted by $\mathfrak X$, has relation classes indexed by Hamming
distance. The distance-one graph is the Hamming graph $H(d,q)$, in which two words are adjacent
when they differ in exactly one coordinate.
We consider its complement $G_{d,q}\coloneqq \overline{H(d,q)}$,  where two words
are adjacent precisely when their Hamming distance is at least two. Since $G_{d,q}$ has edge set $\bigcup _{i=2}^d R_i$, it arises from the Hamming scheme.

\begin{remark}
    In the maximum-distance graph $X_d(q)$ of the Hamming scheme, two words $x,y\in X$ are adjacent if, and only if, $x_i\neq y_i$ for all $i\in[d]$. Thus, $X_d(q)\cong (K_q)^d$. This graph is homomorphically equivalent to $K_q$ (see~\cite[10.9.1] {godsil2016erdos}), which  for $q\geq 3$ has a commutativity gadget and is $\RE$-complete by~\cite{culf2025existence}.%
\end{remark}

\subsubsection{Complements of Hamming graphs}
\label{subsec:hamming-complement}

\begin{figure}[t]
  \centering
 \includegraphics[width=.75\linewidth]{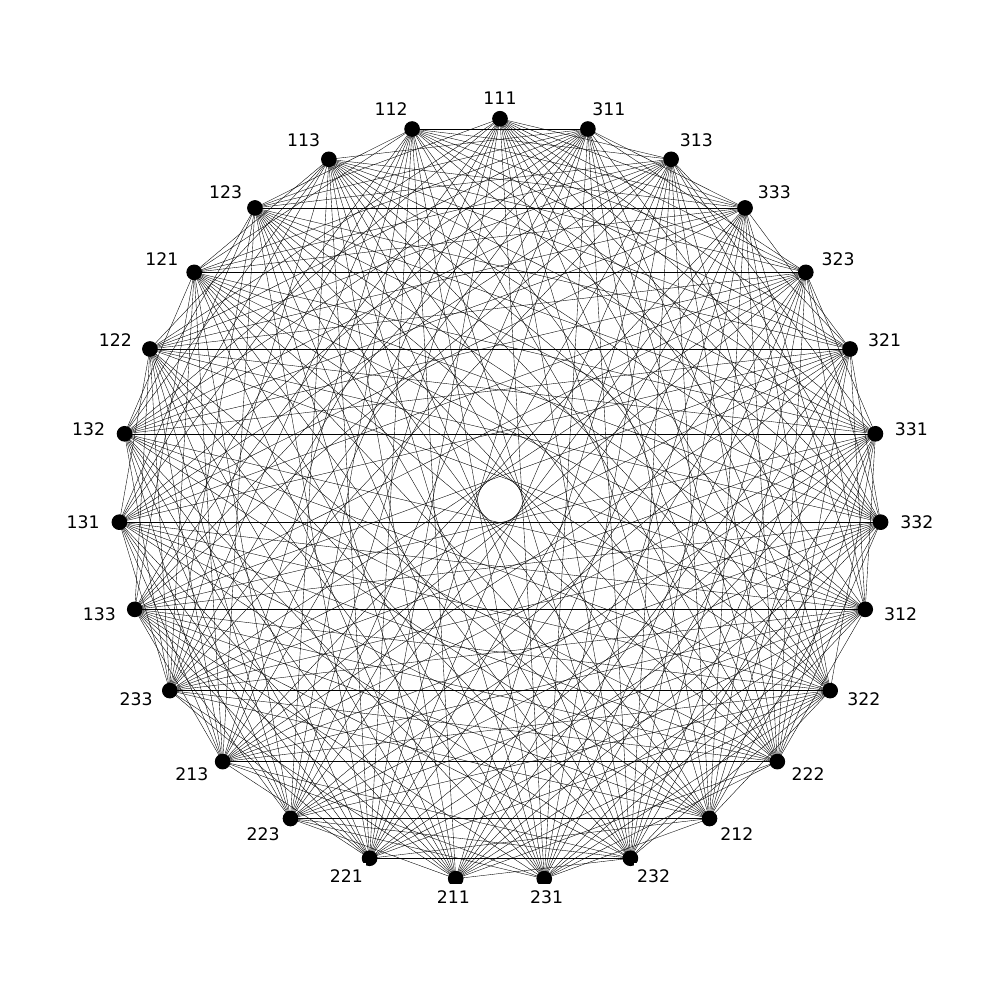}
  \caption{The complement $\overline{H(3,3)}$ of the Hamming graph $H(3,3)$.  The label $x_1x_2x_3$ denotes the word $(x_1,x_2,x_3)\in\{1,2,3\}^3$, and two distinct words are adjacent precisely when their Hamming distance is at least $2$.  The graph has $27$ vertices and $270$ edges.}
  \label{fig:hamming-complement-h3-3}
\end{figure}

Write
\[
 \C^{[q]}=\C\1\oplus U,
 \qquad U=\{u \mid \langle u,\1\rangle=0\},
\]
and let $V_j=E_j\mathbb C^{X}$ be the sum of the tensor subspaces
having a factor $U$ in exactly $j$ coordinates.

\begin{theorem}
\label{thm:hamming-complement}
For every $d\geq2$ and $q\geq3$, the graph $G_{d,q}$ has a
tame Delsarte-rigid  representation.  Consequently, the graph admits a
commutativity gadget and the corresponding $\qCSP$ is $\RE$-complete. 
\end{theorem}

\begin{proof}
Since $d\geq2$ and $q\geq 3$, the graph $(K_q)^d$ is connected and non-bipartite.
It is a
spanning subgraph of $G_{d,q}$, hence $G_{d,q}$ is connected and
non-bipartite.

For $x\in X$, define
\[x_{-i}=(x_j)_{j\in[d]\setminus\{i\}}, \qquad T_x=\{(i,x_{-i}) \mid i\in[d]\}. \]

The map $x\mapsto T_x$ is injective since $d\geq 2$.  Adjacent words agree in at most
$d-2$ coordinates, and hence have no common features.  Thus,
$x\mapsto T_x$ defines a disjointness representation with feature set
\[P=\{(i,a) \mid i\in[d],\ a\in[q]^{[d]\setminus\{i\}}\}
\]
and with
\begin{equation}
    \label{eqn_1411_0709}
 |T_x|=d,\qquad |P|=dq^{d-1},\qquad \mu=q^{-(d-1)}.                                          
\end{equation}
For $i\in[d]$ and $a\in[q]^{[d]\setminus\{i\}}$, the feature fibre and feature function of $T$ are
\[
 S_{i,a}=\{x\in[q]^d \mid x_{-i}=a\},
 \qquad
 \chi_{i,a}(x)=\indfunction{x\in S_{i,a}}.
\]
For a fixed $i\in[d]$, we have
\[ \Span\{  \chi_{i,a} \mid a\in [q]^{[d]\setminus\{i\}} \} = 
\C^{[q]} \otimes\dots \C^{[q]} \otimes \C\1\otimes \C^{[q]}\otimes\dots\otimes \C^{[q]}, \]
where the factor $\C\1$ is in the $i$-th position.
  Taking all $i$ gives
\[
                  \incspace=V_0\oplus\cdots\oplus V_{d-1}. 
\]
Since $V_0=\C\mathbf 1$, the representation is scheme-compatible with index set $\{1, \dots, d-1 \}.$

In order to show Delsarte-rigidity using~\Cref{prop:bose-mesner-spectral-test}, we construct a suitable matrix $B\in\C[\mathfrak X]$. 
Recall that $E_i$ is the projector onto $V_i$, and put
\[ B \coloneqq \kappa E_0 - \sum_{j=1}^{d-1}E_j + \theta E_d,  \]
where $\kappa$ and $\theta$ are to be determined in such a way that $\kappa >1, -1<\theta<\kappa,$ and 
\begin{equation} \label{eqn_0935_1609}
    B\in \Span\{A_i\mid i\in \{2,\dots, d \} \}.\end{equation}
Using $\sum_{i=0}^d E_i = I$, we can rewrite $B$ as
\[ B= (\kappa +1)E_0 + (\theta + 1) E_d - I. \]
Since $B\in\C[\mathfrak X]$ by definition,~\cref{eqn_0935_1609} is satisfied if $B$ vanishes on $R_0$ and $R_1$.  
If $x,y$ have Hamming distance $s$, then
\[
 (E_d)_{x,y}=q^{-d}(-1)^s(q-1)^{d-s}.
\]
Furthermore, $E_0=q^{-d}J$. Evaluating $B_{x,x}$ and $B_{x,y}$ for $x,y$ with $h(x,y)=1$ results in the constraints 
\begin{align*}
    0&=B_{x,x}=(\kappa+1)q^{-d} + (\theta+1)q^{-d}(q-1)^d-1, \\
    0&=B_{x,y}=(\kappa+1)q^{-d} - (\theta+1)q^{-d}(q-1)^{d-1}.
\end{align*}
Solving for $(\kappa+1)$ and $(\theta+1)$ yields
\[ B= q^{d-1}E_0+  \left(\frac q{q-1}\right)^{d-1}E_d-I,\]
which has eigenvalues $q^{d-1}-1, -1$ and $\left(\frac q{q-1}\right)^{d-1}-1$ on the spaces $V_0, V_1\oplus\cdots\oplus V_{d-1}$ and $V_d$, respectively.
Since $q\geq 3$ and $d\geq 2$, we have $q^{d-1} - 1>1$ and $q^{d-1} - 1 > \left(\frac q{q-1}\right)^{d-1} -1 > -1$.
Thus, $B$ satisfies the conditions of ~\Cref{prop:bose-mesner-spectral-test} with $\Lambda=\{1,\ldots,d-1\}$,
which proves
that the representation is Delsarte-rigid.

We next determine the edge kernel.
We claim that
\begin{equation}
\label{eqn_1506_0409}
 \{f\in\incspace \mid \operatorname{supp}(f)
                   \subseteq X\setminus N_{G_{d,q}}(y)\}
 =\Span\{\chi_{i,y_{-i}} \mid i\in[d]\}.                        
\end{equation}
The inclusion from right to left is immediate. For the reverse inclusion, fix $y\in[q]^d$ and let
$f\in\incspace$ vanish outside $X\setminus N_{G_{d,q}}(y)$.  For
arbitrary $a_i\ne y_i$, the identity
\[
 \left\langle f,
   \bigotimes_{i=1}^d
   \bigl(\indfunction{\{y_i\}}-\indfunction{\{a_i\}}\bigr)
 \right\rangle=0
\]
follows from $\incspace=V_0\oplus\cdots\oplus V_{d-1}=V_d^\perp$. 
Expanding it and using that $f(x)=0$ whenever $h(y,x)\geq 2$ yields
\begin{equation}
    \label{eqn_1508_0409}
 f(y)=\sum_{i=1}^d
 f(y_1,\ldots,y_{i-1},a_i,y_{i+1},\ldots,y_d).             
\end{equation}
Varying one $a_i$ at a time shows that $f$ is constant on each
set of the form $S_{i, y_{-i}}\setminus\{y\}$, and~\cref{eqn_1508_0409} states that
its value at $y$ is the sum of those $d$ constants.  This proves~\cref{eqn_1506_0409}.

Every feature fibre has $q\geq3$ vertices, and two distinct words in
the feature fibre %
share exactly that feature.  Hence
\Cref{lem:local-incidence-edge-kernel}, applied to~\cref{eqn_1506_0409}, gives
\[
 \edgekernel
 =\Span\{\chi_{i,a}\otimes\chi_{i,a} \mid (i,a)\in P\}.         
\]
In particular, every element of $\edgekernel$ is symmetric,
establishing~(\ref{tag:R1}).%

It remains to prove~\textup{(\ref{tag:R2})}.  Expand $F$ as
\[
 F=\sum_{(i,a)\in P}\gamma_{i,a}\,
       \chi_{i,a}\otimes\chi_{i,a}\in\diagonalSubspace.
\]
Recall from~\cref{eqn_18_08_1510} that $\Delta_z\colon \diagonalSubspace\rightarrow\C^{N(z)\times N(z)}$ is given by $\Delta_z(F)_{a,b}=F(a,a)-F(a,b)$. Suppose that
$\Delta_z(F)=0$, and let $\lambda$ be the common
diagonal value of $F$.  There exist $u,v\in N(z)$ with
$h(u,v)=2$: for $i=1,2$ choose $u_i\neq v_i$ distinct from $z_i$,
and leave all other coordinates equal to those of $z$.
Since $u\sim v$,
\[
       0=F(u,v)=F(u,u)=\lambda.                            
\]
Now take a feature fibre $S_{i,a}$ not containing $z$. Choose two distinct points $x,y\in S_{i,a}$ such that $x_i\neq z_i\neq y_i$. 
This is possible since $q\geq 3$. Because $z\notin S_{i,a}$, there is at least one other coordinate $j\neq i$ where $x$ and $y$ differ from $z$. 
We thus obtain distinct
$x,y\in S_{i,a}\cap N(z)$, and they share exactly the feature
$(i,a)$.  Therefore, evaluating $F$ at $(x,y)$ gives
\[
             \gamma_{i,a}=F(x,y)=F(x,x)=\lambda=0.         
\]
Only the $d$ feature fibres containing $z$ can remain.  Changing only coordinate
$i$ of $z$ produces a word that is contained in exactly the $i$-th one
of those fibres.  Evaluating its zero diagonal therefore kills the
 coefficient of the corresponding feature function.  Thus $F=0$, and $\Delta_z$ is
injective.

By~\Cref{lem:delsarte-iff-schrijver}, the Delsarte-rigid representation is
Schrijver-rigid, so~\Cref{thm:criterion}  proves  that the graph admits a commutativity gadget.
\end{proof}

\begin{remark}
In~\Cref{thm:hamming-complement},
we cannot use~\Cref{lem:diagonal-feature-rigidity} to prove tameness of the representation since the feature functions are linearly dependent. Indeed, 
\[ \sum_{ a\in [q]^{[d]\setminus \{1\}}} \chi_{1, a} = \1 =    \sum_{ a\in [q]^{[d]\setminus \{2\}}} \chi_{2, a} .\]
\end{remark}

\begin{proposition}
If $d\geq3$, then  $G_{d,q}$ has no Hoffman-rigid
representation. For $d=2$, the disjointness representation in~\Cref{thm:hamming-complement}
 is Hoffman-rigid. 
\end{proposition}
\begin{proof}  The degree of $G_{d,q}$
and its least adjacency eigenvalue are
\[
 D=q^d-1-d(q-1),\qquad
 \tau=-(d-1)(q-1).                                        
\]
Indeed, the distance-one Hamming graph
has eigenvalue
$d(q-1)-qj$ on $V_j$ (see \cite{godsil2016erdos}). Taking the complement of a regular graph changes every nonprincipal
eigenvalue $\theta$ to $-1-\theta$, since $A(G^c)=J-I-A(G)$.  Thus Hoffman's bound is

\begin{equation}
    \label{eqn_1413_0709}
 \alpha(G_{d,q}) \leq \frac{q^d(d-1)(q-1)}{q^d-q}.                 
\end{equation}

A feature fibre has size $q$, while the associated dual
$\vartheta^-$ certificate and
\cref{eqn_1411_0709} give $\alpha(G_{d,q})\leq\mu|X|=q$.  Hence
$\alpha(G_{d,q})=q$.
The right-hand side of~\cref{eqn_1413_0709} equals $q$ when $d=2$ and is
strictly larger than $q$ when $d\geq3$.    By
\Cref{rem:hoffman-tightness}, no Hoffman-rigid  representation exists
in the latter case. When $d=2$, then $\incspace = V_0\oplus V_1$ and $V_1$ is the least adjacency eigenspace. Therefore, the representation in \Cref{thm:hamming-complement} is Hoffman-rigid.  
\end{proof}

\addcontentsline{toc}{section}{References}
\printbibliography

\end{document}